\documentclass[11 pt]{article}
\usepackage{subfiles}
\usepackage{enumitem}

\pdfoutput=1

\usepackage[backend=bibtex,
            isbn=false,
            doi=false,
            giveninits=true,
            style=alphabetic,
            useprefix=true,
            maxcitenames=4,
            citestyle=alphabetic]{biblatex}
\renewbibmacro{in:}{}
\bibliography{Mult-Quiver.bib}

\usepackage[all,2cell]{xy} \UseAllTwocells \SilentMatrices \SelectTips{cm}{}
\usepackage{latexsym,amsfonts,amssymb}
\usepackage{amsmath,amsthm,amscd}
\usepackage[dvips]{epsfig}
\usepackage{psfrag}

\usepackage{color}
\usepackage{etoolbox}
\usepackage{hyperref}

\usepackage{graphicx}
\usepackage{fancyhdr}
\usepackage{mathtools}
\usepackage{topcapt}
\usepackage{enumitem}
\usepackage{booktabs}
\usepackage{stmaryrd}
\usepackage[scaled]{helvet}
\usepackage[T1]{fontenc}
\input xy
\xyoption{all}
\usepackage{tikz}
\usepackage{braids}

\usepackage{tikz}
\usepgflibrary{shapes.geometric}
\usetikzlibrary{matrix}
\usepackage{pgfplots}
\usepackage{braids}
\newcommand{\tik}[1]{\begin{tikzpicture}[baseline=(current bounding box.center), scale=0.65] #1 \end{tikzpicture} }

\usetikzlibrary{arrows}
\tikzstyle cross=[preaction={draw=white, -, line width=4pt}, thick]
\tikzstyle normal=[thick]
\tikzstyle chord=[densely dotted, thick]
\tikzstyle zero=[ultra thick, gray]
\tikzstyle zerocross=[preaction={draw=white, -, line width=4pt}, ultra thick, gray]
\tikzstyle point=[draw,circle,inner sep=1,fill=black]
\tikzstyle petitpoint=[draw,circle,inner sep=0.3,fill=black]

\newcommand{\negative}[3][-]{\draw[normal,#1] (#2+1,-#3).. controls (#2+1,-#3-0.3) and (#2,-#3-0.7)..(#2,-#3-1); \draw[cross,#1] (#2,-#3).. controls (#2,-#3-0.3) and (#2+1,-#3-0.7)..(#2+1,-#3-1);}

\newcommand{\fieldgoal}[3][-]{\draw[cross,#1] (#2+2.5,-#3).. controls (#2+2.5,-#3-0.3) and (#2+1,-#3-0.7)..(#2+1,-#3-1); \draw[cross,#1] (#2,-#3).. controls (#2,-#3-0.3) and (#2+2.5,-#3-0.7)..(#2+2.5,-#3-1);\draw[cross,#1] (#2+1.5,-#3).. controls (#2+1.5,-#3-0.3) and (#2,-#3-0.7)..(#2,-#3-1);}

\newcommand{\fieldgoalsm}[3][-]{\draw[cross,#1] (#2+2.5,-#3).. controls (#2+2.5,-#3-0.3) and (#2+1,-#3-0.7)..(#2+1,-#3-1); \draw[cross,#1] (#2+1,-#3).. controls (#2+1,-#3-0.3) and (#2+2.5,-#3-0.7)..(#2+2.5,-#3-1);\draw[cross,#1] (#2+1.5,-#3).. controls (#2+1.5,-#3-0.3) and (#2,-#3-0.7)..(#2,-#3-1);}

\newcommand{\lmove}[3][-]{\draw[cross,#1] (#2+1,-#3).. controls (#2+1,-#3-0.3) and (#2,-#3-0.7)..(#2,-#3-1);}
\newcommand{\lmovesm}[3][-]{\draw[cross,#1] (#2+0.5,-#3).. controls (#2+0.5,-#3-0.3) and (#2,-#3-0.7)..(#2,-#3-1);}
\newcommand{\lmovelg}[3][-]{\draw[cross,#1] (#2+2,-#3).. controls (#2+2,-#3-0.3) and (#2,-#3-0.7)..(#2,-#3-1);}
\newcommand{\rmovesm}[3][-]{\draw[cross,#1] (#2,-#3).. controls (#2,-#3-0.3) and (#2+.5,-#3-0.7)..(#2+.5,-#3-1);}

\newcommand{\rmove}[3][-]{\draw[cross,#1] (#2,-#3).. controls (#2,-#3-0.3) and (#2+1,-#3-0.7)..(#2+1,-#3-1);}
\newcommand{\coupon}[4]{\draw[cross] (#1-0.2,-#2) rectangle (#1+#3+0.2,-#2-1);\node at (#1+#3*0.5,-#2-0.5) {#4};}

\newcommand{\up}[3][-]{ 
\draw[cross,#1] (#2,-#3).. controls (#2,-#3-0.3) and (#2+0.2,-#3-0.5).. (#2+0.5,-#3-0.5);
\draw[cross] (#2+0.5,-#3-0.5).. controls (#2+0.8,-#3-0.5) and (#2+1,-#3-0.3).. (#2+1,-#3);}
\newcommand{\ap}[3][-]{
\draw[cross,#1] (#2,-#3-1).. controls (#2,-#3-0.7) and (#2+0.2,-#3-0.5).. (#2+0.5,-#3-0.5);
\draw[cross] (#2+0.5,-#3-0.5).. controls (#2+0.8,-#3-0.5) and (#2+1,-#3-0.7).. (#2+1,-#3-1);}

\newcommand{\straight}[3][-]{\draw[cross,#1] (#2,-#3) -- (#2,-#3-1);}

\usetikzlibrary{shapes,snakes}

\newcommand{\KP}[1]{%
  \begin{tikzpicture}[baseline=-\dimexpr\fontdimen22\textfont2\relax]
  #1
  \end{tikzpicture}%
}
\newcommand{\KPA}{%
  \KP{\filldraw[color=gray, fill=none, thick] circle (0.3);}%
}
\newcommand{\KPB}{%
  \KP{
    \draw[color=gray,thick] (-0.3,0.3) -- (0.3,-0.3);
    \draw[color=gray,thick] (-0.3,-0.3) -- (-0.05,-0.05);
    \draw[color=gray,thick] (0.05,0.05) -- (0.3,0.3);
  }%
}
\newcommand{\KPC}{%
  \KP{%
    \draw[color=gray,thick] (-0.3,0.3) .. controls (0,-0.05) .. (0.3,0.3);
    \draw[color=gray,thick] (-0.3,-0.3) .. controls (0,0.05) .. (0.3,-0.3);
  }%
}
\newcommand{\KPD}{%
  \KP{%
    \draw[color=gray,thick] (-0.3,-0.3) .. controls (0.05,0) .. (-0.3,0.3);
    \draw[color=gray,thick] (0.3,-0.3) .. controls (-0.05,0) .. (0.3,0.3);
  }%
}

\newcommand{\de}{\partial}

\newcommand{\ULR}{U_\Rt^{\mathrm{L}}}

\newcommand{\UK}{U_\Ct}
\newcommand{\UqL}{U_q^{\mathrm{L}}}

\newcommand{\uq}{u_q}

\newcommand{\Aq}{{A_q}}
\newcommand{\Bq}{{B_q}}
\newcommand{\OK}{\O_\Ct}
\newcommand{\Oq}{\O_q}
\newcommand{\OR}{\O_\Rt}

\newcommand{\Rep}{\mathsf{Rep}}
\newcommand{\Repq}{\mathsf{Rep}_q}
\newcommand{\RepR}{\mathsf{Rep}_\Rt}
\newcommand{\RepK}{\mathsf{Rep}_\Ct}

\newcommand{\Ann}{\mathrm{Ann}}

\newcommand{\Sp}{{S^\circ}}

\newcommand{\Mfld}{\operatorname{Mfld}}
\newcommand{\Disk}{\operatorname{Disk}}

\newcommand{\rightloop}{%
           \mathrel{\raisebox{.1em}{%
           \reflectbox{\rotatebox[origin=c]{-90}{$\circlearrowright$}}}}}

\newcommand{\Br}{G^\circ}
\newcommand{\ot}{\otimes}
\newcommand{\op}{\mathrm{op}}

\newcommand{\C}{\mathbb{C}}

\newcommand{\Q}{\mathbb{Q}}

\newcommand{\Z}{\mathbb{Z}}

\newcommand{\cA}{\mathcal{A}}

\newcommand{\cC}{\mathcal{C}}
\newcommand{\cD}{\mathcal{D}}

\newcommand{\D}{\mathcal{D}}
 
\newcommand{\cE}{\mathcal{E}}

\newcommand{\Ct}{\mathcal{K}}
\newcommand{\cM}{\mathcal{M}}
\renewcommand{\O}{\mathcal{O}}
\newcommand{\cO}{\mathcal{O}}
\newcommand{\Rt}{\mathcal{R}}
\newcommand{\cV}{\mathcal{V}}
\newcommand{\cW}{\mathcal{W}}

\newcommand{\cZ}{\mathcal{Z}}

\newcommand{\Dist}{\operatorname{Dist}}

\newcommand{\T}{\mathrm{T}}

\newcommand{\g}{\mathfrak g}
\renewcommand{\b}{\mathfrak b}

\newcommand{\gl}{\mathfrak{gl}}
\newcommand{\glN}{\mathfrak{gl}_N}

\newcommand{\bt}{\boxtimes}

\newcommand{\Tr}{\operatorname{Tr}}
\newcommand{\inv}{^{-1}}
\newcommand{\Der}{\mathrm{Der}}
\newcommand{\Hom}{\text{\rm Hom}}
\newcommand{\Mat}{\text{\rm Mat}}

\newcommand{\GL}{\text{\rm GL}}

\newcommand{\SL}{\text{\rm SL}}

\newcommand{\Lie}{\mathrm{Lie}}
\newcommand{\into}{\hookrightarrow}

\DeclareMathOperator{\Spec}{Spec}
\newcommand{\Proj}{\mathrm{Proj}}
\newcommand{\Sym}{\text{\rm Sym}}

\newcommand{\id}{\text{\rm Id}}

\newcommand{\Fun}{\text{\rm Fun}}
\newcommand{\End}{\text{\rm End}}

\newcommand{\Fr}{\text{\rm Fr}}
\newcommand{\dS}{\Big{/}\!\!\Big{/}}
\newcommand{\sS}{\Big{/}}

\newcommand{\MQdf}{\cM_{fr}(Q,\dv)}
\newcommand{\MatQd}{\Mat(Q, \mathbf{d})}

\newcommand{\DqMatMN}{\cD_q\left(\overset{N}{\bullet}\rightarrow\overset{M}{\bullet}\right)}
\newcommand{\OqMatMN}{\O_q\left(\overset{N}{\bullet}\rightarrow\overset{M}{\bullet}\right)}
\newcommand{\OqMatNM}{\O_q\left(\overset{M}{\bullet}\rightarrow\overset{N}{\bullet}\right)}
\newcommand{\Dqloop}{\cD_q\left(\overset{N}{\bullet}\rightloop\right)}
\newcommand{\DqMate}{\mathcal{D}_q(\Mat(e))}
\newcommand{\DqQ}{\mathcal{D}_q(\Mat(Q,\dv))}

\newcommand{\RepqGLd}{\Repq(\GL_{\dv})}
\newcommand{\MQd}{\cM(Q,\dv,\xi,\theta)}
\newcommand{\MQdtriv}{\cM(Q,\dv,\xi,0)}
\newcommand{\MQds}{\cM^s(Q,\dv,\xi,\theta)}

\newcommand{\DqMatQd}{\mathcal{D}_q(\Mat(Q, \mathbf{d}))}
\newcommand{\DRMatQd}{\mathcal{D}_\Rt(\Mat(Q, \mathbf{d}))}

\newcommand{\Kr}{_\text{\rm Kr}}

\newcommand{\Vect}{\mathsf{Vect}}

\newcommand{\detq}{\text{\rm det}_q}

\newcommand{\ad}{\text{\rm ad}} 

\newcommand{\dv}{\mathbf{d}}

\theoremstyle{plain}
\newtheorem{theorem}{Theorem}[section]
\newtheorem{prop}[theorem]{Proposition}

\newtheorem{lemma}[theorem]{Lemma}
\newtheorem{cor}[theorem]{Corollary}

\newtheorem{assumption}[theorem]{Assumption}

\theoremstyle{definition}
\newtheorem{definition}[theorem]{Definition}
\newtheorem{example}[theorem]{Example}
\newtheorem{rmk}[theorem]{Remark}

\numberwithin{equation}{section}

\title{The quantum Frobenius for character varieties and multiplicative quiver varieties}
\author{Iordan Ganev \and David Jordan \and Pavel Safronov}
\begin{document} 
\maketitle
\begin{abstract}
We prove that quantized multiplicative quiver varieties and quantum character varieties define sheaves of Azumaya algebras over the corresponding classical moduli spaces, and we prove that the Azumaya locus of the Kauffman bracket skein algebras contains the smooth locus, proving a strong form of the Unicity Conjecture of Bonahon and Wong.

The proofs exploit a strong compatibility between quantum Hamiltonian reduction and the quantum Frobenius homomorphism as it arises in each setting.  We therefore introduce the concepts of Frobenius quantum moment maps and their Hamiltonian reduction, and of Frobenius Poisson orders. We use these tools to construct canonical central subalgebras of quantum algebras, and explicitly compute the resulting Azumaya loci we encounter, using a natural nondegeneracy assumption.
\end{abstract}

\setcounter{tocdepth}{2}
\tableofcontents

\section{Introduction}\label{sec:intro}

In this paper we study \emph{quantized character varieties}, and \emph{quantized multiplicative quiver varieties}, when the quantization parameter is a root of unity.  Our main results describe these quantized moduli spaces as Azumaya algebras -- meaning that, {\' e}tale-locally, they are sheaves of matrix algebras -- over an explicitly given open locus on the spectrum of their centers.  In the quiver examples, this spectrum is the classical multiplicative quiver variety, while for character variety examples it is the classical character variety.

As a key application, we also treat the so-called Kauffman bracket skein algebra of a surface.  The `Unicity Conjecture' of Bonahon-Wong asserts that the Kauffman bracket skein algebra is Azumaya over some non-empty open subset\cite{BonahonWongI,BonahonWongII,BonahonWongIII}; in this form it was first proved by Frohman, Kania-Bartoszynska, and L\^{e} \cite{FrohmanUnicity}, who showed that the set of Azumaya points was non-empty and open but did not determine the locus precisely.  We improve upon their result by showing that the Azumaya locus in fact contains the smooth locus.

Both classes of quantizations are constructed via the process of quantum Hamiltonian reduction, and in both cases the theory of Poisson orders equips the quantization with a connection over an open symplectic leaf.  Our primary techniques, therefore, exploit several remarkable compatibilities between Lusztig's quantum Frobenius homomorphism on the quantum group, the theory of Poisson orders, and the procedure of quantum Hamiltonian reduction.  To this end, we develop the frameworks of Frobenius Poisson orders and their Frobenius quantum Hamiltonian reduction. We exploit the Frobenius Poisson order to determine the Azumaya points before reduction, and then we show that Frobenius Poisson orders descend to Poisson orders on the Hamiltonian reduction, and hence conclude that the Azumaya algebras we constructed before reduction descend to Azumaya algebras on the Hamiltonian reduction.  A further remarkable feature of both classes of examples is that they are \emph{non-degenerate} Hamiltonian $G$-Poisson spaces, which allows us to describe the symplectic leaves very explicitly in terms of the classical multiplicative moment map.

We first recall our two motivating classes of examples, then state our main results concerning them, and then discuss the general results and techniques in more detail.

\subsection{Character varieties and their quantization}

Given a reductive group $G$, the $G$-character stack $Ch_G(S)$ of a connected topological surface $S$ (possibly with boundary) is the moduli space of $G$-local systems on $S$, equivalently of representations $\pi_1(S)\to G$, modulo the conjugation action.  To construct $Ch_G(S)$, one considers first the \emph{framed} character variety, $Ch_G^{fr}(S)$, consisting of a $G$-local system with a fixed trivialization at one point; this is an affine variety, equipped with a $G$-action by changes of framing. The framed character variety carries the Fock--Rosly Poisson structure, and admits a multiplicative moment map $Ch_G^{fr}(S) \rightarrow G$.  The character variety $Ch_G(S)$ is then the quotient of $Ch_G^{fr}(S)$ by the $G$-action.

Given now a \emph{closed} surface $S$ of genus $g$, let us denote by $\Sp$ the surface obtained by removing some small disk.  An important observation is that the character variety of $S$ is obtained from the character variety of $\Sp$ via a procedure of group-valued Hamiltonian reduction \cite{AlekseevLiegroupvalued1998}, \cite{AlekseevQuasiPoisson}, \cite{SafronovQuasiHamiltonian}.  This means that $Ch^{fr}_G(\Sp)$ carries a multiplicative moment map valued in the group $G$, which records the holonomy around the boundary.  The multiplicative Hamiltonian reduction then computes the joint effect of attaching the disk and quotienting the $G$-action.  The character variety obtained in this way is complicated -- it is singular in general, and may have several irreducible components.  There is an important open subset of the character variety called the `good locus'; the good locus consists of the closed orbits whose stabilizer is the center.  It is empty in genus one, but non-empty for genus greater than one (and in fact dense for $G=\GL_N, \SL_N$).

Functorial quantizations of character varieties were introduced in \cite{Ben-ZviIntegratingquantumgroups2018}; it was also proved there that the framed quantizations of punctured surfaces could be described algebraically via (mild generalizations of) certain ``moduli'' algebras $A_S$ defined combinatorially by Alekseev, Grosse, and Schomerus \cite{AlekseevCombinatorialquantizationHamiltonian1996}.  In the case of closed surfaces, the resulting quantized character varieties were shown in \cite{Ben-ZviQuantumcharactervarieties2018} to admit a description via quantum Hamiltonian reduction of the algebras $A_S$, echoing the classical construction.  In particular, there was introduced in \cite{Ben-ZviIntegratingquantumgroups2018} a ``distinguished object'' -- a non-commutative stand-in for the structure sheaf -- whose endomorphism algebra gives a quantization of the affine character variety, and which is computed via quantum Hamiltonian reduction.

In the case $G=\SL_2$, the Kauffman bracket skein algebra provides another celebrated quantization of the character variety.  The Kauffman bracket skein algebra $K_A(S)$ with parameter $A\in\C^\times$ is the vector space spanned by isotopy classes of links $\langle L \rangle$ drawn in the cylinder $S\times I$ over the surface, modulo the relations,
$$\left\langle\KPA\right\rangle=1,\qquad  \left\langle L \cup \KPA\right\rangle=(-A^{2}-A^{-2})\langle L\rangle$$
$$\left\langle\KPB\right\rangle=
  A\left\langle\KPC\right\rangle + A^{-1} \left\langle \KPD \right\rangle$$
where the diagrams represent links which are as depicted in some oriented 3-ball, and identical outside of it.  The algebra structure on $K_A(S)$ is obtained by vertically stacking links in $S$.  It was shown by Turaev \cite{TuraevSkein} that the Kauffman bracket skein algebra provides a deformation quantization of the $\SL_2(\C)$-character variety of $S$. 

\subsection{Multiplicative quiver varieties and their quantization}
Let $Q$ be a quiver with dimension vector $\dv$. The  multiplicative quiver variety is a moduli space of representations of a doubled quiver satisfying certain moment map relations, first introduced in \cite{Crawley-BoeveyMultiplicativepreprojectivealgebras2006}.  It is constructed by first recalling that the collection of representations of the doubled quiver of $Q$ with  dimension vector $\dv$ forms a product of matrix spaces. The framed representation variety $\MQdf$ is an open locus of this product of matrix spaces, defined by the non-vanishing of certain determinants, and admits a multiplicative moment map to the gauge group  $\mathbb{G}^\dv = \prod \GL_{d_v}$, where the product runs over the set of vertices of $Q$.  The multiplicative quiver variety $\MQd$  is the multiplicative GIT Hamiltonian reduction of the framed representation variety $\MQdf$ by the gauge group $\mathbb{G}^\dv$ at a moment parameter $\xi$, and with stability parameter $\theta$  (see  \cite{VandenBerghDoublePoissonalgebras2008, YamakawaGeometryMultiplicativePreprojective2008}):
$$\MQd = \MQdf \dS_{\!\!\!\xi,\theta} \mathbb{G}^\dv.$$
According to \cite{Crawley-BoeveyMultiplicativepreprojectivealgebras2006}, certain special cases of multiplicative quiver varieties yield moduli spaces of $GL_N$-connections with irregular singularities or, equivalently, moduli spaces of representations of $\pi_1(S)$, with prescribed monodromies around the punctures.

Quantizations of multiplicative quiver varieties were introduced in \cite{JordanQuantizedmultiplicativequiver2014}. The construction involved first quantizing $\MQdf$ via an algebra $\D_q(\MatQd)$, then quantizing the moment map, and defining $\MQd$ as a quantum Hamiltonian reduction.  For a more thorough recollection about multiplicative quiver varieties and their quantization, see Section~\ref{sec:DqMat}. The algebras $\D_q(\MatQd)$ give deformations of $\MQdf$, which are flat over the ring $\C[q,q^{-1}]$.  In \cite{JordanQuantizedmultiplicativequiver2014} it was shown that the quantum Hamiltonian reductions of $\D_q(\MatQd)$ are formally flat (i.e., they are flat when tensored over the ring $\C\llbracket\hbar\rrbracket$, where $\hbar=\log q$) whenever  the classical multiplicative moment map is flat; by \cite{Crawley-BoeveyMultiplicativepreprojectivealgebras2006} this can be read off from the dimension vector and the moment map parameters. However, flatness of the quantum Hamiltonian reduction over $\C[q,q^{-1}]$ remains unsettled.

\subsection{Main results: examples and applications} Our main results in the context of the preceding examples are as follows:

\begin{theorem}
Let $G$ be a connected reductive group and $q$ a primitive $\ell$-th root of unity, which together satisfy Assumption \ref{EllAssumption}. Let $S$ be a closed topological surface of genus $g$, and let us denote by $\Sp$ the surface obtained by removing some open disk from $S$. Then:
\begin{enumerate}
\item The moduli algebra $A_\Sp$ is finitely generated over its center, which is isomorphic to the coordinate ring of the classical framed $G$-character variety $Ch_G^{fr}(\Sp)\cong G^{2g}$.
\item Moreover, the Azumaya locus of the moduli algebra $A_\Sp$ coincides with the preimage of open cell $\Br\subset G$ under the classical moment map $Ch_G^{fr}(\Sp) \to G$.
\item The quantized character variety of the closed surface $S$ is finitely generated over its center, which is isomorphic to the coordinate ring of the classical character variety.  It may be constructed as a Frobenius quantum Hamiltonian reduction of $A_\Sp$.
\item The quantized character variety of the closed surface $S$ is Azumaya over the entire `good locus' of $Ch_G(S)$.
\end{enumerate}\label{ChG-thm-intro}
\end{theorem}

We remark that the proofs of Statements 3 and 4 apply identically to the \emph{twisted character variety} of $S$, where we take $G=\GL_N$, and we require that that the holonomy around the boundary of $\Sp$ is a primitive $N$th root of unity \cite{Hausel}.  In this case all points are `good' in the above sense and the quotient is smooth; we obtain in this way an Azumaya algebra defined over the entire twisted character variety.

Our techniques apply, in particular, to skein algebras such as the Kauffman bracket skein algebra.  In a series of papers \cite{BonahonWongI,BonahonWongII,BonahonWongIII} it was proved that the Kauffman bracket skein algebra at roots of unity has a central subalgebra isomorphic to the functions on the classical $\SL_2(\C)$-character variety, and in \cite{FrohmanUnicity} it was proved that this is in fact the whole center and moreover that the Azumaya locus is open and dense.  We prove the following theorem in Section \ref{Kauffman-section}:

\begin{theorem}\label{thm-Kauffman-intro}
Suppose that $\ell>2$ is an odd integer, and that $\zeta$ is a primitive $\ell$th root of unity. Then the skein algebra $K_\zeta(S)$ is Azumaya over the whole smooth locus of $Ch_{\SL_2}(S)$.
\end{theorem}

\begin{rmk}
While Theorem \ref{thm-Kauffman-intro} fits naturally in the broader framework we develop in the paper, those readers who are only interested in the proof of Theorem \ref{thm-Kauffman-intro} (i.e. in the precise determination of the non-singular locus), and who are already familiar with \cite{FrohmanUnicity} may wish to skip directly to Section \ref{Kauffman-section}, referring back to Sections \ref{sec:Poisson} and \ref{sec:frobenius} only as needed.
\end{rmk}

It is remarkable that the proof of the corresponding `Unicity theorem' for quantum character varieties proved in Theorem \ref{ChG-thm-intro} is uniform in all groups, and follows from functoriality and generalities about quantum Hamiltonian reduction, in contrast to the hands on algebraic methods in the skein literature.  Precise determination of the Azumaya locus, as well as an extension of the unicity theorem to more general groups in this way was a major motivation for this work.

Turning now to the quiver examples, we have:

\begin{theorem} Let $\ell>1$ be an odd integer, and  $q$  a primitive $\ell$-th root of unity. Then:
\begin{enumerate}
\item The algebra $\D_q(\MatQd)$ is finitely generated over a central subalgebra, which is isomorphic to the coordinate ring $\O(\MQdf)$ of the classical framed multiplicative quiver variety.
\item Moreover, $\D_q(\MatQd)$ is Azumaya over the preimage in $\MQdf$ of the big cell $\Br\subset G$ under the multiplicative moment map $\MQdf\to G$. 
\item Frobenius quantum Hamiltonian reduction defines a coherent sheaf of algebras over the classical multiplicative quiver variety $\MQd$, which is Azumaya over the locus $\MQds$ of $\theta$-stable representations.
\end{enumerate}\label{quiver-thm-intro}
\end{theorem}

In particular, these theorems identify the center of the algebra $\D_q(\MatQd)$ (resp.\ $A_{\Sp}$) with functions on $\MQdf$ (resp.\ $\overline{Ch_G(S)}$), and likewise identify the center of the quantized multiplicative quiver variety (resp.\ quantum character variety) with the affinization of $\MQd$ (resp.\ of the classical character variety).  It is already difficult to compute such centers directly, and our results imply, by the flatness in $q$, that the center is trivial away from roots of unity, a fact which is again not easy to see directly.

The Azumaya property asserts moreover that each sheaf is \'etale-locally the endomorphism algebra of a vector bundle on the classical variety, or equivalently, that the fiber of the algebra at each point is a matrix algebra over $\C$. In each statement, the second assertion is derived from the first via the process of Frobenius quantum Hamiltonian reduction, which is described in the next subsection.

\begin{rmk}
Modules over an Azumaya algebra form an invertible sheaf of categories which is locally trivial for the \'{e}tale topology, i.e., a gerbe. It would be interesting to understand this gerbe more conceptually. In a related direction, we expect that modules over the quantum group at a root of unity form a factorizable category relative to its M\"{u}ger center. Using the results of Gwilliam--Scheimbauer \cite{GwilliamScheimbauerMorita}, one may prove that this implies that this category defines an invertible object in the Morita category of braided monoidal categories relative to the classical representation category of the group, i.e., it might be interpreted as a ``higher Azumaya algebra''. In particular, this will formally imply that its factorization homology over a topological surface $S$ forms an invertible sheaf of categories over the classical character variety $Ch_G(S)$.
\end{rmk}

\subsection{Main results: methods and general results}\label{subsec:methods}
The proofs of our main results are rooted in a collection of  beautiful ideas emerging from the literature on quantum groups and geometric representation theory, most notably the seminal paper \cite{BezrukavnikovCherednikalgebrasHilbert2006} where the Hamiltonian reduction of Azumaya algebras in characteristic $p$ was first carried out for differential operators with applications to Cherednik algebras, and \cite{VaragnoloDoubleaffineHecke2010}, where a $q$-analog was developed to study $q$-difference operators and double affine Hecke algebras at roots of unity.  Similar techniques were used in the study of hypertoric varieties in positive characteristic \cite{STADNIK2013}, and of their $q$-analogs, quantum multiplicative hypertoric varieties \cite{GanevQuantizationsmultiplicativehypertoric2018a}.

A significant technical portion of the paper develops a framework generalizing these examples to our setting.  Let us therefore review some of the ingredients here, many of which are well-known, and some which are new.

\paragraph{Integral forms of the quantum group}

Let $G$ be a connected reductive group, with $\g$ its Lie algebra, $U = U(\g)$ is universal enveloping algebra. For the remainder of the paper we will reserve the letter $q$ to denote a complex root of unity, and $\ell$ to denote its order. We will assume $G$ and $\ell$ satisfy a number of mild assumptions (see Assumption \ref{EllAssumption}).

We shall require several related forms of the quantum group associated to $\g$.  Our basic reference is \cite{ChariGuideQuantumGroups1995}. For us, the quantum group at generic parameters refers to the Drinfeld--Jimbo rational form of the quantum group, defined over the base ring $\Ct = \C(t)$ of rational functions in a variable $t$ with coefficients in $\C$; we will denote this rational form of the quantum group by $\UK$.  It is generated by the quantum Cartan subalgebra, isomorphic to the group algebra of the coweight lattice of $G$, and by the Serre generators $E_\alpha$ and $F_\alpha$ for each positive simple root $\alpha$. We do not recall the relations in detail for general $\g$ (see instead \cite[Chapter 9]{ChariGuideQuantumGroups1995}), because we will only use some essential functorial properties, which we detail later in this section.

In addition to the rational form of the quantum group, we consider the so-called divided powers integral from of the quantum group, introduced by Lusztig \cite{LusztigQuantumdeformationscertain1988}, and defined as follows. Let $\Rt = \C[t^{\pm 1}]$ denote the subring of $\Ct$ consisting of Laurent polynomials in $t$, and consider the $\Rt$-subalgebra $\ULR$ of $\UK$ generated by  the quantum Cartan generators and an integer family of \emph{divided powers} $E_\alpha^{(r)}=\frac{E_\alpha^r}{[r]!}$ for each Serre generator.  Here $[r]$ denotes the quantum integer, and $[r]!$ the quantum factorial, in the variable $t$.  We reserve the notation $\UqL$ for the base-changed algebra,
$$\UqL = \ULR \ot_\Rt \C,$$
via the algebra homomorphism from $\Rt$ to $\C$ given by $t\mapsto q$, our chosen root of unity.  Finally, we denote by $\uq$ the small quantum group, which we regard as a subalgebra of $\UqL$ generated by the ``undivided powers'' $E_\alpha^{(1)}$, $F_{\alpha}^{(1)}$, together with the quantum Cartan generators. 

\begin{definition}\label{def:repqG} Let $\RepK(G)$, $\RepR(G)$, $\Rep_q(G)$, $\Rep(\uq(\g))$, and $\Rep(G)$ denote the categories of locally finite-dimensional modules, respectively, of the rational form $\UK$, Lusztig's integral form $\ULR$, its specialization $\UqL$, the small quantum group $\uq$, and the classical enveloping algebra $U$, such that the weights of the Cartan subalgebra lie in the weight lattice of $G$.\end{definition}

\paragraph{The quantum Frobenius}

An important feature of Lusztig's integral form is that it admits a homomorphism of Hopf algebras,
$$\text{\rm Fr} : \UqL \rightarrow U,$$
uniquely defined so that for all simple roots $\alpha$, we have:
$$\Fr(E_\alpha^{(n)}) = \begin{cases} E_\alpha^{(n/\ell)},& \ell \mid n\\ 0,& \ell \nmid n\end{cases}, \qquad \Fr(F_\alpha^{(n)}) = \begin{cases} F_\alpha^{(n/\ell)},& \ell \mid n\\ 0,& \ell \nmid n\end{cases}.$$
The ``quantum Frobenius'' map $\Fr$ is a surjective homomorphism of quasi-triangular Hopf algebras, whose kernel is the two-sided ideal generated by the augmentation ideal $\ker(\epsilon)$ of the small quantum group $\uq(\g)$. Thus, we have Lusztig's resulting ``short exact sequence'' of Hopf algebras:
\begin{equation}\uq \to \UqL \xrightarrow{\Fr} U\end{equation}
Basic references for the quantum Frobenius include \cite{ChariGuideQuantumGroups1995}, \cite{LusztigQuantumgroupsroots1990}, \cite{LusztigIntroductionQuantumGroups2010}, \cite{lentner_frobenius_2015},\cite{lentner_factorizable_2017}, \cite{arkhipov_another_2003}, \cite{kremnizer_proof_2006}, \cite{NegronLogModular}.

This setup gives rise to a remarkable adjoint pair of braided tensor functors,
$$\Fr^*:\Rep(G) \to \Repq(G), \qquad \Fr_*: \Repq(G)\to \Rep(G),$$
the pull-back via $\Fr$, and the passage to $\uq$-invariants, respectively. We note that $\Fr_* \circ \Fr^*$ is the identity functor on $\Rep(G)$; in particular, the functors $(\Fr^*, \Fr_*)$ form an adjoint pair.  The functor $\Fr^*$ is braided monoidal, and maps into the M\"uger center of $\Repq(G)$. That is, for any object $V$ in $\Rep(G)$ and any object $W$ in $\Rep_q(G)$, the two  braidings,
$$\sigma_{\Fr^*(V),W} \textrm{ and } \sigma_{W,\Fr^*(V)}^{-1}: \Fr^*(V) \ot W \rightarrow W \ot \Fr^*(V),$$
coincide.  In fact, each is the switch-of-factors map on the underlying vector space.  

\paragraph{Frobenius Poisson orders}

The notion of a Poisson order, introduced in \cite{BrownPoissonorderssymplectic2003a}, consists of a non-commutative algebra $A$, a central subalgebra $Z$ of $A$, a Poisson bracket on $Z$, and a linear map from $Z$ to $\Der(A)$. The general and elegant formalism developed in \textit{loc. cit.} produces isomorphisms between fibers of any two points in the same symplectic core of $\Spec(Z)$, in particular on any two points of the open symplectic leaf. For instance, if $A$ is generically Azumaya over $Z$, we get that it is in fact Azumaya over the whole open symplectic leaf.  Poisson orders have been applied more recently to the theory of discriminants of PI algebras \cite{BrownYakimov}, \cite{NgyTramYak} and Sklyanin algebras \cite{YakWalWang}. 

The general method to obtain Poisson orders is as follows. Suppose $A_\Rt$ is a family of associative algebras parametrized by $t\in\Rt=\C[t, t^{-1}]$ together with a central subalgebra $Z\subset A_q$ at the special value $t=q$. Then under a mild assumption (see Proposition \ref{prop:PoissonOrderDegeneration}) we get the structure of a Poisson order on the pair $(A_q, Z)$. For instance, this assumption is satisfied when $Z$ is the whole center of $A_q$ (see Lemma \ref{lm:PoissonOrderAutomatic} which goes back to Hayashi \cite{HayashiSugawara}).

In our examples $A_\Rt$ is a family of algebras in $\RepR(G)$, so we combine the quantum Frobenius map and the notion of a Poisson order into the notion of a Frobenius Poisson order. Namely, a Frobenius Poisson order consists of an algebra $A_q\in \Repq(G)$, a Poisson algebra $Z\in\Rep(G)$ and a central embedding $\Fr^*(Z)\subset A_q$ such that $(A_q, Z)$ form a Poisson order. If $A_\Rt$ is a flat family of associative algebras in $\RepR(G)$ and $Z\in\Rep(G)$ is a central subalgebra at the special value $t=q$, then under the same mild assumption $(A_q, Z)$ form a Frobenius Poisson order, see Proposition \ref{prop:FrobeniusPoissonOrderDegeneration}.

\paragraph{Frobenius quantum Hamiltonian reduction} If $G$ is a Poisson-Lie group and $\cZ=\Spec Z$ is a Poisson $G$-variety equipped with a classical moment map $\mu\colon \cZ\rightarrow G$ (we recall the relevant formalism of group-valued moment maps in Section \ref{sec:Poisson}), the affine quotient $\cZ \dS G=\mu^{-1}(e) / G = \Spec \cO(\mu^{-1}(e))^G$ carries a natural structure of a Poisson variety, see Proposition \ref{prop:poissonreduction}.

Similarly, if $\Aq$ is an algebra in $\Repq(G)$ equipped with a quantum moment map $\mu_q\colon \Oq(G)\rightarrow \Aq$ from the reflection equation algebra $\Oq(G)$ (see Sections \ref{sec:REA} and \ref{sect:quantummomentmaps} for the relevant definitions), then we can form the quantum Hamiltonian reduction $\Aq \dS \UqL(\g)$ which is still an associative algebra (see Proposition \ref{prop:qhamreduction}).

We are interested in obtaining a Poisson order structure on the Hamiltonian reduction of $(\Aq, Z)$, so we need to assume the two moment maps are compatible with each other: this leads to the notion of a Frobenius quantum moment map for a Frobenius Poisson order $(\Aq, Z)$. Namely, we assume that the composite $\cO(G)\subset \Oq(G)\xrightarrow{\mu_q} \Aq$ factors through the central subalgebra $Z\subset \Aq$ which gives rise to a classical moment map $\cZ\rightarrow G$.

Our first result in this direction is that given a Frobenius Poisson order $(\Aq, Z)$ equipped with such a Frobenius quantum moment map satisfying some extra compatibilities (see Definition \ref{def:HamiltonianFrobeniusPoissonOrder} for the notion of a $G$-Hamiltonian Frobenius Poisson order) the Hamiltonian reduction $(\Aq \dS \UqL(\g), Z \dS G)$ becomes a Poisson order (see Proposition \ref{prop:FrobeniusPoissonReduction}).

Our second result in this direction is the following mechanism for obtaining Hamiltonian Frobenius Poisson orders (see Proposition \ref{prop:StrongHamiltonianOrder}). Suppose $A_\Rt\in\RepR(G)$ is a flat family of algebras such that the map on invariants $(A_\Rt)^{\ULR(\g)}\rightarrow (\Aq)^{\UqL(\g)}$ is surjective, $Z\subset \Aq$ is a central subalgebra and $\mu_\Rt\colon \OR(G)\rightarrow A_\Rt$ is a quantum moment map. Then $(\Aq, Z)$ becomes a Hamiltonian Frobenius Poisson order. In examples we check the surjectivity assumption using a good filtration on $A_\Rt$ (see Section \ref{sect:goodfiltrations}).

\paragraph{Distinguished fibers}
In each of our examples there is a distinguished point -- corresponding to the trivial local system, and the trivial quiver representation, respectively, where the Azumaya property can be checked directly:  for character varieties this proceeds by reducing to the small quantum group, while for quiver varieties it requies a long computation.

\paragraph{Non-degenerate $G$-Hamiltonian varieties}
In order to fully exploit the method of Poisson orders, we require a description of the open symplectic leaf of the framed moduli space. It turns out in our examples these can be described purely in terms of the moment map, because both of our examples are \emph{non-degenerate} Poisson $G$-varieties as defined in \cite{AlekseevQuasiPoisson}.  For framed character varieties, this is proved in {\it loc. cit.}, while for framed quiver varieties, we use a convenient characterization of nondegeneracy due to Li-Bland--Severa \cite{LiBlandSevera} to prove that they are nondegenerate by a direct computation in coordinates.  We show (Theorem \ref{thm:nondegenerateleaves}) that given a nondegenerate Poisson $G$-variety $X$ equipped with a moment map $\mu\colon X\rightarrow G$, the open symplectic leaf of $G$ is given by the preimage $\mu^{-1}(\Br)$ of the big cell $\Br=B_+B_-\subset G$.

\paragraph{Langlands duality at even roots of unity}

Throughout the paper we make simplifying assumptions on $G$ and the order $\ell$ of the root of unity $q$, e.g. that $G$ in the semisimple case is of adjoint type and $\ell$ is odd. This ensures that the M\"{u}ger center of $\Repq(G)$ is identified with the symmetric monoidal category $\Rep(G)$ equipped with the obvious braiding. This no longer holds if we relax the assumptions on $G$ and $\ell$.

Let us now assume that $\ell$ is divisible by $4$ and $G=\mathrm{Spin}(2n+1)$ for $n\geq 1$. Then the M\"{u}ger center of $\Repq(G)$ coincides with $\Repq({}^LG)$, where ${}^LG = \mathrm{PSp}(2n)$, the adjoint group of type $C_n$. Therefore, in this case the quantum character variety for $G$ forms a sheaf of algebras over the classical character variety for ${}^LG$.

Note that in these cases one still has a factorizable braided monoidal category $\Vect\otimes_{\Rep({}^LG)} \Repq(G)$, but it is not given by modules over a Hopf algebra. We expect that our approach nevertheless admits a minor modification which would prove that the quantum character variety is Azumaya over the classical character variety for a more general class of $G$ and $\ell$.

In the case of skein algebras, we note that the results of \cite{FrohmanUnicity} also apply in the case of even order roots of unity, and so we can apply the method of Poisson orders there as well.  We hope to return to this in future work as well.

\subsection{Outline}

We briefly describe the contents of this paper. In Section \ref{sec:Poisson}, we recall factorizable Poisson-Lie groups (Section \ref{subsec:factorizable}), multiplicative moment maps (Section \ref{subsec:mmm}), and nondegenerate Poisson $G$-varieties (Section \ref{subsec:nondegen}). The main result therein (Theorem \ref{thm:nondegenerateleaves}, Section \ref{subsec:symplecticleaves}) is a description of symplectic leaves of a nondegenerate Poisson $G$-variety $X$ as the preimages under the multiplicative moment map $\mu : X \rightarrow G$ of the $G^*$-orbits on $G$ (i.e. orbits under the dressing action), where $G^*$ is the Poisson-Lie dual group of $G$. 

In Section \ref{sec:frobenius}, we first consider reflection equation algebras, for generic parameters, for Lusztig's integral form, and at a root of unity (Section \ref{sec:REA}). We then set up the notions of Frobenius Poisson orders and Frobenius quantum moment maps (Section \ref{subsec:Frobeniuspairs}), and discuss degenerations of quantum groups and quantum algebras in light of these notions (Section \ref{subsec:degenqalgebras}). We also formulate the procedure of Frobenius quantum Hamiltonian reduction and prove that Azumaya algebras descend to Azumaya algebras under this procedure (Theorem \ref{thm:frobtwistedAzumaya}, Section \ref{sec:template-for-results}).

In Section \ref{sec:CharVar}, we apply the techniques developed in Section \ref{sec:frobenius} to the setting of character varieties. After recalling the construction of quantum character varieties and stacks, we show that the framed character variety, together with its ``quantum character sheaf'', form a Frobenius Poisson order (Section \ref{subsec:CharVarFrobPoissonOrder}). We exhibit a Frobenius quantum moment map in this setting (Section \ref{subsec:CharVarFrobqmm}), and run the procedure of Frobenius quantum Hamiltonian reduction (Section \ref{subsec:CharVarFrobqHred}) to obtain Azumaya algebras over  classical affine character varieties (Theorem \ref{ChG-thm-body}). 

In Section \ref{sec:DqMat}, we turn our attention to multiplicative quiver varieties and their quantizations at a root of unity.  We then recall in Section \ref{subsec:mqv} the construction of the multiplicative quiver variety, as a Hamiltonian reduction of the framed multiplicative quiver variety.  We recall in Section \ref{subsec:quantizationofQ} the quantization of the framed multiplicative quiver variety, and in Section \ref{subsec:quantizationofQ} we identify the central subalgebra, establish the existence of good filtrations, and finally construct the Frobenius Poisson order, in Theorem \ref{thm:quiver-center}.  In Section \ref{subsec:qmqvnondeg} we show that the resulting Poisson $G$-variety is nondegenerate (Theorem \ref{thm:qmqvnondeg}), and in Section \ref{subsec:qmqvazumaya} we prove that the zero representation is an Azumaya point.  In Section \ref{subsec:qmqvmom-map}, we construct the Frobenius quantum moment map (Theorem \ref{thm:mmgeneral}).  In Section \ref{subsec:QMQVrootofunity}, we define the quantized multiplicative quiver variety associated to an arbitrary GIT parameter $\theta$ via Frobenius quantum Hamiltonian reduction.  Finally, in Section \ref{subsec:sheaf}, we bring all the elements together to establish the Azumaya property on the smooth locus of the quantum multiplicative quiver variety (Theorem \ref{quiver-thm-intro}).

\subsection{Acknowledgments}

We would like to thank David Ben-Zvi and Kobi Kremnitzer for their guidance and encouragement throughout our work, Nicholas Cooney for many helpful conversations, Michael Gr\"ochenig for discussions about the Langlands dual case, Pavel Etingof for remarks about the small quantum group, Florian Naef for helpful remarks about non-degenerate quasi-Poisson structures, and Thang Le for assistance navigating skein theory literature.

The work of I.G. was supported by the Advanced Grant ``Arithmetic and Physics of Higgs moduli spaces'' No. 320593 of the European Research Council.  The work of D.J. was supported by the European Research Council (ERC) under the European Union's Horizon 2020 research and innovation programme [grant agreement no. 637618]. The work of P.S. was supported by the NCCR SwissMAP grant of the Swiss National Science Foundation.

\section{Multiplicative Hamiltonian reduction} \label{sec:Poisson}

The goal of this section is to collect some useful results on moment maps for factorizable Poisson-Lie groups. Throughout this section, we assume that $G$ is an arbitrary connected reductive group. For $x\in\g$ we denote by $x^L,x^R\in\T_G$ the left- and right-invariant vector fields with value $x$ at the unit. We assume the reader is familiar with the theory of quasi-Poisson groups and quasi-Poisson spaces, see e.g. \cite{AlekseevManinPairs} and \cite[Section 4.1]{SafronovQuantumMoment}. Let us just recall that a Poisson $G$-variety is a $G$-variety $X$ equipped with a Poisson structure such that the action map $G\times X\rightarrow X$ is Poisson. All varieties we consider in this section are affine.

\subsection{Factorizable Poisson-Lie groups}\label{subsec:factorizable}

Fix a nondegenerate element $c\in\Sym^2(\g)^G$. Let $\phi=[c_{12}, c_{13}]\in\wedge^3(\g)^G$. With respect to a basis $\{e_a\}$ of $\g$ we may write,
\[ [e_a, e_b] = \sum_c f^c_{ab} e_c, \qquad c = \sum_{a, b} c^{ab}e_a\otimes e_b,\qquad \phi = \frac{1}{6} \sum_{a, b, c, d, e} c^{ae} c^{bd} f_{ed}^c e_a\wedge e_b\wedge e_c.\]
We have the following quasi-Poisson structures:
\begin{itemize}
\item Denote by $G_\phi$ the quasi-Poisson group $G$ equipped with the zero bivector and trivector $\phi$.

\item Consider $G$ as a $G$-variety under conjugation and equip it with the bivector
\[\pi_{STS, \phi} = \sum_{a, b} c^{ab} e_a^R\wedge e_b^L.\]
By \cite[Proposition 3.1]{AlekseevQuasiPoisson} we get a quasi-Poisson $G_\phi$-variety that we denote by $G_{STS, \phi}$.
\end{itemize}

\begin{definition}
A classical twist is an element $t\in\wedge^2(\g)$ satisfying the equation
\[\frac{1}{2}[t, t] = -\phi.\]
\end{definition}

Equivalently, we may say that the classical $r$-matrix $r = t+c$ satisfies the classical Yang--Baxter equation
\[[r_{12}, r_{13}] + [r_{12}, r_{23}] + [r_{13}, r_{23}] = 0.\]
\begin{example}
In the case $\g=\glN$ we use the classical $r$-matrix
\[r = \frac{1}{2}\sum_i \cE^i_i\otimes \cE^i_i + \sum_{i>j} \cE^j_i\otimes \cE^i_j,\]
where $\cE_i^j$ is the elementary matrix with one in the $i$-th row and $j$-th column.
\end{example}

Using $t$ we may perform the following twists:
\begin{itemize}
\item Twist the quasi-Poisson structure $G_\phi$ into a Poisson-Lie structure
\[\pi_{Sk} = t^L - t^R\]
on $G$ that we denote by $G_{Sk}$. We call Poisson-Lie structures obtained in this manner \emph{factorizable}.

\item Twist the quasi-Poisson $G_\phi$-variety $G_{STS, \phi}$ to a Poisson $G_{Sk}$-variety $G_{STS}$ with the Poisson bivector
\[\pi_{STS} = \pi_{STS, \phi} - t^{\ad},\]
where $(-)^{\ad}$ is the map $\wedge^\bullet(\g)\rightarrow \Gamma(G, \wedge^\bullet \T_G)$ given by differentiating the conjugation action of $G$ on itself. The Poisson structure $\pi_{STS}$ was introduced in \cite{SemenovTianShansky}.
\end{itemize}

The classical twist gives $\g$ the structure of a Lie bialgebra and provides an embedding of Lie algebras $\g^*\rightarrow \g\oplus \g$. In particular, the action of $\g\oplus \g$ on $G$ by left and right translations induces a $\g^*$-action on $G$ which is known as the dressing action. Since the diagonal subalgebra $\g\subset \g\oplus \g$ and $\g^*\subset \g\oplus \g$ are transverse Lagrangians, the dressing $\g^*$-action on $G$ is free at the unit $e\in G$. In particular, there is an open orbit $\Br\subset G$ of $e\in G$. The following is shown in \cite[Section 4]{AlekseevMalkin}.
\begin{prop}
The symplectic leaves of $G_{STS}$ are the intersections of the conjugacy classes of $G$ with the dressing $\g^*$-orbits.
\label{prop:STSleaves}
\end{prop}

We will assume that the Lie algebra $\g^*$ integrates to a group $G^*$ and that there is a $\g$-action on $G^*$ such that the map $G^*\rightarrow G$ is $\g$-equivariant.

\begin{example}
Suppose $G$ is a connected reductive group equipped with a choice of a maximal torus $T\subset G$ and a Borel subgroup $B_+\subset G$ containing it. Let $c\in\Sym^2(\g)^G$ be a nondegenerate element. Let $\Delta^+$ be the set of positive roots and pick an orthonormal basis $\{k_i\}$ of $\Lie(T)$. Then $\{e_\alpha, f_\alpha, k_i\}$ is a basis of $\g$, where $e_\alpha$ for $\alpha\in\Delta^+$ is a standard basis of $\Lie(B_+)$. Then we have the \emph{standard} $r$-matrix
\[r_{std} = \frac{1}{2}\sum_i k_i\otimes k_i + \sum_{\alpha\in\Delta^+} d_\alpha f_\alpha\otimes e_\alpha.\]
Let $B_-$ be the opposite Borel subgroup. Denote by $p_\pm\colon B_\pm\rightarrow T$ the abelianization maps. Then the dual Poisson-Lie group is given by
\[G^* = \{(b_+, b_-)\in B_+\times B_-\ |\ p_+(b_+)p_-(b_-) = 1\}.\]
The open subscheme $\Br\subset G$, the image of $G^*\subset G$, is given by the subscheme $B_+B_-\subset G$ which is isomorphic to the big Bruhat cell. Note that in the case $G=\GL_n$ we have $Z(\GL_n)\subset \Br$.
\label{ex:standardrmatrix}
\end{example}

\subsection{Multiplicative moment maps}\label{subsec:mmm}

Fix a factorizable Poisson-Lie structure on $G$. Pick a basis $\{e_i\}$ of $\g$ and let $\{e^i\}$ be the dual basis of $\g^*$. Denote by $\tilde{a}\colon \g^*\rightarrow \Gamma(G, \T_G)$ the dressing $\g^*$-action on $G$.

\begin{definition}
Let $X$ be a Poisson $G$-variety. A $G$-equivariant map $\mu\colon X\rightarrow G$ is a moment map if for every $h\in \cO(G)$ and $a\in\cO(X)$ we have
\begin{equation}
\{\mu(h), a\} = \sum_i \mu(\tilde{a}(e^i).h) e_i.a.
\label{eq:classicalmomentmap}
\end{equation}
In this case we say that $X$ is a Hamiltonian $G$-Poisson variety.
\label{def:classicalmomentmap}
\end{definition}

\begin{prop}
Suppose $X$ is a Poisson $G$-variety and $\mu\colon X\rightarrow G$ is a moment map. Then $\mu\colon X\rightarrow G_{STS}$ is a Poisson morphism.
\label{prop:momentmapPoisson}
\end{prop}
\begin{proof}
By \cite[Proposition 4.18]{SafronovQuantumMoment}, our definition of moment maps agrees with the definition of Alekseev and Kosmann-Schwarzbach introduced in \cite{AlekseevManinPairs}. Now, let $X'$ be the quasi-Poisson $G_\phi$-variety obtained by twisting $X$ using $-t\in\wedge^2(\g)$. Then by \cite[Proposition 3.3]{AlekseevQuasiPoisson} the map $\mu\colon X'\rightarrow G_{STS, \phi}$ is quasi-Poisson and hence after twisting back we get that $X\rightarrow G_{STS}$ is Poisson.
\end{proof}

Using a moment map we may construct a Hamiltonian reduction of $X$.

\begin{definition}
Let $\xi\in G$ be an element in the center and $X$ a Poisson $G$-variety equipped with a moment map $\mu\colon X\rightarrow G$. The Hamiltonian reduction is
\[X \dS_{\xi} G = \Spec \cO(\mu^{-1}(\xi))^G.\] 
\end{definition}

Hamiltonian reduction carries a natural Poisson structure constructed in the following way. Let $I\subset \cO(X)$ be the ideal defining $\mu^{-1}(\xi)$. Since $G$ is reductive, the natural morphism
\[\cO(X)^G / I^G\longrightarrow (\cO(X) / I)^G\]
is an isomorphism. So, it is enough to construct a Poisson structure on $\cO(X)^G / I^G$.

\begin{prop}
$\cO(X)^G\subset \cO(X)$ is a Poisson subalgebra and $I^G\subset \cO(X)^G$ is a Poisson ideal.
\label{prop:poissonreduction}
\end{prop}
\begin{proof}
Suppose $a,b\in\cO(X)^G$. Then for every $x\in\g$
\[x.\{a, b\} = \{x.a, b\} + \{a, x.b\} + (x_{[1]}.a)(x_{[2]}.b),\]
where $\delta(x) = x_{[1]}\otimes x_{[2]}$ is the Lie cobracket. All three terms vanish due to $G$-invariance of $a$ and $b$.

By the first claim it is enough to show that $\{\cO(X)^G, I\}\subset I$. Consider $a\in\cO(X)^G$, $b\in\cO(X)$ and $h\in\cO(G)$ such that $h(\xi)=0$. Then
\begin{align*}
\{a, b\mu(h)\} &= \{a, b\}\mu(h) + \{a, \mu(h)\} b \\
&= \{a, b\} \mu(h) - b\sum_i \mu(\tilde{a}(e^i).h) e_i.a,
\end{align*}
where we use the moment map equation \eqref{eq:classicalmomentmap} in the last line. Since $a$ is $G$-invariant, the last term is zero. So, $\{a, I\}\subset I$.
\end{proof}

Therefore, we obtain a natural Poisson structure on $\cO(X)^G / I^G$, i.e. $X\dS G$ is a Poisson variety.

\subsection{Nondegenerate Poisson $G$-varieties}\label{subsec:nondegen}

Let $a\colon \g\rightarrow \Gamma(X, \T_X)$ be the $\g$-action on a $G$-variety $X$. The following was introduced in \cite[Definition 9.1]{AlekseevQuasiPoisson}.

\begin{definition}
Let $G$ be a quasi-Poisson group and $(X, \pi_X)$ a smooth quasi-Poisson $G$-variety. $X$ is nondegenerate if the map
\[(a, \pi_X^\sharp)\colon \g\oplus \T^*_X\longrightarrow \T_X\]
is surjective.
\end{definition}

\begin{rmk}
Clearly, if $X$ is a Poisson $G$-variety, the locus where the Poisson structure is symplectic is contained in the nondegeneracy locus.
\end{rmk}

\begin{lemma}
Let $G$ be a quasi-Poisson group, $(X, \pi_X)$ a quasi-Poisson $G$-variety, and  $t\in\wedge^2(\g)$. Then $(X, \pi_X)$ is nondegenerate if, and only if, its twist $(X, \pi_X - a(t))$ is nondegenerate.
\end{lemma}
\begin{proof}
Indeed, the images of $\pi_X^\sharp$ and $\pi_X^\sharp - a(t)^\sharp$ differ by the image of $a$.
\end{proof}

We have the following convenient criterion to determine that a Poisson $G$-variety is nondegenerate, proven in \cite[Theorem 3]{LiBlandSevera}.

\begin{prop}
Let $r$ be a classical $r$-matrix defining a factorizable Poisson-Lie structure on $G$ and suppose $(X, \pi_X)$ is a Poisson $G$-space equipped with a moment map $\mu\colon X\rightarrow G$. Denote $\widetilde{\pi}_X = \pi_X + a(r)\in\Gamma(X, \T_X\otimes \T_X)$ and $\widetilde{\pi}_X^\sharp\colon \T^*_X\rightarrow \T_X$ the induced map. Then $(X, \pi_X)$ is nondegenerate if, and only if, $\widetilde{\pi}_X^\sharp\colon \T^*_X\rightarrow \T_X$ is an isomorphism.
\label{prop:LiBlandSevera}
\end{prop}

Let us now show that nondegeneracy is preserved by fusion. Consider the element
\[\psi = \sum_{a, b} r^{ab} e_a^1\wedge e_b^2\in\wedge^2(\g\oplus \g),\]
where $r$ is the classical $r$-matrix.

\begin{prop}
Suppose $(X, \pi_X)$ is a Poisson $G\times G$-variety and consider its \emph{fusion} $X_{fus} = (X, \pi_X - a(\psi))$. Then:
\begin{enumerate}
\item $X_{fus}$ is a Poisson $G$-variety with respect to the diagonal $G$-action.

\item Suppose $\mu_1\times \mu_2\colon X\rightarrow G\times G$ is a moment map. Then $\mu_{fus} = \mu_1\mu_2\colon X_{fus}\rightarrow G$ is a moment map for the diagonal $G$-action.

\item Suppose $X$ is equipped with a moment map and is nondegenerate. Then $X_{fus}$ is nondegenerate.
\end{enumerate}
\label{prop:Poissonfusion}
\end{prop}
\begin{proof}
Let $a_1, a_2\colon \g\rightarrow \Gamma(X, \T_X)$ be the two $\g$-actions on $X$ and let $a_{diag} = a_1 + a_2$ be the diagonal $\g$-action.

Applying the twist by $t$, we can turn $X$ into a quasi-Poisson $G_\phi\times G_\phi$-variety with the bivector $\pi' = \pi + a_1(t) + a_2(t)$. Let $\psi' = \sum_{a, b} c^{ab} e_a^1\wedge e_b^2$. Then $X$ with the bivector $\pi'_{fus} = \pi' - a(\psi')$ is the fusion in the sense of quasi-Poisson varieties \cite[Section 5.1]{AlekseevQuasiPoisson}. In particular, $(X, \pi'_{fus})$ is a quasi-Poisson $G_\phi$-variety with respect to the diagonal $G$-action. Applying the twist we get that $(X, \pi'_{fus} - a_{diag}(t))$ is a Poisson $G$-variety with respect to the diagonal $G$-action. We have
\[\pi'_{fus} - a_{diag}(t) = \pi + a_1(t) + a_2(t) - a_{diag}(t) - a(\psi'),\]
where
\[a_1(t) + a_2(t) - a_{diag}(t) = -\frac{1}{2} \sum_{a, b} t^{ab} e_a^1\wedge e_b^2 - \frac{1}{2} \sum_{a, b}t^{ab} e_a^2\wedge e_b^1 = -\sum_{a, b} t^{ab} e_a^1\wedge e_b^2.\]
Therefore, $\pi'_{fus} - a_{diag}(t) = \pi_X - a(\psi)$ which proves the first claim.

The last two claims follow in the same way from \cite[Proposition 5.1]{AlekseevQuasiPoisson} and \cite[Section 10]{AlekseevQuasiPoisson} respectively applied to $(X, \pi')$.
\end{proof}

\subsection{Symplectic leaves of nondegenerate Poisson $G$-varieties}\label{subsec:symplecticleaves}

Nondegenerate Poisson $G$-varieties equipped with moment maps have an easy description of symplectic leaves. Let $a_G\colon \g\rightarrow \Gamma(G, \T_G)$ be the adjoint action.

\begin{theorem}
Let $X$ be a nondegenerate Poisson $G$-variety equipped with a moment map $\mu\colon X\rightarrow G$. Then the symplectic leaves of $X$ are given by the preimages of the dressing $\g^*$-orbits $G_\alpha\subset G$.
\label{thm:nondegenerateleaves}
\end{theorem}
\begin{proof}
The moment map equation is equivalent to the commutativity of the diagram
\[
\xymatrix{
\T^*_{X, x} \ar^{-\pi^\sharp_X}[r] \ar^{a^*}[d] & \T_{X, x} \ar^{\mathrm{d}\mu_x}[d] \\
\g^* \ar^{\tilde{a}}[r] & \T_{G, \mu(x)}
}
\]
In particular, symplectic leaves of $X$ are contained in the preimages $\mu^{-1}(G_\alpha)$. Since $X$ is nondegenerate, it is enough to prove that if $y\in\g$ is such that $a_G(y)$ is in the image of $\tilde{a}$, then $a(y)$ is in the image of $\pi_X^\sharp$.

By Proposition \ref{prop:momentmapPoisson} the moment map $\mu\colon X\rightarrow G_{STS}$ is Poisson, so we have a commutative diagram
\[
\xymatrix{
\T^*_{X, x} \ar^{\pi_X^\sharp}[r] & \T_{X, x} \ar^{\mathrm{d}\mu_x}[d] \\
\T^*_{G, \mu(x)} \ar_{\mu^*}[u] \ar^{\pi_{STS}^\sharp}[r] & \T_{G, \mu(x)}
}
\]

Since $a_G(y)$ is in the image of $\tilde{a}$, by Proposition \ref{prop:STSleaves} the element $a_G(y)$ lies in the image of $\pi_{STS}^\sharp\colon \T^*_{G, \mu(x)}\rightarrow \T_{G, \mu(x)}$ and hence $a(y)$ is in the image of $\pi_X^\sharp$ by the above commutative diagram.
\end{proof}

In particular, under the above assumptions $\mu^{-1}(\Br)\subset X$ is an open symplectic leaf. We will now use this observation to show that the Hamiltonian reduction of a nondegenerate Poisson $G$-variety is symplectic. Denote by $\pi\colon \mu^{-1}(\xi)\rightarrow X\dS G$ the projection.

\begin{prop}
Let $X$ be a nondegenerate Poisson $G$-variety equipped with a moment map $\mu\colon X\rightarrow G$ and $\xi\in\Br\subset G$ a point lying in the center. Suppose $U\subset \mu^{-1}(\xi)$ and $V\subset X\dS G$ are open subsets such that $\pi\colon U\rightarrow V$ is a $G$-torsor. Then $V\subset X\dS G$ is a smooth symplectic variety.
\label{prop:symplecticquotient}
\end{prop}
\begin{proof}
Consider $x\in U\subset \mu^{-1}(\xi)$. The moment map condition gives a commutative diagram
\[
\xymatrix{
\T^*_{X, x} \ar^{-\pi^\sharp_X}[r] \ar^{a^*}[d] & \T_{X, x} \ar^{\mathrm{d}\mu_x}[d] \\
\g^* \ar^{\tilde{a}}[r] & \T_{G, \xi}
}
\]

Since $\xi$ lies in an open $G^*$-orbit, the bottom map is an isomorphism. By Theorem \ref{thm:nondegenerateleaves} $x$ lies in an open sympletic leaf, so the top map is an isomorphism as well. Since the $G$-action on $U$ is free, the vertical map on the left is surjective. Therefore, the vertical map on the right is surjective as well. In other words, $U$ is smooth.

Let $Y\rightarrow V$ be an \'{e}tale cover over which $U\rightarrow V$ splits. In other words, $U\times_V Y\cong Y\times G$. Since $U$ is smooth, so is $U\times_V Y$ and hence $Y$ is smooth. By descent this implies that $V$ is smooth as well.

Denote by $[x]\in V$ the image of $x\in U$. We have a diagram of tangent and cotangent spaces
\[
\xymatrix{
&&& 0 \ar[d] & \\
& 0 \ar[d] && \T^*_{[x]} V \ar[d] & \\
0 \ar[r] & \g \ar[r] \ar[d] & \T^*_x X \ar[r] \ar^{\sim}[d] & \T^*_x U \ar[r] \ar[d] & 0 \\
0 \ar[r] & \T_x U \ar[r] \ar[d] & \T_x X \ar[r] & \g^* \ar[r] \ar[d] & 0 \\
& \T_{[x]}V \ar[d] && 0& \\
& 0 &&&
}
\]
The Poisson structure on $V$, i.e. the map $\T^*_{[x]} V\rightarrow \T_{[x]} V$, by Proposition \ref{prop:poissonreduction} is given by traversing this diagram. The snake lemma shows that this map is in fact an isomorphism.
\end{proof}

In most examples we will consider the $G$-action on $X$ has a generic stabilizer, so let us explain how to deal with that.

\begin{prop}
Let $G$ be a factorizable Poisson-Lie group, $X$ a Poisson $G$-variety, $G_{stab}\subset G$ is a normal subgroup which acts trivially on $X$. We assume the pairing on $\g$ restricts to a nondegenerate pairing on $\g_{stab}$. Denote $\overline{G} = G / G_{stab}$.
\begin{enumerate}
\item If $\mu\colon X\rightarrow G$ is a moment map for the $G$-action on $X$, then $\overline{\mu}\colon X\rightarrow G\rightarrow \overline{G}$ is a moment map for the $\overline{G}$-action on $X$.

\item If $X$ is a nondegenerate Poisson $G$-variety, then it is a nondegenerate Poisson $\overline{G}$-variety.
\end{enumerate}
\label{prop:momentmapstabilizer}
\end{prop}
\begin{proof}$ $
\begin{enumerate}
\item Let $r\in\g\otimes \g$ be the classical $r$-matrix defining the factorizable Poisson-Lie structure on $G$ and denote by $\overline{r}\in\overline{\g}\otimes\overline{\g}$ its image. Then it is still a classical $r$-matrix and by assumption its symmetric part is nondegenerate. Therefore, $\overline{G}$ inherits the structure of a factorizable Poisson-Lie group.

Moreover, the natural diagram
\[
\xymatrix{
\overline{\g}^* \ar[d] \ar[r] & \overline{\g}\oplus \overline{\g} \\
\g^* \ar[r] & \g\oplus \g \ar[u]
}
\]
is commutative. Therefore, $\cO(\overline{G})\rightarrow \cO(G)$ is $\overline{\g}$-equivariant which implies the first claim.

\item The image of $\overline{\g}\rightarrow \Gamma(X, \T_X)$ coincides with the image of $\g\rightarrow \Gamma(X, \T_X)$ which implies the second claim.
\end{enumerate}
\end{proof}

Recall that the usual moment map on a symplectic $G$-variety is unique up to an addition of a character. A similar statement is true for multiplicative moment maps on nondegenerate Poisson $G$-varieties.

\begin{prop}
Let $X$ be a connected nondegenerate Poisson $G$-variety and $\mu_1, \mu_2\colon X\rightarrow G$ two moment maps. Then there is an element $g\in G^*$ landing in the center $Z(G)\subset G$ under the projection $G^*\rightarrow G$ such that $\mu_2(x) = g\mu_1(x)$.
\label{prop:momentmapunique}
\end{prop}
\begin{proof}
Consider an open dense subset $U = \mu_1^{-1}(\Br)\cap \mu_2^{-1}(\Br)\subset X$. It will be enough to establish the claim on this open subset.

We have $\mu_1, \mu_2\colon U\rightarrow \Br$. We may choose a connected \'{e}tale cover $\tilde{U}\rightarrow U$ and lift the moment maps as
\[
\xymatrix{
\tilde{U} \ar[d] \ar^{\tilde{\mu}_i}[r] & G^* \ar[d] \\
U \ar^{\mu_i}[r] & \Br
}
\]

Let $\theta\in\Omega^1(G^*;\g^*)$ be the left-invariant Maurer--Cartan form on $G^*$. It is shown in \cite[Example 5.1.3]{AlekseevManinPairs} (see also \cite[Lemma 4.22]{SafronovQuantumMoment}) that the moment map equation \eqref{eq:classicalmomentmap} reduces to
\[a(v) = \pi_{\tilde{U}}^\sharp(\tilde{\mu}_i^* \langle \theta, v\rangle),\]
where $v\in\g$ and $a\colon \g\rightarrow \Gamma(\tilde{U}, \T_{\tilde{U}})$ is the action map. By Theorem \ref{thm:nondegenerateleaves} $\tilde{U}$ is symplectic, so the above equation implies that
\[\tilde{\mu}_1^* \langle\theta, v\rangle = \tilde{\mu}_2^* \langle\theta, v\rangle.\]
If we write $\tilde{\mu}_2(x) = g(x)\tilde{\mu}_1(x)$ for some function $g\colon \tilde{U}\rightarrow G^*$, then the above equation implies that $g$ is locally constant. Since $\tilde{U}$ is connected, $g$ is constant. Therefore, we obtain that $\mu_2(x) = g\mu_1(x)$ for some $g\in G^*$. Since both $\mu_1$ and $\mu_2$ are $G$-equivariant and the map $G^*\rightarrow G$ is $\g$-equivariant, we obtain that the image of $g\in G^*$ under $G^*\rightarrow G$ is invariant under conjugation, i.e. it lies in the center $Z(G)\subset G$.
\end{proof}

\section{Frobenius compatibilities}\label{sec:frobenius}
In this section we spell out a number of compatibility conditions between the quantum Frobenius homomorphism and each of the following:  quantum moment maps, Poisson orders, and Hamiltonian reductions.  We use these compatibilities to sheafify the quantum Hamiltonian reduction over the classical multiplicative Hamiltonian reduction, and  show that the reduction procedure preserves the Azumaya property on fibers. Throughout this section we fix a connected reductive group $G$ and a primitive $\ell$-th root of unity $q$.

\subsection{Factorizability of the small quantum group}

\begin{assumption}
The group $G$ and the number $\ell$ satisfy the following conditions:
\begin{itemize}
\item $\ell$ is odd.

\item $G$ is a product of groups of the following kinds: $\GL_n$ and a simple group whose determinant of the Cartan matrix is coprime to $\ell$.

\item If $G$ contains a factor of type $\mathrm{G}_2$, $\ell$ is not divisible by $3$.
\end{itemize}
\label{EllAssumption}
\end{assumption}

The reason we consider these restrictions is due to the following statement.

\begin{theorem}
Suppose $(G, \ell)$ satisfy Assumption \ref{EllAssumption}. Then there is an equivalence between $\Rep(G)$ and the M\"{u}ger center of $\Repq(G)$ such that the functor $\Rep(G)\rightarrow \Repq(G)$ is ribbon. Moreover, there is a factorizable Hopf algebra $\uq(\g)$, the small quantum group, together with an equivalence of ribbon categories $\Rep(\uq(\g))\cong \Vect\otimes_{\Rep(G)} \Repq(G)$.
\label{thm:smallquantumgroupfactorizable}
\end{theorem}
\begin{proof}
Let $X$ be the weight lattice of $G$ and let $\alpha_i\in X$ be the simple roots. Fix a $\Q$-valued symmetric bilinear pairing on $X$ such that $2\frac{(\alpha_i, \alpha_j)}{(\alpha_i, \alpha_i)}=a_{ij}$ is the Cartan matrix. We normalize $(-, -)$ so that the norm squared of a short root is $2$. Let
\[d_i = \frac{(\alpha_i, \alpha_i)}{2}.\]

In general the pairing $(-, -)$ is valued in $\frac{1}{D}\Z$, where $D$ divides the determinant of the Cartan matrix. By assumptions $\ell$ is coprime to $D$. Thus, by the Chinese remainder theorem we may choose $v\in\C$ such that $v^D = q$ and $v^\ell = 1$. For two weights $\mu,\nu\in X$ we set by definition $q^{(\mu, \nu)} = v^{D(\mu, \nu)}$. The braiding $\sigma_{W_1, W_2}$ for $W_1,W_2\in \Repq(G)$ is given by
\[\sigma_{W_1, W_2}(w_1\otimes w_2) = q^{-(\mu, \nu)} \Theta(w_1\otimes w_2),\]
where $\Theta$ is the so-called quasi $R$-matrix (see \cite[Chapter 4]{LusztigIntroductionQuantumGroups2010}) and $w_i\in W_i$ are vectors of weights $\mu,\nu\in X$ respectively.

The simple objects in $\Repq(G)$ are parametrized by dominant highest weights $\lambda\in X$. It is shown in \cite[Theorem 4.3]{NegronLogModular} that the M\"{u}ger center is the tensor category generated by simples with highest weight in the sublattice $X^M\subset X$ defined in the following way:
\[X^M = \{\lambda\in X\ |\ 2D(\lambda, \mu)\in \ell \Z\ \forall \mu\in X\}.\]

Suppose $X^M = \ell X$. Then Lusztig's quantum Frobenius map \cite[Chapter 35]{LusztigIntroductionQuantumGroups2010} defines a monoidal functor $\Rep(G)\rightarrow \Repq(G)$ where objects in $\Rep(G)$ have weights in the lattice $X^M$ which by the above result of \cite{NegronLogModular} lands in the M\"{u}ger center.

From the explicit formula for the quasi $R$-matrix $\Theta$ it is clear that it annihilates all representations of the form $\Fr^*(V)$. Therefore, $\Fr^*\colon \Rep(G)\rightarrow \Repq(G)$ is a braided monoidal functor iff $q^{(\mu, \nu)} = 1$ for every $\mu,\nu\in X^M$. But $\ell$ is coprime to $2D$, so $2D(\lambda, \mu) \in \ell\Z$ implies that $(\lambda, \mu)\in\ell\Z$ for every $\lambda,\mu\in X^M$. The fact that under the condition $X^M = \ell X$ the functor $\Fr^*\colon \Rep(G)\rightarrow \Repq(G)$ is ribbon is proved as in \cite[Lemma 6.1]{NegronLogModular}.

We will now check the conditions $X^M = \ell X$ in our cases.

\begin{itemize}
\item Suppose $G=\GL_n$. The weight lattice is $X=\Z^n$ with generators $\{e_i\}$ such that $(e_i, e_j) = \delta_{ij}$. The simple roots are $\alpha_i = e_i - e_{i+1}$ for $1\leq i< n$. Let $\lambda = \sum n_i e_i$. Then $2(\lambda, e_j) = n_j$. So, $X^M = \ell X$.

\item Suppose $G$ is a simple simply-connected group. Then the weight lattice $X$ is generated by the fundamental weights $\omega_i$. Let $b_{ij}$ be the inverse of the Cartan matrix $a_{ij}$. Consider
\[\Omega_{ij} = (\omega_i, \omega_j) = \sum_{k,l} b_{ik} b_{jl} d_k a_{kl}.\]
Let $\Omega^{-1}_{ij}$ be its inverse. Note that if $G$ is simply-laced (i.e. $d_k=1$), we get $\Omega_{ij} = b_{ij}$ and $\Omega^{-1}_{ij} = a_{ij}$.

Consider $\lambda = \sum_i n_i\omega_i$. Then $\lambda\in X^M$ iff
\[2(\lambda, \omega_j) = 2\sum_i n_i \Omega_{ij} = \ell m_j\]
for some integers $m_j$, i.e.
\[n_i = \frac{\ell}{2} \sum_j \Omega^{-1}_{ij} m_j.\]
By an explicit calculation, one sees that $\Omega^{-1}_{ij}\in \frac{1}{4}\Z$ if $G$ is not of type $G_2$ and $\Omega^{-1}_{ij}\in\frac{1}{3}\Z$ if $G$ is of type $G_2$. In these cases $\ell$ is coprime to the denominators of $\Omega^{-1}_{ij}$, and hence $n_i$ is divisible by $\ell$, i.e. $X^M\subset \ell X$.

Since $\ell$ is coprime to $D$, $2D\ell(\lambda, \mu)$ is divisible by $\ell$ for every $\lambda,\mu\in X$. Therefore, $\ell X\subset X^M$.

\item Suppose $G$ is an arbitrary simple group. The inclusion $\ell X\subset X^M$ is proved identically to the simply-connected case.

Now consider the symmetric pairing $A\colon X / \ell X\otimes X/\ell X\rightarrow \Z/\ell Z$ given by
\[A(\lambda, \mu) = 2D(\lambda, \mu)\]
which is well-defined due to the inclusion $\ell X\subset X^M$. Then $X^M = \ell X$ iff $A$ is nondegenerate.

Let $X_{sc}\supset X$ be the lattice spanned by fundamental weights. The index of $X\subset X_{sc}$ divides $D$ and hence is coprime to $\ell$ by assumption. Therefore, the inclusion $X / \ell X\rightarrow X_{sc} / \ell X_{sc}$ is an isomorphism. This means that the pairing $A$ is nondegenerate for $X$ iff $A$ is nondegenerate for $X_{sc}$ which we have already checked above.
\end{itemize}

The equivalence $\Vect\otimes_{\Rep(G)} \Repq(G)\cong\Rep(\uq(\g))$ of ribbon categories is the content of \cite{arkhipov_another_2003}, \cite{AngionoDeequivariantization} and \cite[Section 5]{NegronLogModular}.
\end{proof}

Note that some results related to the previous theorem were previously obtained in \cite{RossoQuantumGroups}, \cite[Chapter XI.6.3]{TuraevQuantumInvariants}, \cite{lentner_factorizable_2017}. The case of even $\ell$ is considered in \cite{NegronLogModular}.

\subsection{Good filtrations} \label{sect:goodfiltrations}
We will require at a few places the notion of a good filtration, also known as a dual-Weyl, or co-Weyl filtration on a $\UqL(\g)$-module.  The reader may safely skip this material at a first read, and may consult \cite{ParadowskiFiltrations} \cite{JantzenRepresentations} for further details.

Let ${\ULR}(\b)$ denote the quantum Borel subalgebra, generated by the $E_\alpha^{(k)}$'s and the quantum Cartan subalgebra.
\begin{definition}  Let $\lambda$ denote a dominant integral weight, regarded as a character of ${\ULR}^+(\g).$  We have:
\begin{enumerate}
\item The Verma module $M(\lambda)$ is the $\ULR$-module induced from $\ULR(\b)$ to $\ULR(\g)$, i.e.
$$M(\lambda) = \ULR(\g)\otimes_{\ULR(\b)} \C_\lambda.$$
\item The Weyl module $\Delta(\lambda)$ is the universal finitely generated $\Rt$-module quotient of $M(\lambda)$, i.e. for any finitely other generated $\Rt$-module quotient of $M(\lambda)\to V$, we have a unique factorization $M(\lambda)\to \Delta(\lambda)\to V$.
\item The dual Weyl module $\nabla(\lambda)$ is the ${\ULR}(\g)$-module coinduced from $\ULR(\b)$ to $\ULR(\g)$, i.e.
$$\nabla(\lambda) = Hom_{\ULR(\b)}(M(0),\lambda),$$
with left $\ULR(\g)$-action obtained by precomposing the evident right $\ULR(\g)$-action by the antipode $S$.
\item A Weyl filtration on a $\ULR(\g)$ module $M$ is a filtration $F_\bullet M$ with successive quotients $F_i/F_{i-1}$ isomorphic to some $\Delta(\lambda)$.
\item A good (a.k.a. dual-Weyl) filtration on a $\ULR(\g)$ module $M$ is a filtration $F_\bullet M$ with successive quotients $F_i/F_{i-1}$ isomorphic to some $\nabla(\lambda)$.
\end{enumerate}
\end{definition}

\begin{prop}[\cite{ParadowskiFiltrations}, see also {\cite[Section A.9]{VaragnoloDoubleaffineHecke2010}}]\label{prop:goodfiltrations}
We have the following:
\begin{enumerate}
\item We have an isomorphism, $\nabla(\lambda)^* \cong \Delta(-\omega_0\lambda)$.	
\item The tensor product of finite dimensional modules with good (resp. Weyl) filtration admits a good (resp. Weyl) filtration.
\item If $V_\Rt$ is a module with good filtration, then the specialization map $(V_\Rt)^{\ULR} \to (V_q)^{\UqL}$ is a surjection.
\end{enumerate}
\end{prop}

\subsection{Reflection equation algebras} \label{sec:REA}
In the geometric representation theory of quantum groups, an important role is played by a canonical ad-equivariant quantization of the coordinate algebra of $G$, the so-called ``reflection equation'' algebra.  In particular, the reflection equation algebra is the domain for quantum moment maps, so its behavior at a root of unity will be important for the notion of Frobenius quantum moment maps and their Hamiltonian reduction.

\subsubsection{Generic parameters} Before discussing the somewhat more complicated behavior of the reflection equation algebra at a root of unity, we recall the situation for generic parameters. Specifically, we discuss the relations between the three equivalent constructions of the rational form $\O_t(G)$ of the reflection equation algebra: (1) the functorial/co-end construction of Lyubashenko-Majid \cite{LyubashenkoBraidedGroupsQuantum1994}, (2) the matrix coefficient/braided dual formulation of Donin-Kulish-Mudrov \cite{DoninUniversalSolutionReflection2003}, and finally (3) the identification with a subalgebra of the quantized enveloping algebra of ad-locally finite elements, via the so-called Rosso isomorphism, due to Joseph-Letzter and Rosso \cite{JosephLocalfinitenessadjoint1992}, \cite{RossoAnalogues}.  We follow the exposition of \cite{Ben-ZviIntegratingquantumgroups2018}. 

The functorial/co-end construction works for any presentable rigid braided tensor category $\cC$. Let $T:\cC\bt\cC\to\cC$ be the tensor multiplication functor, which we will always assume to be cocontinuous. The braiding on $\cC$ equips $T$ moreover with the structure of a tensor functor, while the cocontinuity of $T$ endows it with a right adjoint $T^R$, which is lax monoidal. It follows that the associated co-end, $TT^R(\mathbf{1}_\cC)$, is a bialgebra object in $\cC$. Explicitly, it may be realized as a quotient of the sum of $V^*\otimes V$, where $V$ ranges over all objects of $\cC$. In addition, for any object $W\in\cC$ there is a ``field-goal'' isomorphism \cite[Corollary 4.6]{Ben-ZviQuantumcharactervarieties2018}
\[\tau_W\colon TT^R(\mathbf{1}_\cC)\otimes W\xrightarrow{\sim} W\otimes TT^R(\mathbf{1}_\cC)\]
given diagrammatically by
\begin{equation}
\tau_W \quad = \quad
\tik{
	\node[label=above:$W$] at (0,0) {};
	\node[label=above:$V^*$] at (1.5,0) {};
	\node[label=above:$V$] at (2.5,0) {};

	\node[label=below:$W$] at (2.5,-1) {};
	\node[label=below:$V^*$] at (0,-1) {};
	\node[label=below:$V$] at (1,-1) {};

	\fieldgoal{0}{0};
}.
\label{eq:fieldgoal}
\end{equation}

When $\cC$ is the category of locally finite modules for a quasi-triangular Hopf algebra $H$, then classic results of Majid-Lyubashenko \cite{LyubashenkoBraidedGroupsQuantum1994} identify the co-end algebra with so called braided dual $\widetilde{H}$.  This is the subalgebra of $H^*$ spanned by matrix coefficients of finite-dimensional $H$-modules, with a natural multiplication structure which invokes the braiding on $\cC$, hence the universal $R$-matrix for $H$.

In the case that $H=\UK(\g)$ is the rational form of the quantum group (see Section \ref{subsec:methods} above), we have that $\cC$ is the category $\RepK(G)$ of locally finite modules for the quantum group $\UK$, the co-end/braided dual algebra is denoted $\OK(G)$.  Regarded as an algebra in vector spaces equipped with a compatible $\UK$-action, the algebra $\OK(G)$ is spanned by matrix coefficients for the irreducible representations, and admits a Peter-Weyl decomposition,
$$\OK(G) \cong \bigoplus_{\lambda} V_\lambda^* \ot V_\lambda,$$
where the sum is over the dominant integral weights.

If we take $H=\ULR(\g)$, Lusztig's integral form, we have a natural coaction $\Delta\colon \OR(G)\rightarrow \ULR(\g)\otimes \OR(G)$ realizing $\OR(G)$ as a $\ULR(\g)$-comodule algebra (see e.g. \cite{KolbStokmanREA}). Applying the natural evaluation homomorphism $\OR(G)\rightarrow \Rt$ to the second factor, we obtain the Rosso homomorphism $\OR(G)\rightarrow \ULR(\g)$ \cite{JosephLocalfinitenessadjoint1992}. For $h\in\OR(G)$ we denote the coaction by $\Delta(h) = h_{(1)}\otimes h_{(2)}$. Then the ``field-goal'' isomorphism for $v\in V$ is given by
\[\tau_V(h\otimes v) = (h_{(1)}\triangleright v)\otimes h_{(2)}.\] 

Finally, let us recall that at generic parameters, the category $\RepK(G)$ is \emph{factorizable}.  This implies, among other things, that the map $\OK(G)\rightarrow \UK(\g)$ is injective, embedding $\OK(G)$ as a sub-algebra of ad-locally finite elements of the quantum group.  We note that this is a strong self-duality property, which holds only in the case that $t$ is generic; the failure of this property at roots of unity, and the precise way in which it fails, will be important point for the sequel.

\subsubsection{Root of unity parameters}

Let us now consider the divided powers quantum group $\UqL(\g)$, specialized to an odd root of unity $q^\ell=1$, with $\ell>1$. We denote by $\Oq(G)$ the co-end algebra of $\Repq(G)$, defined functorially as in the generic case.  Two very special phenomena occur at this specialization.

The first is that $\Oq(G)$ develops a large $\ell$-center, in addition to its Harish-Chandra center.  This is best understood as a consequence of the quantum Frobenius homomorphism, as follows:  the natural tensor structure on $\Fr^*$ induces an isomorphism $\Fr^*\circ T \cong T\circ (\Fr^*\bt\Fr^*)$ (where we abuse notation and write $T$ for the tensor multiplication in both tensor categories), and hence determines by standard adjunctions a bialgebra monomorphism in $\Repq(G),$
$$i:\Fr^*(\O(G)) \into \Oq(G).$$
Because the image of $\Fr^*$ lies in the M\"uger center of $\Repq(G)$, it follows  that $\Fr^*(\O(G))$ is a central subalgebra of $\Oq(G)$.  It is therefore natural to view $\Oq(G)$ as a sheaf of algebras over $G$.

The second special phenomenon at a root of unity is that factorizability breaks down.  For this reason, the Rosso map $\Oq(G)\to\UqL(\g)$ is not injective. Rather, it factors as a surjection onto the small quantum group $\uq(\g)\subset \UqL(\g)$, and its kernel is the augmentation ideal of $\cO(G)$.

In a similar way, the $\UqL(\g)$-coaction on $\Oq(G)$ is induced from a $\uq(\g)$-coaction on $\Oq(G)$. Moreover, the $\uq(\g)$-coaction on $\cO(G)\subset \Oq(G)$ is trivial. So, for a point $\xi\in G$ we have a $\uq(\g)$-comodule algebra
\[o_q(G; \xi) = \Oq(G)\otimes_{\cO(G)} \C,\]
where the $\cO(G)$-action on $\C$ is via the character $\xi^*\colon \cO(G)\rightarrow \C$.

\begin{rmk}
We note that for any representation $V$, of $\UqL(\g)$, the representation $V^*\ot V$ has weights entirely in the root lattice.  In particular, since $\Oq(G)$ is a co-end of such representations, its weights are also contained in the root lattice.
\label{rmk:REAcenteraction}
\end{rmk}

We record the following very important modification of the Peter-Weyl theorem.

\begin{theorem}[\cite{JantzenRepresentations}, \cite{ParadowskiFiltrations}]
The algebra $\Oq(G)$ admits a good filtration with successive quotients,
$$\nabla(-\omega_0\lambda)\ot\nabla(\lambda).$$
\label{thm:REAgoodfiltration}
\end{theorem}
\begin{proof}
In the parallel setting of algebraic groups defined in characteristic $p$, this is proved in \cite[Proposition 4.20]{JantzenRepresentations}.  Following the discussion in \cite{ParadowskiFiltrations}, the proof is identical for the quantum group at a root of unity. 
\end{proof}

\subsection{Quantum moment maps}
\label{sect:quantummomentmaps}

Let $H$ be a Hopf algebra over a commutative ring $k$ and $H'$ a $k$-algebra equipped with an $H$-comodule and an $H$-module structure which preserve the algebra structure. In the paper we will be interested in the following three cases:
\begin{enumerate}
\item $k = \Rt$, $H=\ULR(\g)$ and $H'=\OR(G)$.

\item $k=\C$, $H=\UqL(\g)$ and $H'=\Oq(G)$.

\item $k=\C$, $H=\uq(\g)$ and $H'=o_q(G; \xi)$, where $\xi\in G$ is a point in the center.
\end{enumerate}

\begin{definition}
Let $A$ be an $H$-module algebra. A map $\mu\colon H'\rightarrow A$ of $H$-module algebras is a quantum moment map if for every $h\in H'$ and $a\in A$ we have
\begin{equation}
\mu(h)a = (h_{(1)}\triangleright a)\mu(h_{(2)}),
\label{eq:quantummomentmap}
\end{equation}
where $h_{(1)}\otimes h_{(2)}\in H\otimes H'$ is the $H$-coaction on $h\in H'$.
\label{def:quantummomentmap}
\end{definition}

\begin{rmk}
In the case $H=\ULR(\g)$, $H'=\OR(G)$ and $A\in\RepR(G)$ we may rephrase the moment map condition as the commutativity of the diagram
\[
\xymatrix{
\OR(G)\otimes A_\Rt \ar^-{\mu\otimes \id}[r] \ar^{\tau_A}[dd] & A_\Rt\otimes A_\Rt \ar^{m}[dr] & \\
&& A_\Rt \\
A_\Rt \otimes \OR(G) \ar^-{\id\otimes \mu}[r] & A_\Rt \otimes A_\Rt \ar_{m}[ur] &
}
\]
as explained in \cite[Theorem 3.10]{SafronovQuantumMoment}.
\label{rmk:momentmapfieldgoal}
\end{rmk}

\begin{rmk} The first notion of quantum moment maps was due to Lu \cite{LuQuantum}, and is the case $H'=H$ of Definition \ref{def:quantummomentmap}.  The notion was generalized in \cite[Section 1.5]{VaragnoloDoubleaffineHecke2010}, to allow $H$ to be replaced by its coideal subalgebra $H'$ -- the main example of interest there was the coideal subalgebra of ad-locally finite elements.

We take the viewpoint that the natural codomain of a quantum moment map is the braided dual algebra $H'=\OR(G)$, or its specialization $\Oq(G)$.  At generic parameters, $\OK(G)$ is indeed identified via the Rosso isomorphism with the subalgebra of locally finite elements, and so Varagnolo--Vasserot's notion suffices.  At roots of unity, the Rosso homomorphism still exists but is no longer an embedding, hence we do not wish to require $H'$ to be a coideal.
\end{rmk}

We will also need a construction of fusion of quantum moment maps. Suppose $A, B$ are two algebras in $\RepR(G)$. Then $A\otimes B$ is naturally an algebra in $\Repq(G)$ with the multiplication
\begin{equation}
(A\otimes B)\otimes (A\otimes B)\xrightarrow{\id\otimes \sigma^{-1}_{A, B}\otimes \id}(A\otimes A)\otimes (B\otimes B)\xrightarrow{m_A\otimes m_B} A\otimes B,
\label{eq:braidedproduct}
\end{equation}
where $\sigma_{A, B}$ is the braiding in $\RepR(G)$. The tensor product of algebras equipped with quantum moment maps also carries a natural quantum moment map, as we show in the following:

\begin{prop}
Suppose $A$ and $B$ are two algebras in $\RepR(G)$ equipped with quantum moment maps $\mu_A\colon \OR(G)\rightarrow A$ and $\mu_B\colon \OR(G)\rightarrow B$. Then
\[\OR(G)\xrightarrow{\Delta} \OR(G)\otimes \OR(G)\xrightarrow{\mu_A\otimes \mu_B} A\otimes B\]
is a quantum moment map for $A\otimes B$.
\label{prop:momentmapfusion}
\end{prop}
\begin{proof}
We can represent equation \eqref{eq:quantummomentmap} following Remark \ref{rmk:momentmapfieldgoal} diagrammatically as follows.  In place of an arbitrary $h\in\OR(G)$, we write $V^*\ot V$, and in place of the element $a$ in $A$ we consider the entire algebra $A$.  With this notation, the moment map equation \eqref{eq:quantummomentmap} reads:
\[
	\mu(h) (a) \quad = \quad\tik{
		\node[label=above:$A$] at (.5,0.5) {};
		\straight{0.5}{-0.5}
		\ap{0}{0}
		\lmove{-1}{1}
		\coupon{-1.5}{2}{1}{$\mu$}\straight{-1.5}{3}\straight{-0.5}{3}\straight{1}{1}\straight{1}{2}\straight{1}{3}
		\node[label=below:$V^*$] at (-1.5,-4) {};
		\node[label=below:$V$] at (-.5,-4) {};
		\node[label=below:$A$] at (1,-4) {};
	}
	\quad=\quad
	\tik{
		\node[label=above:$A$] at (.5,0.5) {};
		\straight{0.5}{-0.5}
		\ap{0}{0}
		\rmove{1}{1}
		\coupon{1.5}{2}{1}{$\mu$}
		\straight{0}{1}\straight{0}{2}\fieldgoal{0}{3}
		\straight{0}{4}\straight{1}{4}\straight{2.5}{4};
		\node[label=below:$V^*$] at (0,-5) {};
		\node[label=below:$V$] at (1,-5) {};
		\node[label=below:$A$] at (2.5,-5) {};
	} \quad=\quad (h_{(1)}\triangleright a) \mu(h_{(2)}).
\]

We now compute:
\begin{align*}
	(h_{(1)}\triangleright (a\ot b)) \mu(h_{(2)})\quad = \quad \tik{
		\node[label=above:$A$] at (1,2.5) {};
		\node[label=above:$B$] at (3,2.5) {};
		\straight{1}{-2.5}\straight{3}{-2.5}
		\ap{.5}{-2}
		\ap{2.5}{-2}
		\negative{1.5}{-1}\lmovesm{0}{-1}  \rmove{3.5}{-1}
		\straight{1.5}{0}\straight{1.5}{1}\fieldgoal{1.5}{2}
		\straight{0}{0}\straight{0}{1}\straight{0}{2}\fieldgoal{0}{3}
		\rmove{2}{1}\lmove{4}{1}\up{3}{1}
		\coupon{2}{0}{1}{$\mu$}
		\coupon{4}{0}{1}{$\mu$}
		\straight{0}{4}\straight{1}{4}\straight{2.5}{4}\straight{4}{3}\straight{4}{4};
		\node[label=below:$V^*$] at (0,-5) {};
		\node[label=below:$V$] at (1,-5) {};
		\node[label=below:$A$] at (2.5,-5) {};
		\node[label=below:$B$] at (4,-5) {};
	}
\quad = \quad
\tik{
		\node[label=above:$A$] at (2,1.5) {};
		\node[label=above:$B$] at (4,1.5) {};
		\straight{2}{-1.5}\straight{4}{-1.5}
		\ap{1.5}{-1}
		\ap{3.5}{-1}
		\straight{3.5}{0}\straight{3.5}{1}\lmove{3}{2}\rmovesm{3.5}{2}
		\straight{1.5}{0}\straight{1.5}{1}\straight{1.5}{2}\straight{2}{2}\fieldgoalsm{.5}{3}
		\straight{2}{1}\lmove{4}{1}\up{3}{1}
		\coupon{2}{0}{1}{$\mu$}
		\coupon{4}{0}{1}{$\mu$}
		\straight{0.5}{4}\straight{1.5}{4}\straight{3}{4}\straight{4}{3}\straight{4}{4};
		\node[label=below:$V^*$] at (0.5,-5) {};
		\node[label=below:$V$] at (1.5,-5) {};
		\node[label=below:$A$] at (3,-5) {};
		\node[label=below:$B$] at (4,-5) {};
	}\\
\qquad\qquad=\quad
\tik{
		\node[label=above:$A$] at (2,1.5) {};
		\node[label=above:$B$] at (4,1.5) {};
		\straight{2}{-1.5}\straight{4}{-1.5}
		\ap{1.5}{-1}
		\ap{3.5}{-1}
		\straight{4}{1}\up{2}{1}
		\straight{2.5}{0}\straight{2.5}{1}\lmovelg{2}{2}\rmove{2.5}{2}
		\straight{1}{1}\straight{1}{2}
		\coupon{1}{0}{1}{$\mu$}
		\coupon{3}{0}{1}{$\mu$}
		\straight{4.5}{0}\straight{4.5}{1}\straight{4.5}{2}
		\node[label=below:$V^*$] at (1,-3) {};
		\node[label=below:$V$] at (2,-3) {};
		\node[label=below:$A$] at (3.5,-3) {};
		\node[label=below:$B$] at (4.5,-3) {};
	}
\quad=\quad
\tik{
		\node[label=above:$A$] at (3,2.5) {};
		\node[label=above:$B$] at (5,2.5) {};
		\straight{3}{-2.5}\straight{5}{-2.5}
		\ap{2.5}{-2}
		\ap{4.5}{-2}
		\negative{3.5}{-1}\lmove{1.5}{-1}
		\straight{4.5}{0}\straight{4.5}{1}\straight{4.5}{2}\lmovelg{2}{2}
		\straight{1}{1}\straight{1}{2}
		\straight{4}{1}\up{2}{1}
		\coupon{1}{0}{1}{$\mu$}
		\coupon{3}{0}{1}{$\mu$}
		\straight{5.5}{-1}\straight{5.5}{0}\straight{5.5}{1}\straight{5.5}{2}
		\node[label=below:$V^*$] at (1,-3) {};
		\node[label=below:$V$] at (2,-3) {};
		\node[label=below:$A$] at (4.5,-3) {};
		\node[label=below:$B$] at (5.5,-3) {};
	}\quad = \quad \mu(h) (a\ot b).
\end{align*}
\end{proof}

As for classical moment maps, quantum moment maps are used to construct a Hamiltonian reduction of $A$. Fix a map of $H$-module algebras $\xi\colon H'\rightarrow k$ and denote $I'=\ker(\xi)$, the corresponding $H$-invariant two-sided ideal.

\begin{definition}
Let $A$ be an $H$-module algebra equipped with a quantum moment map $\mu\colon H'\rightarrow A$. The quantum Hamiltonian reduction is
\[A \dS  H = (A / A\mu(I'))^H.\]
\end{definition}

Note that $I_q=A\mu(I')\subset A$ is only a left ideal, so $A / A\mu(I')$ is merely an $A$-module. Nevertheless, $A\dS H$ carries a natural algebra structure that we explain presently. For $a \in A$, we write $[a]$ for the image of $a$ in $A/I_q$.

\begin{prop}
 Let $a, b \in A$, and suppose $[a]$ and $[b]$ lie in $A\dS H$.  Then  $[ab]$ also lies in $A\dS H$. There is a well-defined algebra structure on $A\dS H$ given by $$A\dS H \ot A\dS H \to A\dS H, \qquad [a]\ot [b] \mapsto [ab].$$
\label{prop:qhamreduction}
\end{prop}
\begin{proof}
To show that the element $[ab] \in A/I_q$ lies in $A\dS H$, first note that, for any $h \in H$, we have that $h \triangleright a \in \epsilon(h)a + I_q$, and similarly for $b$. Consequently, 
\begin{equation*}\label{eq:hab} h \triangleright ab = (h_{(1)}\triangleright a)(h_{(2)}\triangleright b) \in (\epsilon(h_{(1)}) a + I_q)(\epsilon(h_{(2)})b + I_q) = \epsilon(h) ab +  a I_q +  I_q b +  I_q.\end{equation*}
Since $I_q$ is a left ideal, we have that $a I_q \subseteq I_q$. We show below that $I_q b \subseteq I_q$.  From this  it follows that $h\triangleright ab \in \epsilon(h) ab  +  I_q$, and so $[ab]$ lies in $A\dS H$. 

Next, we prove that $I_q b \subseteq I_q$. It suffices to show that $\mu(h)b \in I_q$ for all $h \in I^\prime$. Fix $h \in I^\prime$ and recall the  moment map equation: $\mu(h)b = (h_{(1)}\triangleright b)\mu(h_{(2)}).$ The fact that $[b]$ lies in $A\dS H$ implies that $h_{(1)} \triangleright b  = \epsilon(h_{(1)} ) b + x$ for some $x \in I_q$. We write $x = \sum_i x_i \mu(h_i)$ for some $x_i \in A$ and $h_i \in I^\prime$. We have:
\begin{align*} \mu(h)b &= (h_{(1)}\triangleright b)\mu(h_{(2)}) =  \epsilon(h_{(1)} ) b \mu(h_{(2)})  + \sum_i x_i \mu(h_i)\mu(h_{(2)}) \\  &= b \mu(h)  + \sum_i x_i \mu(h_i h_{(2)}). \end{align*}
In the last expression, the first term lies in $I_q$ since $h \in I^\prime$, and the second term lies in $I_q$ since each $h_i$ is in the two-sided ideal $I^\prime$. We conclude that $I_q b \subseteq I_q$. 

It remains to show that the product map is well-defined, that is, that the element $[ab]$ is independent of the lifts of $[a], [b] \in A/I_q$ to $A$. Indeed, if $a^\prime, b^\prime \in A$ are such that $[a] = [a^\prime]$ and $[b]  = [b^\prime]$, then $a^\prime b^\prime$ is contained in $(a + I_q)(b + I_q)$ $ = ab + a I_q + I_q b + I_q = ab + I_q,$ and so $[a^\prime b^\prime] = [ab].$
\end{proof}

\subsection{Frobenius Poisson orders and their quantum Hamiltonian reduction}
\label{subsec:Frobeniuspairs}

In this section we combine Poisson and associative algebras into the notion of a Poisson order and, correspondingly, classical and quantum moment maps into the notion of a Hamiltonian Frobenius Poisson order.  Finally, we explain that the quantum Hamiltonian reduction of a Hamiltonian Frobenius Poisson order carries naturally the structure of a Poisson order.

\begin{definition}
A Poisson order is given by the following data:
\begin{itemize}
\item An algebra $A$.

\item A commutative algebra $Z$ equipped with a Poisson structure. We denote $\cZ=\Spec Z$.

\item A central embedding $Z\subset A$ such that $A$ is finitely generated as a $Z$-module.

\item A map $D\colon Z\rightarrow \Der(A)$ which satisfies
\[D(z_1)(z_2) = \{z_1, z_2\}\]
for every $z_1, z_2\in Z$.
\end{itemize}
\end{definition}

\begin{rmk}
We will identify Poisson order structures where $D$ differs by inner derivations. See Proposition \ref{prop:FrobeniusPoissonOrderDegeneration} for our source of Poisson orders where the corresponding structure $D$ is only defined modulo inner derivations.
\end{rmk}

We have the following important theorem about Poisson orders, see \cite[Theorem 4.2]{BrownPoissonorderssymplectic2003a}.

\begin{theorem}
Suppose $(A, Z=\cO(\cZ))$ is a Poisson order with $\cZ$ smooth. Moreover, assume $\cZ_0\subset \cZ$ is an open symplectic leaf. Then $A|_{\cZ_0}$ is locally free and its fibers at any two closed points are isomorphic as algebras.
\label{thm:browngordon}
\end{theorem}

\begin{definition}
A Frobenius Poisson order is given by the following data:
\begin{itemize}
\item An algebra $A_q\in \Repq(G)$.

\item A Poisson algebra $Z$ equipped with a Poisson $G$-action. As before, we have a Poisson $G$-variety $\cZ=\Spec Z$.

\item A central embedding $\Fr^*(Z)\subset A_q$ in $\Repq(G)$.

\item A linear map $D\colon Z\rightarrow \Der(A_q)$ giving $(A_q, Z)$ the structure of a Poisson order. \end{itemize}
\end{definition}

\begin{rmk}
Since the $R$-matrix on $\UqL(\g)$ is expressed via the generators of the small quantum group, the subalgebra $\Fr^*(Z)\subset A_q$ is central in $\Repq(G)$ if, and only if, it is central in the category of vector spaces.
\end{rmk}

We may now define the notion of a moment map for a Frobenius Poisson order.

\begin{definition}
Let $(A_q, Z)$ be a Frobenius Poisson order. We say it is weakly $G$-Hamiltonian if there is a quantum moment map $\mu_q\colon \Oq(G)\rightarrow A_q$ (Definition \ref{def:quantummomentmap}) and a classical moment map $\mu\colon \cZ\rightarrow G$ (Definition \ref{def:classicalmomentmap}) such that the diagram
\[
\xymatrix{
Z \ar^{i}[r] & A_q \\
\cO(G) \ar^{\mu^*}[u] \ar[r] & \Oq(G) \ar^{\mu_q}[u]
}
\]
commutes.
\end{definition}

\begin{definition}
A $G$-Hamiltonian Frobenius Poisson order is a weakly $G$-Hamiltonian Frobenius Poisson order which satisfies the following properties:
\begin{enumerate}
\item $D(z)(\mu_q(h)) = 0$ for every $z\in Z^G$ and $h\in\Oq(G)$.

\item $D(z)(-)\colon A_q\rightarrow A_q$ is $\UqL(\g)$-equivariant for every $z\in Z^G$.
\end{enumerate}
\label{def:HamiltonianFrobeniusPoissonOrder}
\end{definition}

An important observation for us is that, given a weakly $G$-Hamiltonian Frobenius Poisson order, its quantum Hamiltonian reduction contains a natural central subalgebra, which identifies with the classical Hamiltonian reduction of $\cZ$ by $G$.

In more details, choose a character $\xi_q$ of $\Oq(G)$, and denote its restriction to $\O(G)$ by $\xi$. Then $\xi_q$ induces a character of $o_q(G; \xi) = \Oq(G)\otimes_{\cO(G)} \C$. Denote by $I\subset Z$ the moment map ideal consisting of elements of the form $z\mu(h)$ where $z\in Z$ and $h\colon G\rightarrow \C$ which vanishes at $\xi$. Denote by $I_q\subset \Aq$ the quantum moment map ideal consisting of elements of the form $a\mu_q(h)$ where $a\in \Aq$ and $h\in\Oq(G)$ such that $\xi_q(h)=0$. We may identify
\begin{equation}
\Aq \dS  \UqL(\g) = (\Aq/I_q)^{\UqL(\g)}\cong (((\Aq/I)\otimes_{o_q(G;\xi)} \C)^{\uq(\g)})^G.
\label{eq:reductioninstages}
\end{equation}
In other words, we may perform the quantum Hamiltonian reduction of $\Aq$ in stages, first with respect to $\uq(\g)$ and then with respect to $G$.

$\Aq/I$ defines a sheaf of algebras over $\mu^{-1}(\xi)$. Denote by $\cE^\xi$ the sheaf of algebras over $\mu^{-1}(\xi)$ obtained as the quantum Hamiltonian reduction
\[\cE^\xi = (\Aq/I) \dS  \uq(\g).\]
Denote by $\cA^\xi$ the sheaf of algebras on $\cZ\dS G$ defined by $\Aq \dS  \UqL(\g)$. By \eqref{eq:reductioninstages} we have an isomorphism of algebras
\[\cA^\xi\cong \pi_*(\cE^\xi)^G\]
over $\cZ\dS G$.

Furthermore, we obtain a Poisson structure on $Z\dS G$ by Proposition \ref{prop:poissonreduction}, an algebra structure on $A_q\dS \UqL(\g)$ by Proposition \ref{prop:qhamreduction} and a central embedding $Z\dS G \subset A_q\dS \UqL(\g)$. We will now show that this can be enhanced to a structure of a Poisson order.

\begin{prop}
Suppose $(A_q, Z)$ is a $G$-Hamiltonian Frobenius Poisson order where $Z$ is Noetherian. Then $(A_q\dS \UqL(\g), Z\dS G)$ is a Poisson order.
\label{prop:FrobeniusPoissonReduction}
\end{prop}
\begin{proof}
Let $I\subset \cO(\cZ)$ be the ideal of functions vanishing on $\mu^{-1}(\xi)$ and $I_q\subset A_q$ the left ideal of elements of the form $a\mu_q(h)$ with $a\in A_q$ and $h\in \Oq(G)$ with $\xi_q(h)=0$.

Recall from Proposition \ref{prop:poissonreduction} that the Poisson structure on $\cO(\cZ\dS G)$ is uniquely characterized by the property that $\cO(\cZ\dS G)\cong Z^G / I^G\leftarrow Z^G\subset Z$ is a diagram of Poisson algebras.

Consider $z\in Z^G$ and $a\mu_q(h)\in I_q$. Then
\[D(z)(a\mu_q(h)) = D(z)(a)\mu_q(h) + aD(z)(\mu_q(h)) = D(z)(a)\mu_q(h).\]
Therefore, $D(z)(I_q)\subset I_q$ and hence $D$ descends to a map
\[D\colon Z^G\otimes A_q/I_q\longrightarrow A_q/I_q.\]

Consider an element $a\in A_q$ and $[a]\in (A_q/I_q)^{\UqL(\g)}$ whose images in $A_q/I_q$ coincide. Then for $h\in \UqL(\g)$ we have $h\triangleright a = \epsilon(h)a + I_q$. For $z\in Z^G$ we therefore get
\begin{align*}
h\triangleright D(z)(a) &= D(z)(h\triangleright a) \\
&= \epsilon(h) D(z)(a) + D(z)(I_q) \\
&= \epsilon(h) D(z)(a) + I_q.
\end{align*}
In other words, $D(z)(a)$ descends to $(A_q/I_q)^{\UqL(\g)}$ and so $D$ gives a linear map
\[D\colon Z^G\otimes (A_q/I_q)^{\UqL(\g)}\longrightarrow (A_q/I_q)^{\UqL(\g)}.\]

By construction $D$ defines a linear map $D\colon Z^G\rightarrow \Der(A_q\dS \UqL(\g))$ and restricts to the Poisson bracket $Z^G\otimes Z^G/I^G\rightarrow Z^G/I^G$ on $Z^G/I^G\cong (Z/I)^G\subset A_q \dS  \UqL(\g)$. Since $Z^G\rightarrow Z^G/I^G$ is surjective, this defines the required structure of a Poisson order.

By assumption $A$ is finitely generated over $Z$. Therefore, $A/I_q$ is finitely generated over $Z/I$. Since $Z/I$ is Noetherian, $(A/I_q)^{\uq(\g)}\subset A/I_q$ is also finitely generated over $Z/I$. Since $G$ is reductive, we obtain that $((A/I_q)^{\uq(\g)})^G=A\dS \UqL(\g)$ is finitely generated over $(Z/I)^G=Z\dS G$.
\end{proof}

\subsection{Classical degenerations of quantum algebras}\label{subsec:degenqalgebras}

The notion of a $G$-Hamiltonian Frobenius Poisson order requires a strong compatibility between several structures.  In this section we show that these compatibilities arise very naturally when specializing algebras from $\RepR(G)$ to $\Repq(G)$. We begin by showing that degenerating $\ULR(\g)$ at a root of unity induces extra structures on $G$, such as the Poisson-Lie structure (see \cite[Section 8]{DeConciniQuantumFunctionAlgebra}). Recall from \cite[Chapter 32]{LusztigIntroductionQuantumGroups2010} that to construct the braiding on $\RepR(G)$ we need to fix a bilinear pairing on $(-, -)$ on the character lattice. We normalize it so that the short roots have square length $2$.

Let $V_\Rt, W_\Rt\in\RepR(G)$ be two representations flat over $\Rt$ and let
\[R_{V_\Rt, W_\Rt}\colon V_\Rt\otimes W_\Rt\longrightarrow V_\Rt\otimes W_\Rt\]
be the isomorphism given by the braiding in $\RepR(G)$ precomposed with the tensor flip. Let $\pi\colon V_\Rt\rightarrow V_q$ be the specialization map taking $t$ to $q$ and similarly for $W_\Rt$. Moreover, assume we have two representations $V, W\in\Rep(G)$ and embeddings $\Fr^*(V)\subset V_q$ and $\Fr^*(W)\subset W_q$.

Choose an element $v\otimes w\in V_\Rt\otimes W_\Rt$ such that $\pi(v)\otimes \pi(w)\in V\otimes W\subset V_q\otimes W_q$. Then
\[\left.R_{V_\Rt, W_\Rt}(v\otimes w)\right|_{t=q} = \pi(v)\otimes \pi(w)\]
since $\Fr^*\colon \Rep(G)\rightarrow \Repq(G)$ is braided monoidal and the braiding on $\Rep(G)$ is the tensor flip.

\begin{prop}
With the above assumptions we have
\[r(\pi(v)\otimes \pi(w)) = \pi\left(\frac{R_{V_\Rt, W_\Rt}(v\otimes w) - v\otimes w}{t-q}\right),\]
where $r = -2\ell^2 q^{-1} r_{std}\in \g\otimes \g$ and $r_{std}$ is the standard $r$-matrix from Example \ref{ex:standardrmatrix} defined by the bilinear pairing $\ell^2(-, -)$ on $\g$.
\label{prop:RMatrixDegeneration}
\end{prop}
\begin{proof}
Let us recall the precise construction of the braiding on $\RepR(G)$ from \cite[Chapter 32]{LusztigIntroductionQuantumGroups2010}. Fix a reduced expression for the longest element of the Weyl group which gives rise to an ordering of the set $\Delta_+$ of positive roots. For a positive root $\alpha\in \Delta_+$ let $d_\alpha = (\alpha, \alpha)/2$. Consider the quasi $R$-matrix (see \cite[Chapter 4]{LusztigIntroductionQuantumGroups2010}, \cite{LevendorskiiR} and \cite{KirillovReshetikhin})
\[\Theta = \prod_{\alpha\in \Delta_+} \sum_n (-1)^n t^{-n(n-1)d_\alpha/2}\{n\}_\alpha F_\alpha^{(n)}\otimes E_\alpha^{(n)},\]
where
\[\{n\}_\alpha = \prod_{a=1}^n(t^{ad_\alpha} - t^{-ad_\alpha}).\]
If we assume $v\in V_\Rt, w\in W_\Rt$ have weights $\mu$ and $\nu$ respectively, then
\[R_{V_\Rt, W_\Rt}(v\otimes w) = t^{-(\mu, \nu)} \Theta(v\otimes w).\]
By our assumption on the group $G$, under the quantum Frobenius map the character lattice of $G$ used in the definition of $\Rep(G)$ is the $\ell$-scaled character lattice used in the definition of $\Repq(G)$. Therefore, $q^{-(\mu, \nu)} = 1$.

Since $V,W\in\Rep(G)$, the only terms in $\Theta$ not vanishing to order $(t-q)$ occur for $n=\ell$. Therefore, the linear term in $(t-q)$ coming from $\Theta$ is
\[-\frac{2\ell^2}{q}\sum_{\alpha\in\Delta_+} \left(d_\alpha f_\alpha\otimes e_\alpha\right)(\pi(v)\otimes \pi(w)).\]

The linear term in $(t-q)$ coming from $t^{-(\mu, \nu)}$ is $-q^{-1} (\mu, \nu)$. Combining the two terms we exactly recover the formula for the standard $r$-matrix $r_{std}$.
\end{proof}

Next, we show that $\g$ carries a natural structure of a factorizable Lie bialgebra. Let $p\colon \ULR(\g)\rightarrow \UqL(\g)$ be the specialization map taking $t$ to $q$. Since $\ULR(\g)$ is free as an $\Rt$-module, we may choose an isomorphism $\ULR(\g)\cong \UqL(\g)\otimes_\C \Rt$ which is the identity at $t=q$. For an element $x\in \ULR(\g)$ such that $\Fr(p(x))\in \g\subset U(\g)$, we have that
\[\delta(\pi(x)) = \Fr\left(p\left(\frac{\Delta(x) - \Delta^{\op}(x)}{t-q}\right)\right)\]
defines a Lie cobracket on $\g$, see e.g. \cite[Proposition 9.1]{EtingofLecturesQuantumGroups}. Since the braiding on $\RepR(G)$ is an isomorphism of $\ULR(\g)$-modules, by Proposition \ref{prop:RMatrixDegeneration} we get
\[\delta(\pi(x)) = [r, \pi(x)\otimes 1 + 1\otimes \pi(x)],\]
i.e. $\delta$ defines a factorizable Lie bialgebra structure. In particular, we obtain a factorizable Poisson-Lie structure on $G$.

Now suppose $A_\Rt$ is an $\Rt$-algebra. Denote by $\pi\colon A_\Rt\rightarrow \Aq$ the evaluation at $q$.

\begin{prop}
Suppose $A_\Rt$ is flat as an $\Rt$-module and $Z\subset \Aq$ a central subalgebra. Then there is a linear map $D\colon Z\rightarrow \Der(\Aq)$ determined by the formula
\[D(\pi(z))(\pi(x)) = \pi\left(\frac{zx - xz}{t-q}\right)\]
for every  $z,x\in A_\Rt$ with $\pi(z)\in Z$. It is independent of the lift of $\pi(x)$ and changing the lift of $\pi(z)$ modifies $D$ by an inner derivation. If, in addition, $D(z_1)(z_2)\in Z$ for $z_1,z_2\in Z$, then $(\Aq, \Spec(Z))$ is a Poisson order.
\label{prop:PoissonOrderDegeneration}
\end{prop}

The condition $D(z)(Z)\subset Z$ in the above Proposition is automatic if $Z$ is the whole center as shown by Hayashi \cite[Proposition 2.6 (2)]{HayashiSugawara}:

\begin{lemma}\label{lem:Hayashi}
Suppose $A_\Rt$ is a flat $\Rt$-algebra and $Z=Z(\Aq)$ be the center of its specialization at $t=q$. Then the map $D\colon Z\rightarrow \Der(\Aq)$ defined in Proposition \ref{prop:PoissonOrderDegeneration} satisfies $D(z_1)(z_2)\in Z$ for every $z_1, z_2\in Z$.
\label{lm:PoissonOrderAutomatic}
\end{lemma}
\begin{proof}
Consider $z_1, z_2, a\in A_\Rt$ such that $\pi(z_i)\in Z$. We have
\[[a, [z_1, z_2]] = [[a, z_1], z_2] + [z_1, [a, z_2]].\]
Since $\pi(z_i)$ lie in the center of $\Aq$, both of the terms on the right-hand side vanish to the second order at $t=q$. Therefore,
\[[\pi(a), D(z_1)(z_2)] = 0\]
and hence $D(z_1)(z_2)\in Z(\Aq)$.
\end{proof}

\begin{prop}
Suppose $(A_\Rt, Z\subset \Aq)$ satisfy the conditions of Proposition \ref{prop:PoissonOrderDegeneration}. In addition, we make the following assumptions:
\begin{itemize}
\item $A_\Rt$ is an algebra in $\RepR(G)$.

\item $Z\subset \Aq$ is $\UqL(\g)$-equivariant and $\uq(\g)$-invariant.
\end{itemize}

Then $(\Aq, Z)$ is a Frobenius Poisson order.
\label{prop:FrobeniusPoissonOrderDegeneration}
\end{prop}
\begin{proof}
Since $Z$ is $\uq(\g)$-invariant, it becomes an algebra in $\Rep(G)$. So, we have an infinitesimal action map $a\colon \g\rightarrow\Der(Z)$. To show that $\Spec(Z)$ is a Poisson $G$-variety, it is enough to work infinitesimally and show that $Z$ is a Poisson $\g$-algebra, i.e. for every $x\in \g$ and $a,b\in Z$ we have
\[x.\{a, b\} - \{x.a, b\} - \{a, x.b\} = (x_{[1]}.a)(x_{[2]}.b)\]
where $\delta(x) = x_{[1]}\otimes x_{[2]}$ is the Lie cobracket on $\g$.

The fact that $A_\Rt$ is an algebra in $\ULR(\g)$ gives the relation
\[h\triangleright(ab) = h_{(1)}\triangleright a\cdot h_{(2)}\triangleright b\]
for $h\in \ULR(\g)$ and $a,b\in A_\Rt$. Now suppose $\pi(a), \pi(b)\in Z$ and $\Fr(p(h))=x\in\g\subset U(\g)$. Taking the commutator we get
\[h\triangleright[a, b] = h_{(1)}\triangleright a\cdot h_{(2)}\triangleright b - h_{(1)}\triangleright b\cdot h_{(2)}\triangleright a.\]
This relation is trivial modulo $(t-q)$ while the $(t-q)$ term gives
\[x.\{\pi(a), \pi(b)\} = \{x.\pi(a), \pi(b)\} + \{\pi(a), x.\pi(b)\} + (x_{[1]}.\pi(a))(x_{[2]}.\pi(b)).\]
\end{proof}

We will now show that the degeneration at a root of unity is compatible with fusion. Suppose $A_\Rt, B_\Rt$ are two algebras in $\RepR(G)$. Then $A_\Rt\otimes B_\Rt$ is naturally an algebra in $\RepR(G)$. Let $p\colon A_\Rt\rightarrow \Aq$ and $p\colon B_\Rt\rightarrow \Bq$ be the maps given by evaluation at $t=q$.

Suppose $A_\Rt, B_\Rt$ are flat over $\Rt$. Suppose $Z_{\Aq}\subset \Aq$ and $Z_{\Bq}\subset \Bq$ are two central $\UqL(\g)$-equivariant and $\uq(\g)$-invariant subalgebras. Then $Z_{\Aq}\otimes Z_{\Bq}\subset \Aq\otimes \Bq$ is central. So, by Proposition \ref{prop:FrobeniusPoissonOrderDegeneration} we obtain a Frobenius Poisson order structure on $(\Aq\otimes \Bq, \Spec (Z_{\Aq}) \times \Spec (Z_{\Bq}))$.

\begin{prop}
The Poisson $G$-variety structure on $\Spec (Z_\Aq) \times \Spec (Z_\Bq)$ obtained by degenerating the algebra $A_\Rt\otimes B_\Rt$ coincides with the fusion of the Poisson $G$-varieties $\Spec (Z_\Aq)$ and $\Spec (Z_\Bq)$.
\label{prop:fusiondegeneration}
\end{prop}
\begin{proof}
By definition, the Poisson structure on $Z_\Aq\otimes Z_\Bq$ is given as follows. Suppose $\pi(z_1), \pi(z_2)\in Z_\Aq$ and $\pi(w_1), \pi(w_2)\in Z_\Bq$. Then
\[\{\pi(z_1)\otimes \pi(w_1), \pi(z_2)\otimes \pi(w_2)\} = \pi\left(\frac{(z_1\otimes w_1)(z_2\otimes w_2) - (z_2\otimes w_2)(z_1\otimes w_1)}{t-q}\right).\]

The product on the right-hand side is the braided tensor product of algebras given by equation \eqref{eq:braidedproduct} defined using the braiding on $\RepR(G)$. Its first-order term in $(t-q)$ therefore can be read off from Proposition \ref{prop:RMatrixDegeneration} and we get
\begin{align*}
\{\pi(z_1)\otimes \pi(w_1), \pi(z_2)\otimes \pi(w_2)\} &= \{\pi(z_1), \pi(z_2)\}\otimes \pi(w_1)\pi(w_2) \\
&+ \pi(z_1)\pi(z_2)\otimes \{\pi(w_1), \pi(w_2)\} \\
&- \pi(z_1)\pi(w_2) \cdot a(r)(\pi(z_2)\otimes \pi(w_1)) \\
&+ \pi(z_2)\pi(w_1) \cdot a(r)(\pi(z_1)\otimes \pi(w_2))
\end{align*}
which coincides with the fusion as defined in Proposition \ref{prop:Poissonfusion}.
\end{proof}

Finally, a degeneration of quantum moment maps gives rise to classical moment maps.

\begin{prop}
Suppose $(A_\Rt, Z\subset \Aq)$ satisfy the conditions of Proposition \ref{prop:FrobeniusPoissonOrderDegeneration}. In addition, suppose there is a quantum moment map $\mu_\Rt\colon \OR(G)\rightarrow A_\Rt$ and a $G$-equivariant map $\mu\colon \cZ\rightarrow G$ such that the diagram
\[
\xymatrix{
\Oq(G) \ar^-{\mu_q}[r] & \Aq \\
\cO(G) \ar^-{\mu^*}[r] \ar[u] & Z\ar[u] 
}
\]
commutes. Then $\mu$ is a moment map for the Poisson $G$-action on $\cZ$ and $(\Aq, Z)$ is a weakly $G$-Hamiltonian Frobenius Poisson order.
\label{prop:HamiltonianFrobeniusPoissonOrderDegeneration}
\end{prop}
\begin{proof}
Let us first check that $\mu$ satisfies the classical moment map condition \eqref{eq:classicalmomentmap}. Consider $z\in A_\Rt$ and $h\in \OR(G)$. The quantum moment map equation gives
\[\mu_\Rt(h)z = (h_{(1)}\triangleright z)\mu_\Rt(h_{(2)})\]
which is equivalent to
\[[\mu_\Rt(h), z] = (h_{(1)}\triangleright z - \epsilon(h_{(1)}) z)\mu_\Rt(h_{(2)}).\]
If we assume $\pi(z)\in Z\subset \Aq$ and $\pi(h)\in\cO(G)\subset \Oq(G)$, the above equation is trivial modulo $(t-q)$ while the linear $(t-q)$ term gives
\[\{\mu(\pi(h)), \pi(z)\} = \sum_i \mu(\tilde{a}(e^i).\pi(h)) e_i.\pi(z),\]
where we use the $\mu(\pi(h)) = \mu_q(\pi(h))$ by assumption and that the first-order term in $h_{(1)}\triangleright z - \epsilon(h_{(1)}) z$ gives the $\g^*$-action on $\cO(G)$ (see \cite[Theorem 4.27]{SafronovQuantumMoment} for a similar degeneration at $q=1$).
\end{proof}

Observe that the classical moment map $\mu\colon \cZ\rightarrow G$ in the above statement is unique if it exists and its existence boils down to the following property of the quantum moment map: $\mu_q(\cO(G))\subset Z$. We will now show that if $Z$ is the whole center, this condition is automatic.

\begin{lemma}
Suppose $\Aq$ is an algebra in $\Repq(G)$ equipped with a quantum moment map $\mu_q\colon \Oq(G)\rightarrow \Aq$. Then $\mu_q(\cO(G))\subset Z(\Aq)$.
\label{lm:HamiltonianPoissonOrderAutomatic}
\end{lemma}
\begin{proof}
Recall that for $W\in \Repq(G)$ we have the field-goal isomorphism $\tau_W\colon \Oq(G)\otimes W\rightarrow W\otimes \Oq(G)$ given by \eqref{eq:fieldgoal}. The subalgebra $\cO(G)\subset \Oq(G)$ is generated by $V^*\otimes V$ where $V$ lies in the M\"{u}ger center of $\Repq(G)$. Therefore, $\tau_W$ on $\cO(G)\subset \Oq(G)$ restricts to the braiding $\cO(G)\otimes W\rightarrow W\otimes \cO(G)$ in $\Repq(G)$. Since $\cO(G)$ lies in $\uq(\g)$-invariants of $\Oq(G)$, the above braiding in $\Repq(G)$ coincides with the tensor flip. Then the claim follows immediately from the interpretation of the quantum moment map equation using the field-goal isomorphism given by Remark \ref{rmk:momentmapfieldgoal}.
\end{proof}

We end with an important result which explains when a weakly $G$-Hamiltonian Frobenius Poisson order is in fact $G$-Hamiltonian.

\begin{prop}
Suppose $(A_\Rt, Z)$ satisfy the conditions of Proposition \ref{prop:HamiltonianFrobeniusPoissonOrderDegeneration}. In addition, suppose that the morphism $(A_\Rt)^{\ULR(\g)}\rightarrow (\Aq)^{\UqL(\g)}$ is surjective. Then $(A_\Rt, Z)$ is a $G$-Hamiltonian Frobenius Poisson order.
\label{prop:StrongHamiltonianOrder}
\end{prop}
\begin{proof}
We have $Z^G\subset (\Aq)^{\UqL(\g)}$ and so we may consider lifts of elements of $Z^G$ to $\ULR(\g)$-invariant elements of $A_\Rt$ in the definition of $D$. Now consider $z\in (A_\Rt)^{\ULR(\g)}$ such that $\pi(z)\in Z^G$, $a\in A$ and $h\in \ULR(\g)$. Then
\[h\triangleright [z, a] = [z, h\triangleright a].\]
This shows that $D(\pi(z))(-)$ is $\UqL(\g)$-equivariant.

For $h\in\OR(G)$ we have
\[\mu_\Rt(h) z = z\mu_\Rt(h)\]
by the quantum moment map equation. Therefore, $D(\pi(z))(\mu_q(\pi(h))) = 0$.
\end{proof}

\subsection{Hamiltonian reduction of matrix algebras}
Before we state our main result on Azumaya algebras, we will prove two preliminary claims on matrix algebras and their reductions.

Let $k$ be a commutative ring, $H$ a Hopf $k$-algebra and $H'$ a $k$-algebra together with a structure of an $H$-module and an $H$-comodule as in Section \ref{sect:quantummomentmaps}. The following Proposition is a version of \cite[Lemma 3.4.4]{BezrukavnikovCherednikalgebrasHilbert2006} and \cite[Proposition 1.5.2]{VaragnoloDoubleaffineHecke2010}.

\begin{prop}
Let $R$ be a commutative $k$-algebra and $V$ a finitely generated projective $R$-module. Suppose $\xi\colon H'\rightarrow k$ is a map of $H$-module algebras such that the composite
\[H'\xrightarrow{\Delta} H\otimes H'\xrightarrow{\id\otimes\xi} H\]
is an isomorphism. Suppose $A=\End_R(V)$ is an $H$-module algebra equipped with a quantum moment map $\mu\colon H'\rightarrow \End_R(V)$. Denote by $(V^*)_{H'}$ the module of $\xi$-twisted $H'$-coinvariants. Then we have an isomorphism of $R$-algebras
\[\End_R(V) \dS  H \cong \End_R((V^*)_{H'})^{\op}.\]
\label{prop:matrixreduction}
\end{prop}
\begin{proof}
We begin with the computation of the coinvariants $\End_R(V)\otimes_{H'}k$. The moment map $\mu\colon H'\rightarrow \End_R(V)$ gives $V$ the structure of an $H'$-module. Since we are considering the $H'$-action on $\End_R(V)$ given by the right multiplication, it is given by precomposition with the moment map, i.e. $H'$ acts on the source. Since $V$ is finitely generated projective, we may identify $\End_R(V)\cong \End_R(V^*)^{\op}$, where now the $H'$-action is on the target. Since $V^*$ is also finitely generated projective, $\Hom_R(V^*, -)$ preserves colimits, and hence
\[\End_R(V^*)\otimes_{H'}k\cong \Hom_R(V^*, V^*\otimes_{H'} k) = \Hom_R(V^*, (V^*)_{H'}).\]

For $a\in A\otimes_{H'}k$ the quantum moment map equation gives
\[\mu(h) a = (h_{(1)}\triangleright a)\xi(h_{(2)}).\]
Since the map $H'\rightarrow H$ given by $h\mapsto h_{(1)} \xi(h_{(2)})$ is an isomorphism by assumption, we may identify $H$-invariants with $\xi$-twisted $H'$-invariants where $H'$ acts on $A\otimes_{H'}k$ on the left. The functor $\Hom_R(-, (V^*)_{H'})$ sends coinvariants to invariants, so we conclude:
\[\End_R(V^*)^{\op}\dS H = (\End_R(V^*)^{\op}\otimes_{H'}k)^H\cong \Hom_R((V^*)_{H'}, (V^*)_{H'})^{\op}.\]
\end{proof}

Now suppose $\cC$ is a braided tensor category and $V$ a dualizable object. Then the internal endomorphism algebra $E(V) = V\otimes V^*$ carries a natural algebra structure, where $V^*$ is the left dual of $V$. We then have the following result \cite[Proposition 2.3]{VanOystaeyenBrauerGroupBraided1998}, which states that the braided tensor product of matrix algebras is again a matrix algebra.

\begin{prop}\label{prop:tensorend}
Let $\cC$ be a braided tensor category and $V,W\in\cC$ dualizable objects. Then the morphism
$$\phi = 1 \ot \sigma_{W \ot W^*, V^*} : V \ot W \ot W^* \ot V^* \stackrel{}{\longrightarrow} V \ot V^* \ot W \ot W^*  $$ defines an isomorphism of algebras $E( V \ot W) \simeq E(V) \ot E(W)$. \end{prop}

\subsection{Summary of results}\label{sec:template-for-results}
Let us now assemble all the tools above to state our main general result, a framework for constructing Azumaya algebras via quantum Hamiltonian reduction.  We fix the following data:
\begin{itemize}
\item $A_\Rt$ is an algebra in $\RepR(G)$ flat over $\Rt$. Denote by $\Aq$ its specialization at $t=q$.

\item $\mu_\Rt\colon \OR(G)\rightarrow A_\Rt$ is a quantum moment map.

\item $Z\subset \Aq$ is a central, $\UqL(\g)$-equivariant and $\uq(\g)$-invariant subalgebra.  Set $\cZ=Spec(Z)$. We assume that $\cZ$ is smooth and connected and $\Aq$ is a finitely generated projective $Z$-module.

\item $G_{stab}\subset G$ is a normal subgroup which acts trivially on $\cZ$ and such that the pairing on $\g$ restricts to a nondegenerate pairing on $\g_{stab}\subset \g$. Denote $\overline{G}=G/G_{stab}$.
 \end{itemize}

\begin{lemma}
Suppose $p\in\cZ$ is a closed point such that $A|_p$ is a matrix algebra. Then $Z$ coincides with the center of $\Aq$.
\label{lm:wholecenter}
\end{lemma}
\begin{proof}
By assumption $A$ is a finitely generated projective $Z$-module. The Azumaya condition is open, so there is an open dense subset $\cZ^\circ\subset \cZ$ such that $A|_{\cZ^\circ}$ is Azumaya over $\cZ^\circ$. In particular, the center of $A|_{\cZ^\circ}$ coincides with $\cO(\cZ^\circ)$.

Now suppose $z\in Z(\Aq)$. Then its image in $A|_{\cZ^\circ}$ is also in the center, i.e. it lies in $\cO(\cZ^\circ)$. The image of $z$ under $A\rightarrow A/Z$ is therefore zero generically on $\cZ$ and since $A/Z$ is flat over $Z$, it has to be zero.
\end{proof}

Since $Z=Z(\Aq)$, combining Lemma \ref{lm:PoissonOrderAutomatic} and Proposition \ref{prop:PoissonOrderDegeneration} we obtain the structure of a Poisson order on $(\Aq, Z)$. Moreover, combining Lemma \ref{lm:HamiltonianPoissonOrderAutomatic} and Proposition \ref{prop:HamiltonianFrobeniusPoissonOrderDegeneration} we see that $(\Aq, Z)$ is a weakly $G$-Hamiltonian Frobenius Poisson order with a classical moment map $\mu\colon \cZ\rightarrow G$. We fix the following additional data:
\begin{itemize}
\item $G_0\subset G$ is a subvariety such that the composite $G_0\rightarrow G\rightarrow \overline{G}$ is \'{e}tale and such that the moment map $\mu\colon \cZ\rightarrow G$ factors through $G_0$.

\item $\xi_q\colon \Oq(G)\rightarrow \C$ is a map of $\UqL(\g)$-module algebras. Denote by $\xi\in G$ the point corresponding to the composite $\O(G)\rightarrow \Oq(G)\xrightarrow{\xi}\C$. We assume $\xi$ lies in the intersection of the big cell $\Br\subset G$ and $G_0\subset G$.

\item $U\subset \mu^{-1}(\xi)$ is a non-empty open subset consisting of stable points (i.e, each point lies in a closed $G$-orbit) with stabilizer $G_{stab}$. Let $\cM^\xi\subset \cZ\dS G$ denote the image of $U$ under the projection $\mu^{-1}(\xi)\rightarrow \cZ\dS G$.  Note that by Luna's \'{e}tale slice theorem the projection $U\rightarrow \cM^\xi$ is a $\overline{G}$-torsor (see \cite[Proposition 5.7]{DrezetLuna}).
\end{itemize}

\begin{theorem}
\label{thm:frobtwistedAzumaya}
Consider $A_\Rt, \mu_q, \xi_q, Z, U$ as above and assume that he map $(A_\Rt)^{\ULR(\g)}\rightarrow (\Aq)^{\UqL(\g)}$ is surjective. Then $(A_q,\cZ)$ defines a $G$-Hamiltonian Frobenius Poisson order.  Assume further that:
\begin{enumerate}
\item The Poisson $G$-variety structure on $\cZ$ is nondegenerate.

\item $G_{stab}$ acts trivially on $(\Aq/I_q)^{\uq(\g)}$.

\item The point $p$ from Lemma \ref{lm:wholecenter} lies in $\mu^{-1}(\Br)$.
\end{enumerate}

Then $\Aq$ is a sheaf of Azumaya algebras over $\mu^{-1}(\Br)$ and its Hamiltonian reduction $\cA^\xi|_{\cM^\xi}$ is a sheaf of Azumaya algebras over $\cM^\xi\subset \cZ\dS G$.
\end{theorem}
\begin{proof}
By Proposition \ref{prop:StrongHamiltonianOrder} the pair $(\Aq, Z)$ is a $G$-Hamiltonian Frobenius Poisson order. By Theorem \ref{thm:nondegenerateleaves} $\mu^{-1}(\Br)\subset \cZ$ is an open symplectic leaf. Therefore, by Theorem \ref{thm:browngordon} the sheaf $\Aq$ is Azumaya over $\mu^{-1}(\Br)$.

Denote $\overline{\mu}\colon X\rightarrow G\rightarrow \overline{G}$ which is a moment map for the $\overline{G}$-action on $X$ by Proposition \ref{prop:momentmapstabilizer}. Let $\overline{\xi}\in\overline{G}$ be the image of $\xi$ under $G\rightarrow \overline{G}$. By assumption $\mu^{-1}(\xi)\rightarrow \overline{\mu}^{-1}(\overline{\xi})$ is \'{e}tale. Therefore, applying Proposition \ref{prop:symplecticquotient} to the $\overline{G}$-Hamiltonian reduction we deduce that the $G$-Hamiltonian reduction $\cM^\xi$ is a smooth symplectic variety.

By Proposition \ref{prop:FrobeniusPoissonOrderDegeneration} $\cA^\xi$ defines a Poisson order over $\cZ\dS G$. Since $\cM^\xi$ is symplectic, Theorem \ref{thm:browngordon} implies that to establish an Azumaya property it is enough to establish it generically. Since $\pi\colon U\rightarrow \cM^\xi$ is a $\overline{G}$-torsor, we have $\pi^*(\cA^\xi|_{\cM^\xi})\cong \cE|_U$. Therefore, the Azumaya property of $\cA^\xi$ over $\cM^\xi$ is equivalent to that of $\cE$ over $U$ which is therefore also enough to establish generically.

Let $\xi_q^{-1}$ be the character of $\Oq(G)$ obtained by precomposing $\xi_q$ with the antipode on $\Oq(G)$. Then the unit algebra $\C\in\Repq(G)$ carries a moment map $\xi_q^{-1}\colon \Oq(G)\rightarrow \C$. Let us denote it by $\C_{\xi_q^{-1}}$. Then we may identify the quantum Hamiltonian reduction of $A$ along $\xi_q$ with the quantum Hamiltonian reduction of $A\otimes \C_{\xi_q^{-1}}$ along the trivial character $\epsilon\colon \Oq(G)\rightarrow \C$. In particular, this allows us to assume $\xi_q = \epsilon$ and $\xi=e\in G$ is the identity element.

Choose an \'{e}tale cover $p\colon Y\rightarrow U$ over which $\Aq/I$ splits as $\End(\cV)$ for a vector bundle $\cV$ over $Y$. The pullback $p^*\cE$ is given by the Hamiltonian reduction $\End(\cV) \dS  \uq(\g)$. By Theorem \ref{thm:smallquantumgroupfactorizable} $\uq(\g)$ is a factorizable Hopf algebra, so the Rosso homomorphism $o_q(G; e)\rightarrow \uq(\g)$ is an isomorphism. Therefore, by Proposition \ref{prop:matrixreduction} $\End(\cV) \dS \uq(\g)\cong \End(\cW)$, where $\cW = (\cV)^*_{\uq(\g)}$. The sheaf $\cW$ is coherent over a smooth scheme, so by generic flatness (see \cite[Theor\`{e}me 6.9.1]{EGA42}) it is generically a vector bundle. Therefore, $p^*\cE$ is generically a matrix algebra over $Y$, so $\cE$ is generically an Azumaya algebra over $U$.
\end{proof}

\begin{lemma}\label{lem:BrAzumaya} Suppose $(A_q,Z)$ is a $G$-Hamiltonian Frobenius Poisson order.  Then the Azumaya locus of $A_q$ is contained within $\mu^{-1}(\Br)$.
\end{lemma}

\begin{proof}
Recall from \cite{JordanWhiteREA, JosephLocalfinitenessadjoint1992, KlimykQuantumGroupsTheir1997} that we have an element denoted $K^{-\rho}$ in $\Oq(G)$, which acts diagonalizably on weight vectors $v_\lambda$ via $K^{-\rho}v_\lambda = q^{-\langle \rho,\lambda\rangle}v_\lambda$.  When $q$ is generic, this element corresponds to monomial in $U_q(\mathfrak{t})$ under the Rosso isomorphism, but it is well-defined as an element of $\Oq(G)$ for all $q$.  The element $K^{\ell\rho}$ lies in the central subalgebra $\O(G)\subset \Oq(G)$, where it generates the vanishing ideal of the complement to $\Br\subset G$.  Now let $p\in \Spec(Z)$ be any point in the preimage of the complement $G\setminus\Br$ of the Bruhat cell.  Then $\mu_q(K^{-\rho})$ defines a non-zero element of the fiber, which $q$-commutes with elements of $A$, and therefore defines a non-zero two-sided ideal in the fiber.  This ideal is proper, because $(K^{-\rho})^{\ell}$ is zero by assumption.  Since finite-dimensional matrix algebras cannot contain non-trivial two-sided ideals, we conclude that $p$ cannot lie in the Azumaya locus.
\end{proof}

\section{Quantum character varieties}\label{sec:CharVar}

In this section we implement the program of of Section \ref{sec:frobenius} to show that quantum character varieties of closed surfaces form Azumaya algebras over their classical degenerations.  We begin by recalling the construction of quantum character stacks and varieties, referring to \cite{Ben-ZviIntegratingquantumgroups2018} for more details.

Let $\Mfld^2_{or}$ denote the topological category with objects being oriented surfaces (possibly with boundary), and with morphisms being the space of oriented embeddings.  Let $\Disk^2_{or}$ denote full subcategory consisting of disjoint unions of disks; hence $\Disk^2_{or}$ is a model for the framed $E_2$-operad.

A ribbon tensor category $\cA$ determines a functor from $\Disk^2_{or}$ to the (2, 1)-category $\Pr$ of locally presentable linear categories. Examples include $\RepK(G)$ for generic quantum parameter, the category $\RepR(G)$ of representations of Lusztig's integral form, its specialization $\Repq(G)$ at root of unity parameter $q$, and finally $\Rep(G)$ in the classical case -- in this latter case the ribbon element is the identity natural isomorphism of the identity functor.  The factorization homology of oriented surfaces with coefficients in $\cA$ is the canonical ``left Kan'' extension,
$$\xymatrix{\Disk^2_{or} \ar[ddr]\ar^{\cA}[rr]&& \Pr\\
\\
& \Mfld^2_{or}\ar@{-->}_{S\mapsto \int_S\cA}[uur]}$$
For any surface $S$, the embedding of the empty surface into $S$ induces a distinguished object $\Dist_S$  in $\int_S\cA$.  In the case $\cA=\Rep(G)$, $\Dist_S$ is simply the structure sheaf, $\O(Ch_G^{fr}(S))$, of the character stack (recall that we regard quasi-coherent sheaves on a quotient stack as $G$-equivariant objects on the framed character variety).  Following \cite{Ben-ZviIntegratingquantumgroups2018}, in the case $\cA=\Repq(G)$, we call $\Dist_S$ the ``quantum structure sheaf''.  To alleviate notation, we will call the classical distinguished object $\Dist_S^{cl}$, and the quantum distinguished object $\Dist_S$.

If $\Sp$ is a surface with boundary, the designation of an interval on the boundary determines on $\int_\Sp\cA$ the structure of an $\cA$-module category, and allows us in particular to define internal homomorphisms.  We denote
\[A_\Sp = \underline{\End}(\O_{\Sp}^q)\in\cA,\]
where $\underline{\End}$ denotes the internal endomorphism algebra. The algebra $A_{\Sp}$ is a deformation quantization of the coordinate algebra of the framed character variety; it was computed in \cite{Ben-ZviIntegratingquantumgroups2018} in terms of Alekseev-Grosse-Schomerus algebras \cite{AlekseevCombinatorialquantizationHamiltonian1996}. We will denote by $A_{\Sp}^\Rt$ the $\Rt$-algebra obtained by replacing $\Repq(G)$ by $\RepR(G)$ in the above construction.

If $S$ is a closed surface, we consider the quantum character variety to be $\End(\Dist_S)$, the global algebra of endomorphisms in $\int_S \Repq(G)$.

\subsection{The Frobenius Poisson order}\label{subsec:CharVarFrobPoissonOrder}
By its construction as a left Kan extension, factorization homology is functorial both in the surface $S$ (under embeddings), and in the braided tensor category $\cA$ (under ribbon tensor functors).  Hence the quantum Frobenius functor induces a further functor
$$\int_S\Fr^*: \int_S\Rep(G)\to\int_S\Repq(G),$$
for any surface $S$. Moreover, this functor is pointed with respect to the distinguished objects $\Dist_S$:  it means that we obtain a canonical isomorphism in $\int_S\Repq(S)$,
$$\left(\int_S\Fr^*\right)\left(\Dist_S^{cl}\right)\cong \Dist_S.$$

For the remainder of this section, let us adopt the following convention: $S$ will henceforth denote a closed connected surface of genus $g$, and $\Sp=S\backslash D$ will denote the surface with boundary obtained by removing a disk $D$ from $S$.

By the above discussion, $\int_{\Sp}\Rep(G)$ is a $\Rep(G)$-module category while $\int_{\Sp}\Repq(G)$ is a $\Repq(G)$-module category. In addition, we obtain an embedding of algebra objects in $\Repq(G),$
\begin{equation}\label{CharVarFrobPair}\Fr^*(\O(Ch^{fr}_G(\Sp))\to A_\Sp.\end{equation}
It follows from the fact that $\Rep(G)$ lies in the M\"uger center of $\Repq(G)$ that the embedding \eqref{CharVarFrobPair} is central. Note that we may identify
\[\O(Ch^{fr}_G(\Sp))\cong \cO(G^{2g})\]
as a $G$-representation and
\[A^\Rt_\Sp\cong \OR(G)^{\otimes 2g}\]
as an object of $\RepR(G)$. In particular, $A^\Rt_\Sp$ is flat over $\Rt$.

\begin{lemma}
$A_\Sp$ is a finitely generated $\O(Ch^{fr}_G(\Sp))$-module.
\label{lm:charactervarietyfinitelygenerated}
\end{lemma}
\begin{proof}
By \cite[Theorem 7.2]{DeConciniQuantumFunctionAlgebra} $\Oq(G)$ is finitely generated over $\cO(G)$. But $A_\Sp\cong \Oq(G)^{\otimes 2g}$ considered as an $\cO(G)^{\otimes 2g}$-module is a tensor product of finitely generated modules, so it is also finitely generated.
\end{proof}

\begin{lemma}
The fiber of the algebra $A_\Sp$ at the trivial local system on $Ch^{fr}_G(S)$ is a matrix algebra.
\label{lm:smallquantumgroupfactorizable}
\end{lemma}
\begin{proof}
The main observation is that the fiber of $A_\Sp$ at the identity is itself an instance of a distinguished object in factorization homology, namely it is the algebra $A_{\Sp}^{sm}$, the internal endomorphism algebra of the distinguished object in the factorization homology category of $\Sp$ with coefficients in the braided tensor category $\Vect\otimes_{\Rep(G)} \Repq(G)$. By \cite{arkhipov_another_2003} this category is equivalent to $\Rep(\uq(\g))$.

Recall the notion of an elliptic double $E_H$ from \cite{BrochierFourier} associated to a quasi-triangular Hopf algebra $H$. The results of \cite{Ben-ZviIntegratingquantumgroups2018} therefore give an isomorphism
$$A_{\Sp}^{sm} \cong E_{\uq(\g)}^{\widetilde{\ot} g},$$
the braided tensor product in $\Rep(\uq(\g))$ of the elliptic double $E_{\uq(\g)}$. By \cite[Theorem 5.6]{BrochierFourier} the elliptic double $E_{\uq(\g)}$ is isomorphic to the Heisenberg double since $\uq(\g)$ is factorizable. By the definition of the Heisenberg double, $E_{\uq(\g)}$ acts faithfully on $\uq(\g)$ and hence (comparing dimensions) we have an isomorphism $$E_{\uq(\g)}\cong \End(\uq(\g)).$$
But the braided tensor product of matrix algebras is a matrix algebra by Proposition \ref{prop:tensorend}.
\end{proof}

In particular, by Lemma \ref{lm:wholecenter} we conclude that $\O(Ch^{fr}_G(\Sp))$ coincides with the center of $A_\Sp$. But then combining Lemma \ref{lm:PoissonOrderAutomatic} and Proposition \ref{prop:PoissonOrderDegeneration}, we obtain the structure of a Frobenius Poisson order on the pair $(A_\Sp, Ch_G^{fr}(\Sp))$.

\begin{prop}
The Poisson $G$-variety $Ch_G^{fr}(\Sp)$ is nondegenerate.
\label{prop:fockroslynondegenerate}
\end{prop}
\begin{proof}
Recall from \cite{FockRosly} that $Ch_G^{fr}(\Sp)$ carries a natural Poisson structure coming from a classical $r$-matrix on $\g$. Using Proposition \ref{prop:RMatrixDegeneration}, one can generalize \cite[Theorem 7.3]{Ben-ZviIntegratingquantumgroups2018} to show that the Poisson structure on $Ch_G^{fr}(\Sp)$ obtained by Proposition \ref{prop:PoissonOrderDegeneration} coincides with the Fock--Rosly Poisson structure associated to the classical $r$-matrix in Proposition \ref{prop:RMatrixDegeneration}.

We may construct the surface $\Sp$ by fusing together cylinders. By Proposition \ref{prop:fusiondegeneration} this means the Poisson structure on $Ch_G^{fr}(\Sp)$ is given by the fusion of the framed character varieties for a cylinder.

The Fock--Rosly Poisson structure on the cylinder is twist equivalent to the quasi-Poisson $(G\times G)$-variety $D(G)$ (see \cite[Example 5.4]{AlekseevQuasiPoisson}) which is nondegenerate by \cite[Example 10.5]{AlekseevQuasiPoisson}.
\end{proof}

\subsection{The Frobenius quantum moment map}\label{subsec:CharVarFrobqmm}

Let us now recall how to compute the quantum character variety of a closed surface in terms of quantum Hamiltonian reduction following \cite{Ben-ZviQuantumcharactervarieties2018}. Let $\Ann$ be the annulus and consider the embedding $Ann\subset \Sp$ as a closed neighborhood of the boundary. We may identify $\int_{\Ann} \RepR(G)$ with the category of $\OR(G)$-modules in $\RepR(G)$. Moreover, $\int_{\Ann} \RepR(G)$ carries a natural monoidal structure coming from stacking of annuli. Moreover, embedding of the disk into the annulus gives a monoidal functor $\RepR(G)\rightarrow \int_{\Ann} \RepR(G)$.

By functoriality of factorization homology, we get the structure of a $\int_{\Ann} \RepR(G)$-module category on $\int_{\Sp} \RepR(G)$. In particular, taking the internal endomorphism algebras of the distinguished objects gives a map
\[\mu_\Rt\colon \OR(G)\longrightarrow A^\Rt_{\Sp}\]
of algebras in $\RepR(G)$. By \cite[Proposition 4.2]{Ben-ZviQuantumcharactervarieties2018} and \cite[Proposition 3.6]{SafronovQuantumMoment} it satisfies the quantum moment map equation \eqref{eq:quantummomentmap}.

Finally, by \cite[Theorem 5.4]{Ben-ZviQuantumcharactervarieties2018} we may identify the quantum character variety of the closed surface as a quantum Hamiltonian reduction:
\[\End(\Dist^\Rt_S)\cong A^\Rt_{\Sp} \dS \ULR(\g)\]
and similarly for its specialization at $t=q$.

Combining Lemma \ref{lm:HamiltonianPoissonOrderAutomatic} and Proposition \ref{prop:HamiltonianFrobeniusPoissonOrderDegeneration} we obtain that $(A_\Sp, Ch_G^{fr}(\Sp))$ is a weakly $G$-Hamiltonian Frobenius Poisson order. In particular, there is a classical moment map $\mu\colon Ch_G^{fr}(\Sp)\rightarrow G$. We will now identify this moment map explicitly.

\begin{prop}
\label{prop:charvarmm}
The classical moment map $\mu\colon Ch_G^{fr}(\Sp)\cong G^{2g}\rightarrow G$ is given by
\[\mu(x_1, y_1, \dots, x_g, y_g) = \prod_{n=1}^g [x_i, y_i].\]
\end{prop}
\begin{proof}
Consider the category $\Delta^1$ consisting of two objects and a single morphism between them. The $(2, 1)$-category $\Fun(\Delta^1, \Pr)$ carries a natural symmetric monoidal structure given by the pointwise tensor product. Then the braided tensor functor $(\Rep(G)\rightarrow \Repq(G))$ defines a functor $\Disk^2_{or}\rightarrow \Fun(\Delta^1, \Pr)$. Since colimits in $\Fun(\Delta^1, \Pr)$ are computed pointwise, the factorization homology is also computed pointwise.

By functoriality of factorization homology we get a morphism
\[
\int_{\Ann} (\Rep(G)\rightarrow \Repq(G))\rightarrow \int_{\Sp} (\Rep(G)\rightarrow \Repq(G))
\]
in $\Fun(\Delta^1, \Pr)$, i.e., a commutative diagram
\[
\xymatrix{
\int_{\Ann} \Repq(G) \ar[r] & \int_{\Sp} \Repq(G) \\
\int_{\Ann} \Rep(G) \ar[r] \ar[u] & \int_{\Sp} \Rep(G) \ar[u]
}
\]
of pointed categories. Computing the internal endomorphisms of the distinguished objects, we get the commutative diagram
\[
\xymatrix{
\Oq(G) \ar^{\mu_q}[r] & A_{\Sp} \\
\cO(G) \ar[r] \ar[u] & \cO(G^{2g}) \ar[u].
}
\]
Since the map $\cO(G)\hookrightarrow \Oq(G)$ is injective, we get that the classical moment map is given by the bottom morphism. But it is induced by the homomorphism $\pi_1(\Ann)\rightarrow \pi_1(\Sp)$ given by monodromy around the boundary.
\end{proof}

\subsection{Proof of Theorem \ref{ChG-thm-intro}}\label{subsec:CharVarFrobqHred}

Consider the open subvariety $Ch_G^{fr, good}(S)\subset Ch_G^{fr}(S)$ consisting of points whose $G$-orbit is closed and whose $G$-stabilizer is the center (see e.g. \cite{SikoraCharacter}). This locus is empty when $g=0,1$, but otherwise non-empty. Denote by $Ch_G^{good}(S)\subset Ch_G(S)$ its image under the quotient map.

We are now ready to prove our main theorem about character varieties, Theorem \ref{ChG-thm-intro}, which we recall here:

\begin{theorem}[Theorem \ref{ChG-thm-intro}]\label{ChG-thm-body}
Let $G$ be a connected reductive group and $q$ a primitive $\ell$-th root of unity, which together satisfy Assumption \ref{EllAssumption}. Let $S$ be a closed topological surface of genus $g$, and let us denote by $\Sp$ the surface obtained by removing some open disk from $S$. Then:
\begin{enumerate}
\item The moduli algebra $A_\Sp$ is finitely generated over its center, which is isomorphic to the coordinate ring of the classical framed $G$-character variety $Ch_G^{fr}(\Sp)\cong G^{2g}$.
\item Moreover, the Azumaya locus of the moduli algebra $A_\Sp$ coincides with the preimage of open cell $\Br\subset G$ under the classical moment map $Ch_G^{fr}(\Sp) \to G$. 
\item The quantized character variety of the closed surface $S$ is finitely generated over its center, which is isomorphic to the coordinate ring of the classical character variety.  It may be constructed as a Frobenius quantum Hamiltonian reduction of $A_\Sp$.
\item Moreover the quantized character variety of the closed surface $S$ is Azumaya over the entire `good locus' $Ch_G^{good}(S)\subset Ch_G(S)$.
\end{enumerate}
\end{theorem}
\begin{proof}
The first statement is established in Lemma \ref{lm:charactervarietyfinitelygenerated}. The containment of the Azumaya locus within $\mu^{-1}(\Br)$ is Lemma \ref{lem:BrAzumaya}.  To confirm the remaining statements, we will use Theorem \ref{thm:frobtwistedAzumaya}, and hence we need only confirm the assumptions of that theorem.

We take for $A_\Rt$ the algebra $A^\Rt_\Sp$, for $\xi_q$ the counit $\epsilon_q$ of $\Oq(G)$, lying over the identity element $e\in G$ viewed as a character on $\cO(G)$.  For $Z$ we take $\O(Ch_G^{fr}(\Sp))$. For $G_{stab}$ we take the center $Z(G)\subset G$.

The moment map $\mu\colon Ch_G^{fr}(\Sp)\rightarrow G$ is a product of commutators, so it lands in the derived subgroup $G_{der}=G_0\subset G$. Note that $G_{der}\subset G\rightarrow G/Z(G)$ is \'{e}tale.

The classical character variety $Ch_G^{fr}(\Sp)$ is isomorphic to $G^{2g}$, so it is smooth and connected. Moreover, by Proposition \ref{prop:fockroslynondegenerate} it is a nondegenerate Poisson $G$-variety which verifies Assumption 1 of Theorem \ref{thm:frobtwistedAzumaya}.

For $U$ we take the locus of ``good representations" $Ch_G^{fr, good}(S)\subset \mu^{-1}(e)$. Note that the center $Z(G)=G_{stab}$ acts trivially on the whole character variety $Ch_G^{fr}(\Sp)$. By Remark \ref{rmk:REAcenteraction} the weights of $\Oq(G)$ lie in the root lattice. Since the quantum algebra $\cA_{\Sp}$ is a tensor product of the algebras $\Oq(G)$, the weights of $\cA_\Sp$ also lie in the root lattice. Therefore, the weights of $\cA_\Sp|_{\mu^{-1}(e)} \dS \uq(\g)$ also lie in the root lattice of $G$, so $Z(G)$ acts trivially which verifies Assumption 2 of Theorem \ref{thm:frobtwistedAzumaya}.

The surjectivity criterion, Assumption 1 of Theorem \ref{thm:frobtwistedAzumaya}, can be checked using the theory of good filtrations. Namely, $\OR(G)$ admits a good filtration by Theorem \ref{thm:REAgoodfiltration}. $\cA^\Rt_\Sp$ is a tensor product of $\OR(G)$, so by Proposition \ref{prop:goodfiltrations} it also admits a good filtration and hence satisfies the surjectivity criterion.

\end{proof}

\begin{rmk}
We note that the proof of Theorem \ref{ChG-thm-body} gives a mechanism to compute the fibers of the algebras $A_\Sp$ and its Hamiltonian reduction over the good locus, even when $q$ and $G$ do not satisfy Assumption \ref{EllAssumption}, i.e. $\Rep(\uq(\g))$ is not factorizable.  We hope to return to this in future work.  
\end{rmk}

Let us now show that the results are, in a sense, optimal when $G=\GL_n$ or $\SL_n$.

\begin{prop}
Suppose $G=\GL_n $ or $\SL_n$ and $g\geq 2$. Then $Ch_G^{good}(S)$ coincides with the smooth locus of $Ch_G(S)$.
\label{prop:bellamyschedler}
\end{prop}
\begin{proof}
Let $Ch_G^{sm}(S)\subset Ch_G(S)$ be the smooth locus and $Ch_G^{symp}(S)\subset Ch_G^{sm}(S)$ be the open symplectic leaf. By Proposition \ref{prop:symplecticquotient} we moreover have $Ch_G^{good}(S)\subset Ch_G^{symp}(S)$.

By \cite[Theorem 1.20]{BellamySchedler} the variety $Ch_G(S)$ is an irreducible symplectic singularity. In particular, by \cite[Theorem 2.3]{Kaledin} $Ch_G(S)$ has finitely many symplectic leaves and hence $Ch_G^{symp}(S)=Ch_G^{sm}(S)$. Moreover, by \cite[Proposition 8.5]{BellamySchedler} $Ch_G^{good}(S) = Ch_G^{symp}(S)$.
\end{proof}

\subsection{Kauffman bracket skein algebras and the Unicity Conjecture}\label{Kauffman-section}

Let $S$ be an oriented surface and let $K_A(S)$ be the Kauffman bracket skein algebra which is a $\C[A, A^{-1}]$-algebra (see e.g. \cite[Section 3]{FrohmanUnicity} for recollections). The following is first proved in \cite{PrzytyckiSikora} (see also \cite[Theorem 3.1]{FrohmanUnicity}).

\begin{theorem}\label{thm:skein-basis}
$K_A(S)$ is free as a $\C[A, A^{-1}]$-module, with basis given by the set of isotopy classes of simple multi-curves on $S$.
\end{theorem}

Furthermore, we have:
\begin{theorem}
Choose a spin structure on $S$. The Poisson algebra $K_1(S)$ equipped with its natural Poisson bracket is isomorphic to the Poisson algebra $\cO(Ch_{\SL_2}(S))$ equipped with its Fock--Rosly (equivalently, Goldman) Poisson structure.
\end{theorem}
\begin{proof}
It is shown in \cite{TuraevSkein} and \cite{BFKB} that the Poisson algebra $K_{-1}(S)$  is isomorphic to $\cO(Ch_{\SL_2}(S))$ equipped with the Fock--Rosly Poisson structure.

In addition, Barrett \cite[Theorem 1]{BarrettSkein} constructs an isomorphism of $\C[A, A^{-1}]$-algebras
\[K_A(S)\cong K_{-A}(S)\]
using a spin structure on $S$.

At $A=1$ both algebras become commutative. So, by functoriality of the Poisson bracket induced on the center (see Proposition \ref{prop:PoissonOrderDegeneration} and Lemma \ref{lm:PoissonOrderAutomatic}), we obtain an isomorphism of Poisson algebras
\[K_1(S)\cong K_{-1}(S).\]
\end{proof}

We will now show that the same Poisson structure appears naturally when we degenerate to roots of unity.  Fix an odd number $\ell>1$ and suppose $\zeta$ is a primitive $\ell$th root of unity. Denote $K_\zeta(S) = K_A(S) / (A-\zeta)$.  A fundamental construction in the study of Kauffman bracket skein modules at root of unity parameters is the following statement \cite{BonahonWongI}.

\begin{theorem}\label{thm:BWqTr}
There exist injective homomorphisms
\[
\Fr^{(\ell)}\colon K_1(S)\to Z(K_\zeta(S)),
\]
natural for surface embeddings.
\end{theorem}

In \cite[Theorem 4.1]{FrohmanUnicity}, the quantum Frobenius map $\Fr^{(\ell)}$ was proved to be an isomorphism onto the center in many cases.  The particular case of their theorem we will require is as follows:

\begin{theorem}
Suppose the surface $S$ is closed. Then $\Fr^{(\ell)}$ is an isomorphism.
\label{thm:BWqTrIsomorphism}
\end{theorem}

Therefore, combining Proposition \ref{prop:PoissonOrderDegeneration} and Lemma \ref{lm:PoissonOrderAutomatic} we obtain a Poisson order structure on $(K_\zeta(S), K_1(S))$. So, for every $\ell$ we obtain a Poisson structure on the same algebra $K_1(S)$.

\begin{lemma}
Let $\Sp$ be surface obtained by removing some disc from a closed surface $S$.  Then $(K_\zeta(\Sp), K_1(\Sp))$ is a Poisson order.
\end{lemma}
\begin{proof}
We would like to apply Lemma \ref{lm:PoissonOrderAutomatic} directly to conclude that $K_1(S)$ is closed under the Poisson bracket, however because $\Sp$ is not closed, $\Fr^{(\ell)}(K_1(\Sp))\neq Z(K_\zeta(\Sp))$, so the Lemma does not apply.  Neither will it suffice to embed $\Sp$ back into the closed surface $S$, because the induced map on skein algebras has a kernel.  Instead, let us embed $\Sp$ into a surface $T$ of genus one greater than $S$, so that by Theorem \ref{thm:skein-basis} the induced map $K_A(\Sp)\to K_A(T)$ is injective.  Since $T$ is closed and $\Fr^{(\ell)}$ is natural for surface embeddings, we may indeed apply Lemma \ref{lm:PoissonOrderAutomatic}, to conclude that 
\[\{\Fr^{(\ell)}(K_1(\Sp)),\Fr^{(\ell)}(K_1(\Sp))\}\subset \Fr^{(\ell)}(K_1(T))\cap K_\zeta(\Sp) = \Fr^{(\ell)}(K_1(\Sp)),\]
as desired.
\end{proof}

\begin{prop}\label{prop:independent}
The Poisson structure on $K_1(S)$ in the Poisson order $(K_\zeta(S), K_1(S))$ is independent, up to a factor, of the order $\ell$ of the root of unity.
\end{prop}
\begin{proof}
The claim is trivial if $S$ has genus 0, so we will assume that it has genus $\geq 1$. Let us first denote by $\Sp$ denote the surface obtained from $S$ by removing some disc.

Choose $\omega$ such that $\omega^{-2} = \zeta$. By \cite[Theorem 1]{BonahonWongQuantumTraces} (see also \cite{MullerSkein}), any suitable triangulation $\Delta$ of $\Sp$ determines an algebra embedding
\[
\Tr^\omega_\Delta\colon K_\zeta(\Sp) \to \mathcal{T}^\Delta_\zeta,
\]
of the Kauffman bracket skein algebra into a quantum torus $\mathcal{T}^\Delta_\zeta$.  Here by quantum torus we mean that $\mathcal{T}^\Delta_\zeta$ is presented with invertible generators $X_1,\ldots X_r$, for some $r$, with relations $X_iX_j = \zeta^{n_{ij}}X_jX_i$, for some skew-symmetric integer matrix $(n_{ij})$.

On the quantum torus we have a simple quantum Frobenius homomorphism,
\[
\Fr^{(\ell)}\colon \mathcal{T}^\Delta_1 \to Z(\mathcal{T}^\Delta_\zeta),
\]
which sends each generator $X_i$ of the ring $\mathcal{T}^\Delta_1$ to the $\ell$-th power of the corresponding generator in $\mathcal{T}^\Delta_\zeta$. In particular, by Proposition \ref{prop:PoissonOrderDegeneration} we obtain the structure of a Poisson order on $(\mathcal{T}^\Delta_\zeta, \mathcal{T}^\Delta_1)$. It is straightforward to see that the induced Poisson bracket on $\mathcal{T}^\Delta_1$ depends on $\ell$ only up to an overall scalar factor.

By \cite[Theorem 21]{BonahonWongI} we have a map of Poisson orders
\[(\Tr^\omega_\Delta, \Tr^{\omega^{\ell^2}}_\Delta)\colon (K_\zeta(\Sp), K_1(\Sp))\hookrightarrow (\mathcal{T}^\Delta_\zeta, \mathcal{T}^\Delta_1).\]

Since any skein in $S$ is isotopic to one missing the disk, the map $p\colon K_A(\Sp)\rightarrow K_A(S)$ is surjective. By naturality of $\Fr^{(\ell)}$ under surface embeddings (Theorem \ref{thm:BWqTr}) we have a map of Poisson orders
\[p\colon (K_\zeta(\Sp), K_1(\Sp))\twoheadrightarrow (K_\zeta(S), K_1(S)).\]

Consider a pair of elements $f, g\in K_1(S)$. Our goal will be to show that $\{f, g\}$ is independent of $\ell$ up to a factor. Since $p\colon K_1(\Sp)\rightarrow K_1(S)$ is surjective, we may find the lifts $\tilde{f}, \tilde{g}\in K_1(\Sp)$ of $f,g$ which we assume are independent of $\ell$.

The Poisson bracket $\{\Tr^{\omega^{\ell^2}}_\Delta(\tilde{f}), \Tr^{\omega^{\ell^2}}_\Delta(\tilde{g})\}$ is independent of $\ell$ up to a factor. Since $\Tr^{\omega^{\ell^2}}_\Delta\colon K_1(\Sp)\hookrightarrow \mathcal{T}^\Delta_1$ is Poisson, the same claim holds for $\{\tilde{f}, \tilde{g}\}$. Since $p\colon K_1(\Sp)\rightarrow K_1(S)$ is Poisson, the same claim also holds for $\{f, g\}$.
\end{proof}

\begin{rmk}
In a previous version of this paper, Proposition \ref{prop:independent} was incorrectly attributed to the literature.  We are grateful to Thang Le for bringing this to our attention, and suggesting that we add a complete proof.
\end{rmk}

\begin{theorem}[Theorem \ref{thm-Kauffman-intro}] 
The skein algebra $K_\zeta(S)$ is Azumaya over the whole smooth locus of $Ch_{\SL_2}(S)$.
\end{theorem}
\begin{proof}
It is shown in \cite[Theorem 2]{FrohmanUnicity} that $K_\zeta(S)$ is generically Azumaya over $Ch_{\SL_2}(S)$. Since $(K_\zeta(S), \cO(Ch_{\SL_2}(S)))$ is a Poisson order, by Theorem \ref{thm:browngordon} $K_\zeta(S)$ is Azumaya over the open symplectic leaf in $Ch_{\SL_2}(S)$. To prove the claim, we have to establish that the open symplectic leaf coincides with the smooth locus.

The case $g=0$ is trivial.

Suppose $g=1$. We have an isomorphism $Ch_{\SL_2}(S)\cong (\C^\times\times \C^\times) \sS \Z_2$ where $\Z_2$ acts on each factor by $z_i\mapsto z_i^{-1}$. In particular, the smooth locus is the locus of pairs $(z_1, z_2)$ where $z_i\neq \pm 1$. The Poisson structure on $Ch_{\SL_2}(S)$ comes from the standard symplectic structure on $\C^\times\times \C^\times$. Since the map $\C^\times\times \C^\times\rightarrow Ch_{\SL_2}(S)$ restricts to a $\Z_2$-torsor over the smooth locus, we get that the Poisson structure on the smooth locus is symplectic.

Suppose $g\geq 2$. Then by Proposition \ref{prop:bellamyschedler} the open symplectic leaf coincides with the smooth locus.
\end{proof}

\section{Quantized multiplicative quiver varieties}\label{sec:DqMat}

We now recall the definition of the multiplicative quiver variety, following \cite{YamakawaGeometryMultiplicativePreprojective2008, Crawley-BoeveyMultiplicativepreprojectivealgebras2006}, and its quantization, following \cite{JordanQuantizedmultiplicativequiver2014}. We subsequently implement the program of Section \ref{sec:frobenius}, in order to deduce that quantized multiplicative varieties define sheaves of Azumaya algebras over classical multiplicative quiver varieties.  Throughout this section, $q$ denotes a primitive root of unity of order $\ell$, where $\ell >1$ is odd.

\subsection{The classical multiplicative quiver variety}\label{subsec:mqv}

Let $Q = (E,V)$ be a finite quiver.  The doubled quiver $\overline{Q}$ has the same vertex set $\overline{V}=V$ but with an additional `dual' edge $e^\vee:j\to i$, for each edge $e:i\to j$ of $Q$, so $\overline{E}=E \sqcup E^\vee$.  For $e \in \overline{E}$, set $\epsilon(e)$ to be equal to $1$ if $e \in E$ and $-1$ if $e \in E^\vee$, and  write $\alpha = \alpha(e) \in V$ and $\beta = \beta(e) \in V$ for the source and target of  $e$.  Fix a dimension vector $\mathbf{d} = (\dv_v)_{v \in V} \in (\Z_{\geq 0})^V$.

A \emph{framed representation} of $\overline{Q}$ with dimension vector $\dv$ is an assignment of a linear map  $\C^{\dv_\alpha} \rightarrow \C^{\dv_\beta}$ to each edge $e \in \overline{E}$. The moduli space $\MQdf$ of framed representations of $\overline{Q}$ is therefore a Cartesian product of spaces of matrices,
$$\MQdf = \prod_{e\in E} \Mat(e) \times \Mat(e^\vee),\quad \textrm{ where }
\Mat(e)=\Hom_\C(\C^{\dv_\alpha},\C^{\dv_\beta}).$$
The group $G=\GL_{\dv} = \prod_v \GL_{d_v}$ acts on $\MQdf$ by change of basis at each vertex.

Let $\MQdf^\circ$ denote the Zariski open locus of $\MQdf$ on which the determinants of the matrices $(\id_\alpha + X_{e^\vee} X_e)$ are  non-vanishing for all $e \in \overline{E}$.  We have  a \emph{multiplicative moment map} $$\tilde \mu:\MQdf^\circ\to \GL_{\dv}$$
\begin{equation}X \mapsto \prod^{\rightarrow}_{e \in \overline{E}} (\id_\alpha + X_{e^\vee} X_e)^{\epsilon(e)}.\label{moment-map-formula}\end{equation}

Fix a function $\xi : V \rightarrow \C^\times$, satisfying $\prod_v\xi_v=1$, which we regard as a collection $\xi = (\xi_v)_{v \in V} \in \GL_{\dv}$ of scalar matrices.  Fix a character $\theta\colon G\rightarrow \C^\times$. Denote by $\C_{\theta}\in\Rep(G)$ the corresponding one-dimensional representation.

\begin{definition}
The multiplicative quiver variety is the GIT Hamiltonian reduction
\[\MQd   = \MQdf^\circ\dS_{\!\!\theta} \GL_\dv = \tilde\mu^{-1}(\xi)\sS_{\!\!\theta} \GL_\dv = \Proj\left(\bigoplus_{m=0}^\infty \cO(\tilde \mu^{-1}(\xi))\otimes \C_{\theta^{-m}}\right)^{\GL_\dv}.\]
\end{definition}

Recall the notion of $\theta$-semistable and $\theta$-stable points of the $G$-variety $\tilde \mu^{-1}(\xi)$ (see \cite[Chapter 1]{MumfordGIT} and \cite[Section 2]{KingModuli}).  Then $\tilde\mu^{-1}(\xi)\sS_{\!\!\theta} G$ may be identified with the quotient of the open subset of $\tilde\mu^{-1}(\xi)$ of $\theta$-semistable points by a certain equivalence relation.  In particular, we obtain a surjective morphism
\[\pi:\tilde\mu^{-1}(\xi)^{\theta\mathrm{-ss}}\longrightarrow \tilde\mu^{-1}(\xi)\sS_{\!\!\theta} G.\] 
 
\begin{definition}
The stable multiplicative quiver variety $\MQds\subset \MQd$ is the image of the $\theta$-stable points of $\tilde\mu^{-1}(\xi)$.
\end{definition}

Note that by construction we have a projective surjection $\tilde\mu^{-1}(\xi)\sS_{\!\!\theta} G\rightarrow \tilde\mu^{-1}(\xi)\sS G$ which is an isomorphism for $\theta$ trivial. The construction of a Poisson structure on $\tilde\mu^{-1}(\xi)\sS_{\!\!\theta} G$ mimics the construction of the Poisson structure on $\tilde\mu^{-1}(\xi)\sS G$ given by Proposition \ref{prop:poissonreduction}.

\begin{rmk} The multiplicative quiver variety was constructed first in \cite{Crawley-BoeveyMultiplicativepreprojectivealgebras2006} as a moduli space of representations of a multiplicative pre-projective algebra.  The construction as a multiplicative Hamiltonian reduction was accomplished in the papers \cite{VandenBerghDoublePoissonalgebras2008} and \cite{VandenBerghQuasiHamiltonian}, using the formalism of quasi-Hamiltonian moment maps. See Remark \ref{rmk-VdB} for more details.
\end{rmk}

\begin{example}
An important special case is when $\theta=0$.  Then it is well-known (see, e.g. \cite{GinzburgQuiver}) that every representation of $Q$ is semistable, and that only the simple representations are stable.  For general $\theta$, A. King gave a purely algebraic description of $\theta$-(semi)stability in terms of excluded `slopes' of sub-representations; see \cite{KingModuli}.
\end{example}

\begin{rmk} It can often happen that $\MQds$ is empty altogether -- for instance when $\theta=0$, not every dimension vector supports a simple reprsentation of the preprojective algebra.  In the other extreme, it can happen that the semi-stable and stable loci coincide.  In this case $\MQd$ defines a symplectic resolution of $\MQdtriv$, by results of \cite{SchedlerTirelli}.  Crawley-Boevey and Shaw have given an explicit characterization for when the stable locus is non-empty in terms of certain root datum on the quiver; we refer to \cite{Crawley-BoeveyMultiplicativepreprojectivealgebras2006}, and \cite{SchedlerTirelli} for the complete statement.
\end{rmk}

\subsection{The quantization of $\MQdf$}
\label{subsec:quantizationofQ}
We recall the quantizations of the framed moduli space  and the multiplicative quiver variety, following \cite[Section 3]{JordanQuantizedmultiplicativequiver2014}, and we describe a straightforward extension of the construction to non-trivial GIT quotients.   We begin by recalling the quantization of the framed moduli space.

\begin{definition}\cite[Example 4.8]{JordanQuantizedmultiplicativequiver2014} \label{def:dqmatmn}
	The algebra $\DqMatMN$ is the algebra generated over $\C$ by elements $x^i_j$ and $\partial_k^l$, for $1 \leq i, k \leq M$ and $1 \leq j,l \leq N$, organized into an $N$ by $M$ matrix $X$ and an $M$ by $N$ matrix $D$, with relations given by the following matrix equations:
	$$RX_2X_1=X_1X_2R_{21}, \quad RD_2D_1=D_1D_2R_{21},$$
	$$D_2R^{-1}X_1=X_1RD_2 + \Omega, \quad \textrm{where } \Omega=\sum E^i_j\ot E^j_i.$$
	 Define $\DqMatMN^\circ$ as the  non-commutative localization at the quantum
	determinants $\det_q$ of the following matrices:
	$$g^\alpha  := \id + (q -q\inv) DX, \qquad  g^\beta := \id + (q -q\inv) XD .$$
	\end{definition}

\begin{lemma}	In coordinates, this means $x^i_j$ and $\partial^k_l$ satisfy:
	\begin{align*} x_m^i x_n^j &= q^{\delta_{mn}} x_n^j x_m^i + \theta(n-m) (q-q\inv) x_m^j x_n^i \qquad & (i >j) \\
	x_m^i x_n^i &= q\inv x_n^i x_m^i \qquad & (m >n) \\
	\partial_m^i \partial_n^j &= q^{\delta_{mn}} \partial_n^j \partial_m^i + \theta(n-m) (q-q\inv) \partial_m^j \partial_n^i \qquad & (i >j) \\
	\partial_m^i \partial_n^i &= q\inv \partial_n^i \partial_m^i \qquad & (m >n) \\
	\partial_m^i x_n^j &= q^{\delta_{in} + \delta_{jm}} x_n^j \partial_m^i + \delta_{in} q^{\delta_{jm}}(q-q\inv) \sum_{p >i} x^j_p \partial^p_m + \delta_{jm}(q^2 -1) \sum_{p < j} \partial_p^i x^p_n + q\delta_{in} \delta_{jm} & \end{align*}
	\end{lemma}

We recall the following PBW basis theorem for the algebra $\DqMatMN$.

\begin{theorem}[\cite{JordanQuantizedmultiplicativequiver2014}]\label{thm:PBW}
The ordered monomials (in any fixed ordering), of the form:
$$(x^{i_1}_{j_1})^{k_1} \cdots (x^{i_m}_{j_m})^{k_m} \cdot (\partial^{r_1}_{s_1})^{t_1} \cdots (x^{r_n}_{s_n})^{t_n},$$
define a basis of the algebra $\DqMatMN$.
\end{theorem}

\begin{definition} The algebra $\Dqloop $  is generated over $\C$ by elements $a_i^j$ and $d_i^j$, organized into $N$ by $N$ matrices $A$ and $D$, with relations given by the following matrix equations:
	$$R_{21}A_1RD_2 = A_2R_{21}A_1R,\quad
	R_{21}D_1RD_2 = D_2R_{21}D_1,\quad
	R_{21} D_1 R A_2 = A_2 R_{21} D_1 (R_{21})\inv.$$ 
The algebra $\Dqloop^\circ$ is the localization at the element $\detq(A) \detq(D)$, which is a $q$-central element.  \end{definition}

\begin{rmk}
An important observation is that the algebra $\Dqloop^\circ$ is precisely the algebra $A_{T^2\backslash D^2}$ for the group $\GL_N$.
\end{rmk}

\begin{definition} Let $Q$ be a quiver equipped with a dimension vector $\mathbf{d}$. Let $e$ be an edge of $Q$, with $d_\alpha = N$ and $d_\beta =M$. Set
	$$\DqMate:= \begin{cases}
	\DqMatMN \qquad &\text{\rm if $e$ is not a loop} \\ 
	  \Dqloop \qquad &\text{\rm if $e$ is a loop (so $d_\alpha = d_\beta= N$)} 
	\end{cases}$$

  \end{definition}

\begin{definition} Let $\RepqGLd$ denote the Deligne--Kelly tensor product of the categories $\Repq(\GL_{d_v})$,
$$\RepqGLd= \underset{v \in V}{\boxtimes} \Repq(\GL_{d_v}).$$
We regard $\Repq(\GL_\dv) $ as a braided tensor category with the product braiding.   The tensor product of Frobenius homomorphisms gives functors
$$\Fr^* : \Rep(\GL_\mathbf{d}) \rightarrow \RepqGLd  \qquad \qquad \Fr_* : \RepqGLd \rightarrow \Rep(\GL_\mathbf{d}).$$
\end{definition}

\begin{rmk}
In other words, an object of $\Repq(\GL_\dv) $ is a vector space equipped with commuting actions of $\UqL(\gl_{\dv_v})$, for each $v\in V$.  The tensor product and braiding on each such module is induced component-wise, meaning in particular that the universal $R$-matrix is simply the product of those for each $v$.
\end{rmk}

The  quantized edge algebras $\DqMatMN$ and $\Dqloop$ were defined invariantly as an algebras in $\RepqGLd$ as a quotient of a tensor algebra by the image of certain morphisms mimicking the Weyl algebra relations (see \cite[Section 3.2]{JordanQuantizedmultiplicativequiver2014}). We now recall the definition of the quantization of the moduli space $\MQdf$ of framed representations of $\overline{Q}$.

\begin{definition}[\cite{JordanQuantizedmultiplicativequiver2014}] \label{def:dqQ} Let $Q$ be a quiver with dimension vector $\mathbf{d}$. The algebra $\DqQ $ (respectively $\DqQ^\circ$)  is the braided tensor product in $\Repq(\GL_{\dv})$ of the corresponding edge algebras:
$$\DqQ = \bigotimes_{e \in E}  \DqMate \qquad \DqQ^\circ = \bigotimes_{e \in E}  \DqMate^\circ.$$
\end{definition}

\begin{definition}\label{def:quiver-moment-map}
The quantum moment map
$$\mu_q\colon \Oq(\GL_\mathbf{d}) \longrightarrow \DqQ^\circ,
$$
is the braided tensor product of moment maps $\Oq(\GL_{d_{\alpha}} \times \GL_{d_{\beta}} ) \to \DqMate$, each given by the formulas,
$$\mu_q(a_j^i \ot a_r^s) = (\delta_j^i + (q - q\inv) \sum_k \partial_k^i x_j^k)\inv (\delta_r^s + (q-q\inv) \sum_k x_k^s \partial_r^k  )  ).$$
\end{definition}
It is proved in \cite[Propositions 7.11 and 7.12]{JordanQuantizedmultiplicativequiver2014} that each $\mu_q\colon \Oq(\GL_{d_{\alpha}}\times \GL_{d_{\beta}})\rightarrow \DqMate$ defines a quantum moment map in the sense of Definition \ref{def:quantummomentmap}, so by Proposition \ref{prop:momentmapfusion} $\mu_q\colon \Oq(\GL_\mathbf{d}) \rightarrow \DqQ^\circ$ is also a quantum moment map.

\begin{rmk}
{\it A priori}, the definitions of $\DqMatQd$ and $\mu_q$ depend on a orientation of $Q$, as well as a total ordering of the edges of $\overline{Q}$; however, in \cite{JordanQuantizedmultiplicativequiver2014}, canonical isomorphisms are constructed identifying the different choices compatibly. 
\end{rmk}

\subsection{The Frobenius Poisson order on $\DqMatQd$}
\label{sec:qmqvfrobeniuspair}
We now turn to the construction of the Frobenius Poisson order structure on $\DqMatQd$.

\begin{lemma}\label{lem:coeffs} Fix $i$, $j$, $n$, and $m$. For $r\geq 1$, define the following quantities:
\begin{align*}
a_r &= \delta_{in} q^{\delta_{jm}r} (q^r - q^{-r}),\qquad & b_r &=  \delta_{jm} q^{\delta_{in}(r-1) + 1}(q^{r} - q^{-r})\\
 c_r &=  \delta_{in}\delta_{jm}(q^{2r} - 1)(1 - q^{-2(r-1)}),\qquad & d_r &=   \delta_{in}\delta_{jm}\frac{q^{2r} - 1}{q - q\inv}\\
t_r &= \theta(n-m) (q^{2r-1} - q\inv) \end{align*}
For $r,s \geq 1$, the following identities hold in $\DqMatMN$:
\begin{align*}
x_m^i (x_n^j)^r &= q^{r \delta_{mn}} (x_n^j)^r x_m^i  + t_r (x_n^j)^{r-1} x_m^j x_n^i & (i>j)\\
(x_m^i)^r x_n^j &= q^{r \delta_{mn}} x_n^j( x_m^i)^r  + t_r x_m^j x_n^i (x_m^i)^{r-1} & (i>j)\\
\partial_m^i (\partial_n^j)^r &= q^{r \delta_{mn}} (\partial_n^j)^r \partial_m^i  + t_r(\partial_n^j)^{r-1} \partial_m^j \partial_n^i & (i>j)\\
(\partial_m^i)^r \partial_n^j &= q^{r \delta_{mn}} \partial_n^j( \partial_m^i)^r  + t_r \partial_m^j \partial_n^i (\partial_m^i)^{r-1} & (i>j)\\
(x_m^i)^r (x_n^i)^s &= q^{-rs} (x_n^i)^s (x_m^i)^r & (m >n)\\
(\partial_m^i)^r (\partial_n^i)^s &= q^{-rs} (\partial_n^i)^s (\partial_m^i)^r & (m >n)
\end{align*}

\begin{align*} \partial_m^i (x_n^j)^r = q^{(\delta_{in} +  \delta_{jm})r} &(x_n^j)^r \partial_m^i + a_r \sum_{p >i} (x_n^j)^{r-1} x_p^m \partial_m^p  + b_r \sum_{p' <j} (x_n^j)^{r-1} \partial_{p'}^i x_n^{p'} + \\  & + c_r \sum_{p>n, p'<m} (x_n^m)^{r-2} x_p^m \partial^p_{p'} x^{p'}_n + d_r  (x_n^j)^{r-1}.\\
 (\partial_m^i)^r x_n^j = q^{(\delta_{in} +  \delta_{jm})r} &x_n^j (\partial_m^i)^r  +  a_r \sum_{p >i} x_p^m \partial_m^p  (\de_m^i)^{r-1}  + b_r \sum_{p' <j} \partial_{p'}^i x_n^{p'}  (\de_m^i)^{r-1} + \\  & + c_r \sum_{p>n, p'<m}  \de_{p'}^i x^{p'}_p \de^{p}_m (\de_m^i)^{r-2} + d_r  (\de_m^i)^{r-1}.  \end{align*}
\end{lemma}

\begin{proof} We give a proof of the identity involving $\partial_m^i (x_n^j)^r $. The justification of the other identities is similar or easier. We proceed by induction on $r$. The base case $r=1$ follows from definitions. Now, if $p >i = n$, then we have that 
$$ x_p^j \partial_j^p  x^j_n  = q^{\delta_{jm} -1}  x^j_n   x_p^j \partial_j^p + \delta_{jm} (q^2-1) \sum_{p' <j} x_p^j \de_{p'}^p x_n^{p'},$$ 
and if $p' <j =m $, then
$$ \de_{p'}^i x_i^{p'} x^j_n =  q^{\delta_{in} -1} x^j_n \de_{p'}^i x_i^{p'} + \delta_{in} (1-q^{-2}) \sum_{p >i} x_p^j \de_{p'}^p x_i^{p'}$$
Finally, if $p>i = n$ and $p' < j=m$, then  $ x_p^j \de_{p'}^p x_i^{p'}  x^j_i = q^{-2} x_i^j x_p^j \de_{p'}^p x_i^{p'}. $ We write
\begin{align*} \partial_m^i (x_n^j)^r = q^{(\delta_{in} +  \delta_{jm})r} &(x_n^j)^r \partial_m^i + a_r \sum_{p >i} (x_n^j)^{r-1} x_p^m \partial_m^p + \\  & + b_r \sum_{p' <j} (x_n^j)^{r-1} \partial_{p'}^i x_n^{p'}  + c_r \sum_{p>n, p'<m} (x_n^j)^{r-2} x_p^j \partial^p_{p'} x^{p'}_n + d_r (x_n^j)^{r-1}\end{align*} for some $a_r, b_r, c_r, d_r$ in $\C$. Straightforward computations imply the recursive relations:
\begin{align*}
a_{r+1} &=  q^{\delta_{jm} -1}a_r + \delta_{in} q^{(1+\delta_{jm}) r + \delta_{jm} } (q - q\inv),\\
b_{r+1} &=  q^{\delta_{in} -1}b_r +   \delta_{jm} q^{(1+\delta_{in})r}(q^2 - 1), \\
c_{r+1} &=  \delta_{in} \delta_{jm} (q^2 - 1)a_r + (1-q^{-2})b_r + q^{-2} c_r, \\
d_{r+1} &=  d_r + \delta_{in} \delta_{jm}q^{2r+1}.\end{align*}
From the initial conditions $a_1 =  \delta_{in} q^{\delta_{jm}}(q-q\inv) $, $ b_1 = \delta_{jm}(q^2 -1) $, $c_1 = 0$, and $d_1 = q$, one deduces the formulae stated in 
Lemma \ref{lem:coeffs}. \end{proof}

\begin{cor}
As an object of $\Repq(\GL_M)$, $\OqMatMN$ is isomorphic to the tensor product, 
$$\left(\Sym_{q^{-1}}((\C^M)^*)\right)^{\ot N},$$
while as an object of $\Repq(\GL_N)$, $\OqMatMN$ is isomorphic to the tensor product,
$$\left(\Sym_{q}(\C^N)\right)^{\ot M},$$
where we denote
$$\Sym_q(V) = T(V)/(\sigma_{V,V}-q), \qquad \Sym_{q^{-1}}(V) = T(V)/(\sigma_{V,V}-q^{-1}).$$
\end{cor}

\begin{proof}
We prove the first statement, the second being identical.  For $i=1,\ldots, N$, let $\OqMatMN_{(i)}$ denote the subalgebra generated by the $x^j_i$, for $j=1,\dots, M$.  Inspection of the defining relations gives an isomorphism,
$$\OqMatMN_{(i)}\cong \Sym_{q^{-1}}((\C^M)^*).$$
The PBW basis Theorem \ref{thm:PBW} implies an isomorphism of objects,
$$\OqMatMN = \OqMatMN_{(1)}\otimes \cdots \otimes \OqMatMN_{(N)}.$$
\end{proof}

\begin{prop}\label{prop:Dqgoodfiltrations}
The algebras $\DqMatMN$, $\Dqloop^\circ$, $\DqMatQd$ admit good filtrations.
\end{prop}
\begin{proof}
The PBW theorem gives an isomorphism $\DqMatMN\cong \OqMatMN \otimes \OqMatNM$.  By the preceding proposition, therefore we need only to construct a good filtration the $q$-symmetric algebras.  We note that each summand $\Sym_q^k(V)$ identifies with the coinvariants for the finite Hecke algebra action on $V^{\ot k}$, hence it is dual to the invariants, $\Delta(k\omega_{N-1})$ of the finite Hecke algebra action on $V^{*\ot k}$, and hence its isomorphic precisely to a dual-Weyl module $\nabla(k\omega_1)$.  The claim for $\Dqloop^\circ$ follows similarly from the tensor decomposition $\Dqloop^\circ \cong \Oq(\GL_N)\otimes \Oq(\GL_N)$, and then application of Theorem \ref{thm:REAgoodfiltration}.  Finally, the good filtration on each edge algebra tensors to give a good filtration on $\DqMatQd$, by applying Proposition \ref{prop:goodfiltrations}.
\end{proof}

\begin{lemma}\label{prop:frobpush} Suppose $q$ is a primitive $\ell$-th root of unity, with $\ell >1$ odd. Then the Frobenius pushforward of the representation $\Sym_q(\C^N)$ of $\UqL(\gl_N)$ is naturally identified with the symmetric algebra of the defining representation $\C^N$ of $\gl_N$. That is,
$$\Fr_* (\Sym_q(\C^N)) = \Sym(\C^N)$$
\end{lemma}

\begin{proof} The Frobenius pushforward functor is the functor of taking small quantum group invariants. We show that the small quantum group invariants of  $\Sym_q(\C^N)$ can be identified with the subspace $S_\ell$ consisting of polynomials in the $\ell$-th powers $x_1^\ell, \dots, x_N^\ell$. This subspace can be naturally identified with $\Sym(\C^N)$, and the induced action of the classical enveloping algebra on $S_\ell$ matches the usual one on $\Sym(\C^N)$. 
	
By definition we have the following formulas for the divided power generators of $\UqL(\g)$ on the generators $x_n^{s_n}$ of $\Sym_q(\C^N)$:
$$ K_j x_n^{s_n} = \begin{cases}
q^{s_n} x_n^{s_n}  \quad \text{\rm if $n=j$} \\
q^{-s_n} x_n^{s_n}  \quad \text{\rm if $n=j-1$} \\
x_n^{s_n}  \quad \text{\rm otherwise} \\
\end{cases}$$
	\begin{equation*}\label{eqn:divpoweraction} E_i^{(r)} x_n^{s_n}  = \begin{cases}
	\begin{bmatrix} s_n \\ r
	\end{bmatrix}_q x_{n-1}^r x_n^{s_n -r} \quad \text{\rm if $i = n-1$} \\
	0 \quad \text{\rm otherwise}
	\end{cases}, \quad 
	F_i^{(r)} x_n^{s_n}  = \begin{cases}
	\begin{bmatrix} s_n \\ r
	\end{bmatrix}_q x_n^{s_n -r} x_{n+1}^r  \quad \text{\rm if $i = n+1$} \\
	0 \quad \text{\rm otherwise}
	\end{cases} \end{equation*}
Here, $n, j \in \{1, \dots N\}$, $s_n \in \Z_{\geq 0}$, and $i \in \{1, \dots, N-1\}.$  We note the identity,

$$\begin{bmatrix} r \\ s \end{bmatrix}_q = \begin{bmatrix} r_0 \\ s_0 \end{bmatrix}_q {r_1 \choose s_1}$$
where $r = r_0 + r_1 \ell$ and $s= s_0 + s_1\ell$ for $0 \leq r_0, s_0 \leq \ell-1$. Thus, a polynomial is invariant for $K_j$ if and only if each of its monomial summands is invariant for $K_j$. From the identities above for the action of $K_j$, we see that the condition that all $K_j$ act on a monomial $\prod_{n=1}^N x_n^{s_n}$ by the scalar 1 is equivalent to the condition that each $s_n$ is divisible by $\ell$.  Thus, the invariants are contained in $S_\ell$. Similarly, the formulas  for the action of the divided powers show that $S_\ell$ is indeed invariant for the $E_i$ and $F_i$. As is well-known, the classical enveloping algebra $U(\glN)$ has a presentation in terms of Serre generators $\bar E_i$, $\bar F_i$, and $H_j$ for $i = 1, \dots, N-1$ and $j = 1, \dots, N$. We have that $\Fr(E_i^{(r)}) = \bar E_i$, $\Fr(F_i^{(r)}) = \bar F_i$. Thus, the induced action of $\gl_N$ on $S_\ell$ is the same as the usual action of $\gl_N$ on $\Sym(\C^N)$. \end{proof}

We are now ready to state the main result of this section:

\begin{theorem}\label{thm:quiver-center}
Let $q$ be a primitive $\ell$-th root of unity, where $\ell >1$ is odd.  Then:
\begin{enumerate}
\item We have a central, $\UqL(\gl_N\times\gl_M)$-equivariant embedding:
$$\Fr^*\left(\O(\Mat(e)\times \Mat(e^\vee))\right) \hookrightarrow \DqMatMN,$$
\item We have a central, $\UqL(\gl_N)$-equivariant embedding:
$$\Fr^*\left(\O(GL_N\times GL_N)\right) \hookrightarrow \Dqloop^\circ,$$
\item We have a central, $\UqL(\gl_\mathbf{d})$-equivariant embedding:
$$\Fr^*\left(\O(\MQdf)\right) \hookrightarrow \DqMatQd,$$
\item We have a Frobenius Poisson order structure on the pair $(\DqMatMN,\MQdf)$.
\end{enumerate}
\end{theorem}

\begin{proof}
For the central subalgebra asserted in (1), we let $Z\subset \DqMatMN$ be the subalgebra generated by the $\ell$th powers $(x^i_j)^\ell$, $(\partial^k_l)^\ell$ of all generators.  Inspecting the formulas \eqref{lem:coeffs}, we note that commutators with all $\ell$th powers of generators vanish when $q^\ell=1$, i.e. the subalgebra $Z$ is central.  Clearly it is isomorphic as an algebra to $\O(\Mat(e)\times \Mat(e^\vee)$.

Hence to complete the proof of (1), it remains to prove equivariance for the action of the restricted quantum group. The algebra $\DqMatMN$ is isomorphic as a $\UqL(\glN)$-module (that is, upon forgetting the $\UqL(\gl_M)$-action) to a tensor product of $2M$ copies of the $q$-symmetric algebra $\Sym_q(\C^N)$, where each sub-algebra is determined by fixing the lower indices.  Hence the equivariance of the inclusion follows immediately from Lemma \ref{prop:frobpush}.

Finally, to conclude the proof of (1), it follows from Lemma \ref{prop:frobpush} that the small quantum group invariants contain the $\ell$-th powers of the generators and the induced action of the classical enveloping algebra $U(\gl_N)$ on the copy of $\O(\Mat(N,M) \times \Mat(M,N))$ they generate coincides with the action induced from $\GL_N$ acting on the space $\Mat(N,M)$ by multiplication. Similar observation hold for the action of $\UqL(\gl_M)$ (here $q$ is replaced by $q^{-1}$, because the defining representation is replaced by its dual), and we conclude that the embeddings in the statement of the theorem are equivariant for the restricted quantum group. 

The claim (2) has already been proved in Section 4, because the algebra $\Dqloop^\circ$ is isomorphic (as an algebra object of $\Repq(G)$) to the framed quantum character variety of the punctured torus, $A_{T^2\backslash D^2}$ for $\GL_N$.

The claim (3) now follows from the fact that $\DqMatQd$ is a braided tensor product of its edge-algebras; this implies easily that the central subalgebras of each edge algebra is $\UqL(\gl_{\dv})$-equivariant, and $\uq(\gl_{\dv})$-invariant.  This further implies that the central subalgebra on each edge is in fact central in the whole algebra, as the braided tensor product commutativity relations become trivial on $\uq(\gl_\dv)$-invariant subalgebras.

Finally, to prove claim (4), we appeal to Proposition \ref{prop:FrobeniusPoissonOrderDegeneration}.  Combining Proposition \ref{prop:Dqgoodfiltrations} and Proposition \ref{prop:goodfiltrations} establishes the required surjectivity, so we need only to show that the resulting derivation $D$ on $\DqMatQd$ preserves $Z$.  For this, we can compute directly with the relations of Lemma \ref{lem:coeffs} to see that, on each edge algebra $\DqMatMN$ we have that the derivation of Proposition \ref{prop:PoissonOrderDegeneration} restricts on $Z$ to give the Poisson bivector,
\begin{align}
\pi  &= \sum_{i >j ; n,m} \left( \delta_{mn} y_n^j y_m^i + 2\theta(n-m) y_m^j y_n^i \right) \frac{\partial}{\partial y_m^i} \wedge \frac{\partial}{\partial y_n^j}  -   \sum_{m > n }  y_n^i y_m^i  \frac{\partial}{\partial y_m^i} \wedge \frac{\partial}{\partial y_n^i} \label{eq:edgePoissonbivector}
\\
&+   \sum_{i >j ; n,m} \left( \delta_{mn} z_n^j z_m^i + 2\theta(n-m) z_m^j z_n^i \right) \frac{\partial}{\partial z_m^i} \wedge \frac{\partial}{\partial z_n^j}     -    \sum_{m > n }  z_n^i z_m^i  \frac{\partial}{\partial z_m^i} \wedge \frac{\partial}{\partial z_n^i} \nonumber \\
&+   2\sum_{i,j,m,n } \left(  (\delta_{in} + \delta_{jm})  y_n^j z_m^i + \delta_{in} \sum_{p >i} y^j_p z^p_m + \delta_{jm} \sum_{p < j} y_n^p z_p^i + \frac{\delta_{in} \delta_{jm}}{(q^2 -1)^\ell}    \right) \frac{\partial}{\partial z_m^i} \wedge \frac{\partial}{\partial y_n^j}. \nonumber
\end{align} 

Meanwhile, we know that the derivation preserves the central subalgebra on loop edges as is proved in Section \ref{subsec:CharVarFrobPoissonOrder}.  And now once again the claim for $\DqMatQd$ follows by considering braided tensor products. \end{proof}

\begin{rmk}
The direct check that $D(Z)(Z)\subset Z$ given here can be avoided:  Theorem \ref{thm:Azumayazero} gives the existence of a single Azumaya point, and $\cZ=\Spec(Z)$ is smooth and connected.  Hence Lemma \ref{lm:wholecenter} implies that $Z$ is the center of $\DqMatQd$, and then the desired containment follows by Lemma \ref{lem:Hayashi}.
\end{rmk}

\subsection{Nondegeneracy of the Poisson $G$-variety $\MQdf$}
\label{subsec:qmqvnondeg}
We are now going to analyze the nondegeneracy of the Poisson $\GL_{\dv}$-variety structure on $\MQdf$ given by Theorem \ref{thm:quiver-center}.

Consider the quiver $(Q, \dv) = \left( \overset{N}{\bullet}\rightarrow\overset{M}{\bullet}\right)$. From the formula \eqref{eq:edgePoissonbivector} it is clear that the Poisson structure is independent of $\ell$ up to scale. So, it is enough to analyze the nondegeneracy of the Poisson structure for $q\rightarrow 1$. Let $r_M$ and $r_N$ be the classical $r$-matrices on $\gl_M$ and $\gl_N$ respectively. By our convention the symmetric parts of $r_M$ and $r_N$ are given by $P/2$, where $P\colon \C^M\otimes \C^M\rightarrow \C^M\otimes \C^M$ is the flip $v\otimes w\mapsto w\otimes v$ and similarly for $\C^N$.

Then the $q\rightarrow 1$ limit of the relations in Definition \ref{def:dqmatmn} gives the following Poisson structure on $\MQdf$:
\begin{align*}
\{X_1, X_2\} &=  r_N X_1X_2 - X_1X_2 (r_M)_{21} \\
\{D_1, D_2\} &= r_M D_1D_2 - D_1D_2 (r_N)_{21} \\
\{X_1, D_2\} &= -D_2 r_N X_1 - X_1 r_M D_2 - \Omega.
\end{align*}

Here the relations take place in $\cO(\MQdf)$ tensored with $\Hom\left(\C^M\otimes \C^M, \C^N\otimes \C^N\right)$ in the first line, $\Hom\left(\C^N\otimes \C^N, \C^M\otimes \C^M\right)$ in the second line, and $\Hom\left(\C^M\otimes \C^N, \C^M\otimes \C^N\right)$ in the last line, respectively, and $\Omega\colon \C^M\otimes \C^N\rightarrow \C^M\otimes \C^N$ simply denotes the identity map.

\begin{theorem}\label{thm:qmqvnondeg}
Let $Q$ be an arbitrary quiver. The Poisson $\GL_{\dv}$-variety $\MQdf^\circ$ is nondegenerate.
\label{thm:quivervarietynondegenerate}
\end{theorem}
\begin{proof}
By Definition \ref{def:dqQ} and Proposition \ref{prop:fusiondegeneration}, $\MQdf^\circ$ is obtained by the fusion of the corresponding varieties for a single edge. But by Proposition \ref{prop:Poissonfusion} fusion preserves nondegeneracy, so it is enough to prove nondegeneracy of $\MQdf^\circ$ when $Q$ is a single edge. When $Q$ is a single loop, the result follows from Proposition \ref{prop:fockroslynondegenerate}. So, we just have to analyze the case $(Q, \dv) = \left( \overset{N}{\bullet}\rightarrow\overset{M}{\bullet}\right)$.

Let $a\colon \gl_N\oplus \gl_M\rightarrow \Gamma(\MQdf, \T_{\MQdf})$ be the infinitesimal action map and consider the new bivector $\widetilde{\pi} = \pi + a(r_M + r_N)\in\Gamma\left(\MQdf, \T_{\MQdf}^{\otimes 2}\right)$ (note that it is no longer antisymmetric). The corresponding biderivation $\{-, -\}'$ is given by the following formulas:
\begin{align*}
\{X_1, X_2\}' &= -X_1X_2 P \\
\{D_1, D_2\}' &= -D_1 D_2 P \\
\{X_1, D_2\}' &= -\Omega \\
\{D_2, X_1\}' &= D_2 P X_1 + X_1 P D_2 + \Omega.
\end{align*}

By Proposition \ref{prop:LiBlandSevera} the nondegeneracy of the Poisson $G$-variety $\MQdf$ is equivalent to the condition that this matrix that we will denote $M$ is invertible. Writing it out in coordinates, we have
\begin{align*}
\{y_i^j, y_k^l\}' &= -y_k^j y_i^l = (M_{XX})_{ik}^{jl} \\
\{z_i^j, z_k^l\}' &= -z_k^j z_i^l = (M_{DD})_{ik}^{jl} \\
\{y_i^j, z_k^l\}' &= -\delta_i^l \delta_k^j = (M_{XD})_{ik}^{jl} \\
\{z_k^l, y_i^j\}' &= \sum_m z_m^l y_i^m \delta^j_k + \sum_m z^m_k y_m^j \delta^l_i + \delta_k^j \delta_i^l = (M_{DX})_{ki}^{lj}.
\end{align*}

We have a block form
\[
M = \left(\begin{array}{cc} M_{XX} & M_{XD} \\ M_{DX} & M_{DD} \end{array}\right)
\]
Since $M_{XD}$ is a scalar matrix, we have $\det(M) = \det(M_{DD}M_{XX} - M_{XD}M_{DX})$. We get
\begin{align*}
(M_{DD}M_{XX})_{ik}^{jl} &= (DX)^j_k (XD)^l_i \\
(M_{XD}M_{DX})_{ik}^{jl} &= -(DX)^j_k \delta^l_i - (XD)^l_i \delta^j_k - \delta_i^l \delta_k^j
\end{align*}
and hence
\[(M_{DD}M_{XX}-M_{XD}M_{DX})_{ik}^{jl} = ((DX)^j_k+\delta^j_k)((XD)^l_i + \delta^l_i).\]
In other words, $M_{DD}M_{XX}-M_{XD}M_{DX}\colon \C^M\otimes \C^N\rightarrow \C^M\otimes \C^N$ is the tensor product of matrices $(1+DX)\colon \C^M\rightarrow \C^M$ and $(1+XD)\colon \C^N\rightarrow \C^N$. We have $\det(1+DX) = \det(1+XD)$, so
\[\det(M) = \det(1+DX)^{N+M}.\]
In particular, it is invertible on the locus $\MQdf^\circ\subset \MQdf$.
\end{proof}

\begin{rmk}\label{rmk-VdB}
Van den Bergh has previously defined a natural quasi-Poisson structure on $\MQdf$ using the theory of double quasi-Poisson structures \cite{VandenBerghDoublePoissonalgebras2008}. In particular, he shows that his quasi-Poisson structure is nondegenerate on $\MQdf^\circ\subset \MQdf$, see \cite[Proposition 8.3.1]{VandenBerghQuasiHamiltonian}. We expect that the twist of his quasi-Poisson structure with respect to the antisymmetric part of the $r$-matrix coincides with the Poisson structure studied in this paper which is obtained by degenerating the quantization given in \cite{JordanQuantizedmultiplicativequiver2014}.\end{rmk}

\subsection{An Azumaya point on $\MQdf$}\label{subsec:qmqvazumaya}

Let $Q\Kr$ be the Kronecker quiver with dimension vector  $\mathbf{d} = (M, N)$. Recall from Definition \ref{def:dqmatmn} the algebra $\DqMatMN$, which has generators $x^i_j$ and $\partial_p^r$ for $1 \leq i, p \leq M$ and $1 \leq j,r \leq N$. Recall also that we organize these generators into an $M \times N$ matrix $X$ and an $N \times M$ matrix $D$. Similarly, set $X^{[\ell]}$ to be the $M \times N$ matrix with entries $(x^i_n)^\ell$ and  $D^{[\ell]}$ to be the $N \times M$ matrix with entries $(\de_i^n)^\ell$. We fix $q$ to be a primitive $\ell$-th root of unity, where $\ell >1$ is odd. By Theorem \ref{thm:quiver-center}, the entries in $X^{[\ell]}$ and $D^{[\ell]}$ generate a central subalgebra $Z_\ell$ isomorphic to the coordinate algebra of $\Mat(M,N)\times \Mat(N,M)$. The quotients
$$A =  \DqMatMN/ ( D^{[\ell]} =0) \qquad \text{\rm and} \qquad \mathcal M =  \DqMatMN/(D =0)$$
of $ \DqMatMN$ by the ideal generated by the entries of $D^{[\ell]}$ and the ideal generated by the $\partial$'s are each modules for the quotient $Z_\ell/ (D^{[\ell]} =0)$. The former is in fact an algebra over  $Z_\ell/ (D^{[\ell]} =0)$, and by the following theorem, is the endomorphism algebra of the latter:

\begin{theorem}\label{thm:Azumayazero}
We have:
\begin{enumerate}
\item The left multiplication action induces isomorphisms,
$$ \DqMatMN/ \left( D^{[\ell]} =0\right) \stackrel{\sim}{\longrightarrow} \End_{Z_\ell/\left(D^{[\ell]} =0\right) }\left(\DqMatMN/(D =0) \right),$$
Hence, the fiber $\DqMatMN/ \left( D^{[\ell]} =0, X^{[\ell]} =0\right)$ of the sheaf defined by $\DqMatMN$ over zero in $\Mat(M,N) \times \Mat(N,M)$ is a matrix algebra.
\item The fiber of $\MQdf$ over the zero representation is a matrix algebra.
\end{enumerate}
\end{theorem}

\newcommand{\LI}{\text{\rm LI}}

\begin{proof} For (1) we observe that the module $\mathcal M$ is of rank $\ell^{MN}$ over $\tilde Z := Z_\ell/ (D^{[\ell]} =0)$ with basis given by the (ordered) monomials
$$\bar x(\bar r) := (x^1_N)^{r^1_N} (x^1_{N-1})^{r_{N-1}^1} \cdots (x^M_1)^{r^M_1}$$ 
where $\bar r = (r^1_N, r^1_{N-1}, \dots r^1_1, r^2_N, \dots, r^M_1)$ ranges over the elements in $\{0, \dots, \ell-1\} \times \{ 0, \dots, \ell-1\}$.  Clearly the element $1$ is a cyclic generator for the module; hence our strategy is to show that we can reach $1$ from an arbitrary element of $\cM$ using the $A$-action, and hence conclude that $\cM$ is an irreducible representation of $A$ of the correct dimension.

Note that $\bar r$ is constantly zero if and only if $\bar x(\bar r) = 1$. An arbitrary element of $\mathcal M$ can be expressed as $$\bar y  = \sum_{\bar r} z_{\bar r} \bar x (\bar r) $$ for some $z_{\bar r} \in \tilde Z$.  Define a total order on the set $S = ( [1,M] \times [1,N]) \coprod \{0\}$ by 
$$ (j,m) \succ 0, \qquad \qquad (i,n) \succ (j,m) \ \text{\rm if} \ \begin{cases}
    n > m\\
   \text{or} \ n=m  \ \text{and} \ i < j
  \end{cases}  $$ for any $i,j \in [1,N]$ and any $m,n \in [1,M]$.  The leading index of the  monomial $\bar x(\bar r)$ is defined as $$\text{\rm LI}( \bar x(\bar r)) = \max  \{ (i,n) \ | \ r_n^i \neq 0\} \in S$$ if $x \neq 1$, and $\text{\rm LI}(1) = 0 \in S$. The leading index of an element  $\bar y  = \sum_{\bar r} z_{\bar r} \bar x (\bar r) $ is defined as $$\text{\rm LI}(\bar y) = \max \{ \text{\rm LI}(\bar x (\bar r)) \ | \ z_{\bar r} \neq 0\} \in S$$
if $\bar y \neq 0$, and $\text{\rm LI}(0) = 0$.

\begin{lemma}\label{lem:expression} For any $p$, $1 \leq p \leq M$, the element 
$$ \de_p^i x^p_n - \delta_{in} q^{2p+1} \in \DqMatMN$$
can be expressed as a sum of elements, each divisible on the right by some $\de_{p'}^s$, where $(p',s) \succeq (p,i)$.  \end{lemma}

\begin{proof} By the defining relations of the algebra, we have 
$$ \de_p^i x^p_n  - \delta_{in} q  = q^{1+ \delta_{in}} x^p_n \de_p^i + \delta_{in} (q^2 -1) \sum_{s >i} x^p_s \de_p^s + (q^2 -1) \sum_{p' < p} \de_{p'}^i x_n^{p'}$$
We proceed by induction on $p$. From the above expression, the base case $p=1$ is clear. If $p>1$, then by the inductive hypothesis, all elements on the RHS are of the desired form except for $(q^2-1) \delta_{in} \sum_{p'<p} q^{2p'+1}$. Adding this to the LHS, we obtain $\delta_{in} q^{2p+1}$. \end{proof}

\begin{lemma}\label{lem:LI} Suppose $y$ is the form $y = (x_n^j)^r x$, where $\LI( x) \prec (j,n)$ and $1 \leq r \leq \ell-1$. Then
 $$\LI\left( \de_m^i \triangleright \left( y\right) \right) =0 \ \text{ for any $(m,i)$ with $(m,i) \succ (j,n)$.}  $$\end{lemma}

\begin{proof} The condition  $(m,i) \succ (j,n)$ implies that we cannot have both $i=n$ and $j=m$, i.e.\ $\delta_{in} \delta_{jm} = 0$. Using Lemma \ref{lem:coeffs}, we compute
\begin{align*}  \partial_m^i \triangleright y = q^{(\delta_{in} +  \delta_{jm})r} &(x_n^j)^r (\partial_m^i \triangleright  x) + a_r \sum_{p >i} (x_n^j)^{r-1} [ x_p^m \partial_m^p  \triangleright  x]  + b_r \sum_{p' <j} (x_n^j)^{r-1} [\partial_{p'}^i x_n^{p'}\triangleright  x], \end{align*} noting that $c_r$ and $d_r$ vanish since they are divisible by $\delta_{in} \delta_{jm}$. 

We now prove the claim by induction on $(j,n)$ in the totally ordered set $S$. The base case is $(j,n) = (1,N)$. In this case, $y = (x^1_N)^r$ for some $1 \leq r \leq \ell -1$, and there are no $p$ or $p'$ with $p >N$ and $p' <1$. Thus, $$\partial_m^i \triangleright y = q^{(\delta_{in} +  \delta_{jm})r} (x_N^1)^r (\partial_m^i \triangleright 1) = 0.$$

For the induction step, let $y = (x_n^j)^r x$, where $\LI( x) \prec (j,n)$ and $1 \leq r \leq \ell-1$. Now, 
\begin{align*}  \partial_m^i \triangleright y  = q^{(\delta_{in} +  \delta_{jm})r} &(x_n^j)^r (\partial_m^i \triangleright  x) + a_r \sum_{p >i} (x_n^j)^{r-1} [ x_p^m \partial_m^p  \triangleright  x]  + b_r \sum_{p' <j} (x_n^j)^{r-1} [\partial_{p'}^i x_n^{p'}\triangleright  x]. \end{align*} The inductive hypothesis implies that $\partial_m^i \triangleright  x = 0$ and $\partial_m^p  \triangleright  x = 0$ since $(m,p) \succ (m,i)$. Suppose $m=j$, so that $b_r$ is nonzero. Then necessarily $i \neq n$.  By Lemma \ref{lem:expression}, we have that $\partial_{p'}^i x_n^{p'}$   can be expressed as a sum of elements, each divisible on the right by some $\de_{p''}^s$, where $(p'',s) \succeq (p',i) \succ (m,i) $. Thus, this element acts on $x$ as zero. The inductive step follows. 
\end{proof}

\begin{lemma} Suppose $ y = (x_n^j)^r  x$, where $\LI( x) \prec (j,n)$ and $1 \leq r \leq \ell-1$. Then $\LI\left( \de_j^n \triangleright  y \right) \preceq (j,n)$ and is nonzero. \end{lemma}

\begin{proof} We record the following computation:
\begin{align*}
\partial_j^n \triangleright (x_n^j)^r  x = q^{2r} &(x_n^j)^r (\partial_j^n \triangleright  x) + a_r \sum_{p >i} (x_n^j)^{r-1} [ x_p^j \partial_j^p  \triangleright  x]  + b_r \sum_{p' <j} (x_n^j)^{r-1} [\partial_{p'}^n x_n^{p'}\triangleright  x]  + \\  & + c_r \sum_{p>n, p'<j} (x_n^j)^{r-2} [ x_p^j \partial^p_{p'} x^{p'}_n \triangleright  x]  + d_r  (x_n^j)^{r-1} x.\end{align*}
We now proceed by induction on $(j,n)$ in the totally ordered set $S$.  The base case is $(j,n) = (1,N)$. In this case, $y = (x^1_N)^r$ for some $1 \leq r \leq \ell -1$, and there are no $p$ or $p'$ with $p >N$ and $p' <1$. Thus, $$\partial_1^N \triangleright y = q^{2r} (x_N^1)^r (\partial_1^N \triangleright 1) + d_r  (x_N^1)^{r-1}  = \frac{q^{2r} - 1}{q - q\inv}  (x_N^1)^{r-1},$$ which is nonzero.

For the induction step, let $y = (x_n^j)^r x$, where $\LI( x) \prec (j,n)$ and $1 \leq r \leq \ell-1$.  Lemma \ref{lem:LI} implies that $\partial_j^n \triangleright  x =0$. Now fix $p>i$ and $p' <j$. Then Lemma \ref{lem:LI} also implies that $x_p^j \partial_j^p  \triangleright  x =0$. By Lemma \ref{lem:expression}, for $p' <j$, we have that $\partial_{p'}^n x_n^{p'} -  q^{2p' +1} $   can be expressed as a sum of elements, each divisible on the right by some $\de_{p''}^s$, where $(p'',s) \succeq (p',n) \succ (j,n) $. Thus, this element acts on $x$ as $q^{2p'-1}$. Finally, $\LI(x^{p'}_n  x) \prec (p',p)$, so $\partial^p_{p'} \triangleright ( x^{p'}_n   x ) = 0$. What remains nonzero in the expression at the beginning of this proof is:
$$ \partial_m^i \triangleright (x_n^j)^r  x = \left((q^{2r} -1) \sum_{p' < j} q^{2p'+1} + \frac{q^{2r} -1}{q- q\inv} \right) (x_n^j)^{r-1}  x =  \frac{q^{2j}(q^{2r} -1)}{q- q\inv} (x_n^j)^{r-1}  x, $$ which is nonzero and of leading index at most $(j,n)$.  \end{proof}

\begin{lemma} Suppose $\LI( y) = (j,n)$. Then, for some $r$, $1 \leq r \leq \ell-1$, we have that
$$\LI\left( (\de_j^n)^r \triangleright  y \right) \prec (j,n), \ \text{\rm and is nonzero}.$$ \end{lemma}

\begin{proof} We can write 
$$ y  = (x^i_n)^r  x^{(r)} +  (x^i_n)^{(r-1)}  x^{(r-1)} + \cdots +  x^i_n  x^{(1)} +  x^{(0)}$$
for some $ x^{(\bullet)} \in \mathcal M$ with $\LI( x^{(\bullet)}) \prec (j,n)$ and $1 \leq r \leq \ell-1$. Then, iterating the computation in the proof of the previous lemma, we obtain $ (\de_j^n)^r \triangleright  y  = \left( \frac{q^{2j}}{q - q\inv} \right)^r \prod_{s = 1}^r (q^{2r} -1)  x^{(r)},$ which is nonzero. \end{proof}

\begin{lemma} For any $ y \in \mathcal M$ nonzero, there is an element $d $  of $\DqMatMN$ such that $d \triangleright  y$ is a nonzero multiple of $1$.  $\LI( y) = (j,n)$. \end{lemma}

\begin{proof} Proceed by induction on $\LI(\bar y) \in S$. If $\LI(\bar y) =0$, there is nothing to show. Otherwise, let $(j,n) = \LI(\bar y)$. By the previous corollary we can lower the leading index using an element of $\DqMatMN$. \end{proof}

To conclude the proof on (1): we have shown that under the action of $\DqMatMN/\left(D^{[\ell]} = 0 \right)$ on $\mathcal M$ any element lies in the orbit of $1$, and conversely that $1$ lies in the orbit of any nonzero element of $\mathcal M$.  Hence the representation is irreducible.  Since it is  a representation whose rank over $\tilde Z$ is equal to the square root of the rank of $\DqMatMN/\left(D^{[\ell]} = 0 \right)$ as a module over $\tilde Z$, we conclude that the action map  $$ \DqMatMN/ \left( D^{[\ell]} =0\right) {\longrightarrow} \End_{\tilde Z }\left(\mathcal M \right)$$ is an isomorphism.

For (2), we once again appeal to the construction of $\MQdf$ via braided tensor products and the result follows from Proposition \ref{prop:tensorend}.
\end{proof}

\subsection{The Frobenius quantum moment map}
\label{subsec:qmqvmom-map}
The quantum moment map 
$$\mu_q : \Oq(\GL_\mathbf{d}) \rightarrow \DqQ^\circ,$$
from Definition \ref{def:dqQ} is valid for $q$ a root of unity when working in $\Repq(\GL_{\dv})$. We also have the group-valued moment map:
$$\tilde \mu : \MQdf^\circ \rightarrow \GL_{\mathbf{d}}$$
defined on a Zariski open subset  $\MQdf^\circ$ of $\MQdf$ (see Section \ref{subsec:mqv}).  These two moment maps fit into a Frobenius moment map structure, as explained in the following key theorem.

\begin{theorem}\label{thm:mmgeneral} We have a central, $\UqL(\gl_\mathbf{d})$-equivariant embedding 
\[\O(\MQdf^\circ) \hookrightarrow \DqQ^\circ\]
which fits into a commutative diagram 
$$\xymatrix{  \O(\MQdf^\circ) \ar[r] & \DqQ ^\circ \\ \O(\GL_\mathbf{d}) \ar[u]^{ \tilde \mu^\#} \ar[r]^{i_\mathbf{d} } & \Oq(\GL_\mathbf{d}) \ar[u]^{\mu_q} },$$ 
where $\tilde \mu^\#$ is pullback along the group-valued moment map $\tilde \mu$. Thus, $(\DqQ^\circ,\O(\MQdf^\circ))$ is a $G$-Hamiltonian Frobenius Poisson order.
\end{theorem}

\begin{proof}
We may apply Proposition \ref{prop:StrongHamiltonianOrder} to conclude that there exists some classical multiplicative moment map $\widetilde{\mu}^\#$ making the diagram commute.  By Proposition \ref{prop:momentmapunique}, the set of possible multiplicative moment maps associated to our fixed $\GL_\dv$-action is a torsor over the center of $\GL_\dv$, so we may assume without loss of generality that it recovers the classical moment map from Equation \eqref{moment-map-formula} exactly. 
\end{proof}

\subsection{The quantized multiplicative quiver variety at a root of unity}\label{subsec:QMQVrootofunity}
First, we recall from the definition of the quantized multiplicative quiver variety with trivial GIT character:

\begin{definition} \cite[Definition 8.1]{JordanQuantizedmultiplicativequiver2014}  
Fix an algebra homomorphism $\xi_q: \Oq(G)\to \C$ in $\Repq(G)$, and let $\mathcal I_{\xi_q}$ denote its kernel. The quantized multiplicative quiver variety corresponding to the above data is defined as the quantum Hamiltonian reduction of $D_q(\Mat(Q, \dv))$ at the ideal $\mathcal I_{\xi_q}$:
$$\mathcal A^\xi ( Q, \dv) := \Hom_{\RepqGLd} \left(  \C, \DqMatQd/I_{\xi_q}   \right).$$
\end{definition}

We now generalize this to non-trivial GIT characters, to the construct a sheaf of algebras on the multiplicative quiver variety, via a version of version of quantum Hamiltonian reduction for non-affine GIT quotients.

First suppose $(A, Z)$ is a $G$-Hamiltonian Frobenius Poisson order. Denote as before by $I\subset Z$ the classical moment map ideal and by $I_q\subset A$ the quantum moment map ideal. Recall that $(A/I_q)^{\uq(\g)}$ is an algebra with a central subalgebra $Z/I$. Therefore,
\[\left(\bigoplus_{m=0}^\infty (A/I_q)^{\uq(\g)} \otimes \C_{\theta^{-m}}\right)^G\]
is a graded algebra with a central subalgebra
\[\left(\bigoplus_{m=0}^\infty Z/I\otimes \C_{\theta^{-m}}\right)^G.\]
\begin{definition} We denote by $\cA^\xi_\theta(Q,\dv)$ the sheaf of algebras over $\cZ\dS_{\!\!\theta} G$ so constructed
\end{definition}

We will say a sheaf of algebras $\cA$ over a scheme $X$ is a Poisson order if $X$ is a Poisson scheme and on each open subset of $X$ we obtain the structure of a Poisson order. Then we get the following variant of Proposition \ref{prop:FrobeniusPoissonReduction}

\begin{prop}
Suppose $(A_q, Z)$ is a $G$-Hamiltonian Frobenius Poisson order where $Z$ is Noetherian. Then $(\cA^\xi_\theta, \cZ\dS_{\!\!\theta} G)$ is a Poisson order.
\label{prop:GITFrobeniusPoissonReduction}
\end{prop}

\subsection{Proof of Theorem \ref{quiver-thm-intro}}\label{subsec:sheaf}

We are now ready to prove the main result of this section, that the sheaf of algebras constructed in the previous section is Azumaya over the entire smooth locus of $\MQd$.

\begin{theorem}[Theorem \ref{quiver-thm-intro}] Let $\ell>1$ be an odd integer, and  $q$  a primitive $\ell$-th root of unity. Then:
\begin{enumerate}
\item The algebra $\D_q(\MatQd)$ is finitely generated over its center, which is isomorphic to the coordinate ring $\O(\MQdf)$ of the classical framed multiplicative quiver variety.
\item Moreover, $\D_q(\MatQd)$ is Azumaya over the preimage in $\MQdf$ of the big cell $\Br\subset G$ under the multiplicative moment map $\MQdf\to G$. 
\item Frobenius quantum Hamiltonian reduction defines a coherent sheaf of algebras over the classical multiplicative quiver variety $\MQd$, which is Azumaya over the locus $\MQds$ of $\theta$-stable representations.
\end{enumerate}
\end{theorem}

\begin{proof}

We follow the template of Section \ref{sec:template-for-results}.  Hence, we let $A_\Rt$ be the algebra $\DRMatQd$, and $\mu_\Rt$ be the quantum moment map of Definition \ref{def:quiver-moment-map}.  For $Z$ we take the central subalgebra constructed in Theorem \ref{thm:quiver-center}.  For the Azumaya point $p$ required in Lemma \ref{lm:wholecenter}, we take the zero representation, as in Theorem \ref{thm:Azumayazero}.  The first claim of the Theorem now follows from Lemma \ref{lm:wholecenter}.

We prove the remaining two claims together by applying Theorem \ref{thm:frobtwistedAzumaya}.  For $G_{stab}$ we take the copy of $\C^\times$ embedded diagonally in $\GL_\dv$.  For $G_0$ we take the subgroup of $\GL_\dv$ on which the product of all determinants is one.  We choose for $\xi_q$ the character $(\xi_v\epsilon_v)$.  We take for the open subset $U$ the entire stable locus (which we assume is non-empty), on which $G$ acts with common stabilizer $G_{stab}$.

Then, Assumption 1 of Theorem \ref{thm:frobtwistedAzumaya} is confirmed in Theorem \ref{thm:quivervarietynondegenerate}, Assumption 2 is clear, since the entire algebra $\DqMatQd$ lies in total degree zero for the grading given by summing degrees in all vertices, hence $G_{stab}$ acts trivially on $\DqMatQd$.  Assumption 3 is clear, since $\xi_q$ lives over the identity element in $G$, which clearly lies in $\Br$.
\end{proof}


\printbibliography

\bigskip

\noindent{\small

\noindent {\sc IST Austria, Klosterneuburg, Austria}

\noindent \href{mailto:iordan.ganev@ist.ac.at}{iordan.ganev@ist.ac.at}

\medskip

\noindent {\sc School of Mathematics, University of Edinburgh, Edinburgh, UK}

\noindent \href{mailto:D.Jordan@ed.ac.uk}{D.Jordan@ed.ac.uk}

\medskip

\noindent {\sc Institut f\"{u}r Mathematik, Universit\"{a}t Z\"{u}rich, Zurich, Switzerland}

\noindent \href{mailto:pavel.safronov@math.uzh.ch}{pavel.safronov@math.uzh.ch}}

\end{document}